\documentclass[10pt,a4]{amsart}
\usepackage{framed,amssymb}
\usepackage[left=40mm,right=40mm,top=40mm,bottom=40mm]{geometry}
\usepackage{color}
\usepackage[T1]{fontenc}
\usepackage{ae,aecompl}
\usepackage{tikz}
\usepackage{comment}
\usepackage{amsthm,amsmath}
\usepackage{hyperref}
\usepackage[mode=buildnew]{standalone}
\usepackage{mathtools}
\usepackage{tensor}
\usepackage[outline]{contour}

\newcommand{\defeq}{\stackrel{{\text{def}}}{=}} 
\newcommand{\pp}[2]{\frac{\partial #1}{\partial #2}} 
\newcommand{\ppp}[2]{\frac{\partial^2 #1}{\partial #2^2}} 
\newcommand{\pppp}[3]{\frac{\partial^2 #1}{\partial #2 \partial #3}} 
\newcommand{\jump}[1]{\left[#1\right]} 
\newcommand{\Landau}{\mathcal{O}} 

\newcommand{\cin}{{c_{\textrm{in}}}} 
\newcommand{\cout}{{c_{\textrm{out}}}} 
\newcommand{\lot}{\textrm{LOT}} 
\newcommand{\ls}[1]{\overset{\scriptscriptstyle{1}}{#1}} 
\newcommand{\rs}[1]{\overset{\scriptscriptstyle{2}}{#1}} 
\newcommand{\is}[1]{\overset{\scriptscriptstyle{i}}{#1}} 

\newtheorem{proposition}{Proposition}

\newtheorem{theorem}{Theorem}

\begin{document}

\title{Shock Interaction in Plane Symmetry}
\date\today
\author[A.~Lisibach]{Andr\'e Lisibach}
\address{Andr\'e Lisibach\\Bern University of Applied Sciences}
\email{andre.lisibach@bfh.ch}



\begin{abstract}
  We consider the problem of interaction of two oncoming shocks in plane symmetry for a barotropic fluid. We establish a local in time solution after the point of interaction, thereby determining the state behind the emerged shocks which originate at the interaction point. The location of these shocks in space time being unknown the problem constitutes a double free boundary problem, i.e.~complete data for the problem is only given in a single point, the interaction point. This article can be viewed as an extension of the article \cite{lisibach2021shock} which deals with the reflection of a shock on a wall.
\end{abstract}

\maketitle

\tableofcontents

\section{Introduction}
The interaction of two oncoming shocks is a classic problem in gas dynamics. We look at the situation in plane symmetry, i.e.~the quantities describing the fluid depend only on one space variable. A typical application is gas flow in a tube. The setup consists of two shock waves traveling towards each other. The domain between the two shocks shrinks to a point when the shocks collide. After the collision, two shocks emerge which move away from each other, leaving a growing zone behind them.

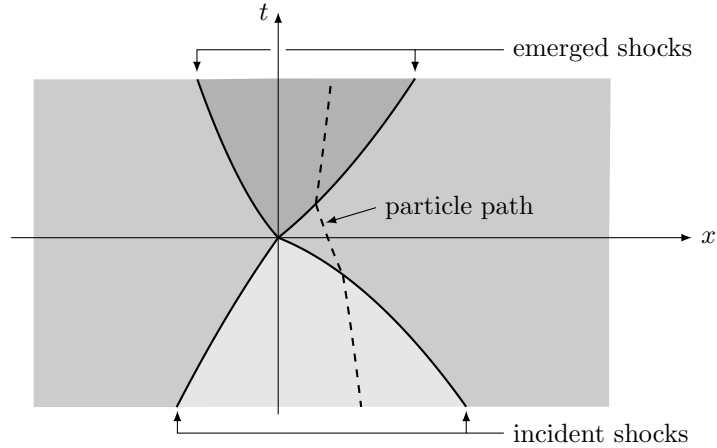
\begin{figure}[h]
  \centering
  \begin{tikzpicture}[scale=1]
  \fill [fill=gray!20, domain=0:2.5, variable=\x] (0,0) -- plot (\x,{-.4*\x-0.2*\x^2}) -- (0,-2.25) -- cycle;
  \fill [fill=gray!20, domain=0:1.35, variable=\x] (0,0) -- plot (-\x,{-1.4*\x-0.2*\x^2}) -- (0,-2.25) -- cycle;
  \fill [fill=gray!40, domain=0:2.5, variable=\x] (0,0) -- plot (\x,{-.4*\x-0.2*\x^2}) -- ++(1.9,0)-- (4.4,0) -- cycle;
  \fill [fill=gray!40, domain=0:1.35, variable=\x] (0,0) -- plot (-\x,{-1.4*\x-0.2*\x^2}) -- ++(-1.9,0)-- (-3.25,0) -- cycle;
  \fill [fill=gray!40, domain=0:1.82, variable=\x] (0,0) -- plot (\x,{.8*\x+0.2*\x^2}) -- ++(2.6,0)-- (4.4,0) -- cycle;
  \fill [fill=gray!40, domain=0:1.075, variable=\x] (0,0) -- plot (-\x,{1.1*\x+0.8*\x^2}) -- ++(-2.175,0)-- (-3.25,0) -- cycle;
  \fill [fill=gray!60, domain=0:1.82, variable=\x] (0,0) -- plot (\x,{.8*\x+0.2*\x^2}) -- (0,2.11848) -- cycle;
  \fill [fill=gray!60, domain=0:1.075, variable=\x] (0,0) -- plot (-\x,{1.1*\x+0.8*\x^2}) -- (0,2.11848) -- cycle;

  \draw[thick,domain=0:1.82,smooth,variable=\x] plot (\x,{.8*\x+0.2*\x^2});
  \draw[thick,domain=0:2.5,smooth,variable=\x] plot (\x,{-.4*\x-0.2*\x^2});

  \draw[thick,domain=0:1.075,smooth,variable=\x] plot (-\x,{1.1*\x+0.8*\x^2});
  \draw[thick,domain=0:1.35,smooth,variable=\x] plot (-\x,{-1.4*\x-0.2*\x^2});
  
  \draw[-latex] (-3.55,0) -- (5.5,0) node[anchor=west] {$x$};
  \draw[-latex] (0,-2.35) -- (0,3) node[anchor=east] {$t$};


  \draw[latex-](-1.08,2.2)--(-1.08,2.5)--(-0.1,2.5);
  \draw(.1,2.5)--(1.82,2.5);
  \draw[latex-](1.82,2.2)--(1.82,2.5)--(3,2.5)node[anchor=west]{emerged shocks};

  \draw[latex-](-1.34,-2.3)--(-1.34,-2.6)--(0.1,-2.6);
  \draw(.1,-2.6)--(2.5,-2.6);
  \draw[latex-](2.5,-2.3)--(2.5,-2.6)--(3,-2.6)node[anchor=west]{incident shocks};

   \begin{scope}[xshift=0.5cm,yshift=0.45cm]
     \draw[thick,dashed,domain=0:.208,smooth,variable=\x] plot (\x,{7*\x+5*\x^2});
     \draw[thick,dashed,domain=0:.355,smooth,variable=\x] plot (\x,{-3*\x+1*\x^2});    
   \end{scope}
   \begin{scope}[xshift=0.86cm,yshift=-0.5cm]
     \draw[thick,dashed,domain=0:.243,smooth,variable=\x] plot (\x,{-6*\x-5*\x^2});    
   \end{scope}
   \draw[latex-](.63,.2)--(1.3,.4)node[anchor=west] {particle path};
  
\end{tikzpicture}

\caption{Shock Interaction, depicting the incident and emerged shocks.}
\end{figure}
If in the states between and behind the incident shocks the gas possesses constant velocity and constant density, the situation can be mathematically represented by piecewise constant solutions of the corresponding differential equations, satisfying the shock conditions across the incident and emerging shocks. See for example the treatment in \cite{courant}. From a point of view of applications, this simple case is of great importance since it represents quite accurately the physical situation in a small neighborhood of an interaction point. However, from a mathematical point of view the existence of a solution for more general conditions is desirable.

We start our mathematical treatment of the problem at the time of interaction ($t=0$) and give two sets of data. Namely one for $x\leq 0$ and one for $x\geq 0$. These data sets correspond to the two states behind the incident shocks at $t=0$ which at the same time are to be the states ahead of the emerged shocks. Thereby we exclude any discussion of the formation, propagation and the impingement of the two incident shocks.
In any case, one can think of the incident shocks as being generated by two moving pistons (see \cite{courant}) and due to the formation and development results (see \cite{riemann}, \cite{ChristodoulouLisibach}) the setting up to $t=0$ can be rigorously established.

We assume that the two given data sets for $x\leq 0$ and for $x\geq 0$ are incompatible at $x=0$, i.e.~we assume that they do not correspond to the same physical state at the point of interaction. Instead we assume that the two data sets at the interaction point, through the jump conditions, yield shock speeds of the emerged shocks which are subsonic relative to the state behind (which is the state between the emerged shocks) and are supersonic relative to the respective states ahead which are given by the two data sets at the interaction point. These two conditions, while mathematically necessary for shocks to be determined in an evolutionary sense (see \cite{riemann}), correspond also to the only shocks observed naturally. In the following we call these conditions collectively the determinism condition. By an appropriate change of reference frame we can assume that the given data sets correspond to vanishing velocity in the state behind the emerged shocks at the interaction point.

Our assumptions consider the data in only one point, the interaction point. Therefore, our assumptions guarantee that the data at the interaction point yields a physical solution to the interaction problem in one point. The existence of such data is therefore justified exactly by the existence of the solution in the case where all the fluid states correspond to constant states as described above.

In addition to the given data we assume the existence of future developments corresponding to future domains of dependence of the data sets for $x\leq 0$ and $x\geq 0$ respectively. See figure \ref{future_development} on page \pageref{future_development}.

The present article can be viewed as an extension of \cite{lisibach2021shock} in which the reflection of a shock is treated. The major difference is that instead of a wall being present, the state in the future of the interaction point is bounded to the left also by a shock. As a consequence the boundary condition along the wall used in \cite{lisibach2021shock} is replaced by the boundary and jump conditions along the left moving shock. These conditions have been already introduced in \cite{lisibach2021shock} as the conditions along the reflected shock and are now used twice along the two emerging shocks. Since the left moving shock is represented, relative to the state behind, by a subsonic curve, the same geometric construction is used to set up the coordinate system but now placing the left moving shock at $u=v$. This set up of the characteristic coordinates, while being asymmetrical (the right moving shock being given in these coordinates by $u=av$), has the advantage to simplify the boundary and jump conditions along the left moving shock and at the same time be close to the presentation in \cite{lisibach2021shock}. While in \cite{lisibach2021shock}, to eliminate $t$ from the characteristic system and to work with the quantity $x$, was the natural choice due to the boundary condition along the wall, in the present article it is merely a choice in order for the presentation to be close to \cite{lisibach2021shock}.

Even though the present article is self contained, we recommend to read \cite{lisibach2021shock} first because it contains already many of the core ideas used in the present article but in a more elementary case.

\medskip

The present problem has been treated in a different way in \cite{LiTaTsin1}, \cite{LiTaTsin2}. While in these works the authors treat a more general set of equations, the proof is in stages of increasing difficulty, dealing with a linear system and fixed boundaries first and then extending the results to a nonlinear system and free boundaries. During this process fewer and fewer details of the proof are given. Also, the actual states ahead of the shocks are taken into account after the local existence has been established. Our approach and presentation gives additional insight into the problem and opens up the road for future progress in more involved situations.

\section{Notation\label{notation}}
We devote this section to a summary of the notation used, in order for the reader to have a single place to browse back to and look things up. This section is not supposed to introduce the notation, as usual the notation will be introduced step by step in the arguments to follow.

\begin{enumerate}
\item A $\ast$ is used to denote functions which correspond to the developments of initial data, i.e functions which describe the states ahead of the emerged shocks. These functions are given in terms of the coordinates $t$ and $x$.
\item Everything which is related to the left moving shock is denoted by a stacked index 1 and everything which is related to the right moving shock is denoted by a stacked index $2$. E.g.
  \begin{align}
    \label{eq:1}
    \ls{\alpha}^\ast(t,x)
  \end{align}
denotes the Riemann invariant $\alpha$ in the state ahead of the left moving shock.
\item In the state between the emerged shocks, which is the state behind these shocks, characteristic coordinates $u$ and $v$ are used (see \ref{u-v-coordinates}). E.g.
  \begin{align}
    \label{eq:2}
    \alpha(u,v)
  \end{align}
denotes the Riemann invariant $\alpha$ in the state behind the shocks.
\item Quantities evaluated along the shocks in the state behind are denoted by an index $+$, quantities evaluated along the shocks in the states ahead are denoted by an index $-$. E.g.
  \begin{align}
    \label{eq:3}
    \rs{\alpha}_+(v)&=\alpha(av,v),\\
    \rs{\alpha}_-(v)&=\rs{\alpha}^\ast(\rs{t}_+(v),\rs{x}_+(v))\label{eq:4}
  \end{align}
denote the Riemann invariant $\alpha$ along the right moving shock in the states behind and ahead respectively.
\end{enumerate}
For a compilation of the notation see figure \ref{notation_figure}.
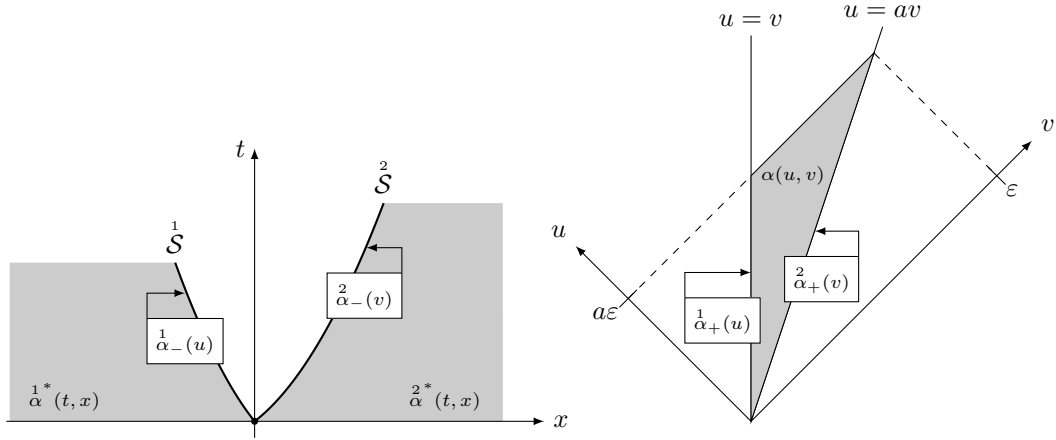
\begin{figure}[h!]
  \centering
  \begin{tikzpicture}[scale=1.1]
  
  \fill [fill=gray!40, domain=0:.6, variable=\x] (0,0) -- plot (-1.6*\x,{2*\x+2*\x^2}) -- ++(-2,0)-- (-2.96,0) -- cycle;
  \fill [fill=gray!40, domain=0:1.2, variable=\x] (0,0) -- plot (1.3*\x,{1*\x+1*\x^2}) -- ++(1.44,0)-- (3,0) -- cycle;
  \draw[thick,domain=0:.6,smooth,variable=\x] plot (-1.6*\x,{2*\x+2*\x^2})node[anchor=south]{$\overset{\scriptscriptstyle{1}}{\mathcal{S}}$};
  \draw[thick,domain=0:1.2,smooth,variable=\x] plot (1.3*\x,{1*\x+1*\x^2})node[anchor=south]{$\overset{\scriptscriptstyle{2}}{\mathcal{S}}$};
  \draw[-latex] (-3,0) -- (3.5,0) node[anchor=west] {$x$};
  \draw[-latex] (0,-.2) -- (0,3.3) node[anchor=east] {$t$};
  \fill (0,0) circle[radius=1.2pt];
  \node at (2.3,.3){\scriptsize$\rs{\alpha}^\ast(t,x)$};
  \node at (-2.3,.3){\scriptsize$\ls{\alpha}^\ast(t,x)$};
  \draw[-latex](1.78,1.8)node[xshift=0.2pt,rectangle,fill=white,draw,anchor=north east]{\scriptsize$\rs{\alpha}_-(v)$}--(1.78,2.1)--(1.35,2.1);
  \draw[-latex](-1.3,1.25)node[xshift=-0.2pt,rectangle,fill=white,draw,anchor=north west]{\scriptsize$\ls{\alpha}_-(u)$}--(-1.3,1.55)--(-.84,1.55);

  \begin{scope}[xshift=6cm,rotate=45,scale=3]
    \draw[-latex](0,0)--(1.6,0)node[anchor=south west]{$v$};
    \draw[-latex](0,0)--(0,1)node[anchor=south east]{$u$};
    \draw(0,0)--(1.1,1.1)node[anchor=south]{$u=v$};
    \draw(0,0)--(1.5,.75)node[anchor=south]{$u=av$};
    \draw[fill=gray!40](0,0)--(1.4,.7)--(0.7,0.7)--cycle;
    \node at(0,.7)[anchor=north east]{$a\varepsilon$};
    \draw[dashed](0,.7)--(.7,.7);
    \draw (-.05,.7)--(.05,.7);
    \draw[dashed](1.4,.7)--(1.4,0);
    \draw (1.4,.05)--(1.4,-.05);
    \node at(1.4,0)[anchor=north west]{$\varepsilon$};
    \node at(.82,0.58){\scriptsize$\alpha(u,v)$};
  \end{scope}
  \begin{scope}[xshift=6cm]
    \draw[-latex](1.3,2)node[xshift=0.2pt,rectangle,fill=white,draw,anchor=north east]{\scriptsize$\rs{\alpha}_+(v)$}--(1.3,2.3)--(.77,2.3);
    \draw[-latex](-.8,1.5)node[xshift=-0.2pt,rectangle,fill=white,draw,anchor=north west]{\scriptsize$\ls{\alpha}_+(u)$}--(-.8,1.8)--(0,1.8);
    
  \end{scope}

\end{tikzpicture}

  \caption{Summary of Notation. The states ahead on the left and the state behind on the right. Quantities in boxes denote evaluation along the shocks}
  \label{notation_figure}
\end{figure}

\section{Equations of Motion}
\subsection{Euler Equations}
We study one dimensional fluid flow without friction. We denote by $\rho$, $w$ and $p$ the density, velocity and pressure respectively. The equations of motion are
\begin{align}
  \label{eq:5}
  \partial_t\rho+\partial_x(\rho w)&=0,\\
  \label{eq:6}
  \partial_t w+w\partial_xw&=-\frac{1}{\rho}\partial_xp.
\end{align}
We assume $p=p(\rho)$ is a given smooth function which satisfies $\frac{dp}{d\rho}(\rho)>0$. We do not take into account entropy. Equations \eqref{eq:5}, \eqref{eq:6} follow from the conservation of mass and momentum once we assume that the quantities describing the fluid are continuously differentiable (see \cite{courant}).

\subsection{Riemann Invariants and Characteristic Equations}
The Riemann invariants are
\begin{align}
  \label{eq:7}
  \alpha\defeq\int^\rho\frac{\eta(\rho')}{\rho'}d\rho'+w,\qquad\beta\defeq\int^\rho\frac{\eta(\rho')}{\rho'}d\rho'-w,
\end{align}
where $\eta\defeq \sqrt{\frac{dp}{d\rho}}$ is the sound speed, see \cite{riemann} and \cite{Earnshaw1998}. As a consequence we have
\begin{align}
  \label{eq:8}
  w=\frac{\alpha-\beta}{2}.
\end{align}
Defining
\begin{alignat}{3}
  \label{eq:9}
  \cout&\defeq w+\eta,&\qquad\qquad\cin&\defeq w-\eta,\\
  \label{eq:10}
  L_{\textrm{out}}&\defeq\partial_t+\cout\partial_x,& L_{\textrm{in}}&\defeq\partial_t+\cin\partial_x,
\end{alignat}
we have
\begin{align}
  \label{eq:11}
  L_{\textrm{out}}\alpha=0=L_{\textrm{in}}\beta,
\end{align}
i.e.~the Riemann invariants $\alpha$, $\beta$ are invariant along the integral curves of $L_\textrm{out}$, $L_{\textrm{in}}$, respectively.

We have
\begin{align}
  \label{eq:12}
  \frac{\partial(\alpha,\beta)}{\partial(\rho,w)}&=
  \begin{pmatrix}
    \eta/\rho & 1\\\eta/\rho & -1
  \end{pmatrix}.\\
\intertext{Therefore,}
  \label{eq:13}
  \frac{\partial(\rho,w)}{\partial(\alpha,\beta)}&=
  \begin{pmatrix}
    \rho/2\eta & \rho/2\eta\\1/2 & -1/2
  \end{pmatrix}.
\end{align}

\subsection{Jump Conditions and the Determinism Condition}
If we assume that there exists a differentiable curve $t\mapsto(t,\xi(t))$ (a so called shock curve) across which the quantities describing the fluid suffer discontinuities but in the closure of both sides the differential equations are satisfied, the conservation of mass and momentum yield the following two conditions (jump conditions) on the discontinuities (see \cite{courant})
\begin{align}
  \label{eq:14}
  \jump{\rho}V&=\jump{\rho w},\\
    \label{eq:15}
  \jump{\rho w}V&=\jump{\rho w^2+p},
\end{align}
where we denote by $V$ the shock speed: $V=d\xi/dt$ and by $\jump{f}$ the difference of the function $f$ across $\xi$, i.e.
\begin{align}
  \label{eq:16}
  \jump{f}=f_+-f_-
\end{align}
where $f_+$ is the quantity evaluated behind the shock and $f_-$ is the quantity evaluated ahead of the shock\footnote{From an evolutionary point of view, looking at the time evolution of a portion of fluid, the state ahead corresponds to the part of the fluid flow line before the intersection with the shock and the state behind corresponds to the part of the fluid flow line after the intersection with the shock. Hence $\jump{f}$ corresponds to the jump in the quantity $f$ while crossing the shock curve.}:
\begin{figure}[h]
  \centering
  \begin{tikzpicture}[scale=1.2]

  \begin{scope}[yshift=7.87cm]
    \fill [fill=gray!20, domain=.5:1.9, variable=\x] (.5, -7.87) -- plot ({\x}, {.5*\x^2-8}) -- (1.9, -7.87) -- cycle;
    \fill [fill=gray!40, domain=.5:1.9, variable=\x] (.5, -6.2) -- plot ({\x}, {.5*\x^2-8}) -- (1.9, -6.2) -- cycle;
    \fill[fill=gray!20] (1.89,-7.87) rectangle (3.8,-6.2);
    \fill[fill=gray!40] (-1,-7.87) rectangle (.51,-6.2);
    \draw[thick,domain=.5:1.9,smooth,variable=\x] plot (\x,{.5*\x^2-8});
    \draw[dashed,thick](1.3,-7.155)--(1.2,-6.2);
    \draw[dashed,thick] (1.3,-7.155)--(1.2,-7.88);
  \end{scope}
  \node[anchor=west] at (1.75,1.5) {$\mathcal{S}$};
  \node[anchor=west] at (1.8,.5) {state ahead};
  \node[anchor=west] at (-1,1.4) {state behind};
  \begin{scope}[xshift=-1cm]
    \draw[-latex] (-0.2,0) -- (5.3,0) node[anchor=west] {$x$};
    \draw[-latex] (0,-0.2) -- (0,2.1) node[anchor=south] {$t$};
  \end{scope}
  \begin{scope}[xshift=3cm,yshift=2cm]
  \draw[dashed](2,-1.5)--(3,-1.5)node[anchor=west]{Fluid flow line};
  \draw[thick](2,-1)--(3,-1)node[anchor=west]{Shock curve $\mathcal{S}$};
  \end{scope}
\end{tikzpicture}

  \caption{The state ahead and behind of the shock.}
\end{figure}
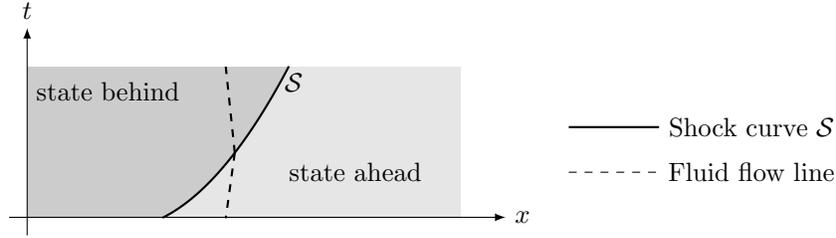

The determinism condition states that the speed of the shock is supersonic relative to the state ahead of the shock and subsonic relative to the state behind the shock. For a more general description of the determinism condition see the epilogue of \cite{christodoulouformation} or section 2.3 of \cite{ChristodoulouLisibach}.

\section{Setting the Scene}
\subsection{The States Ahead of the Emerged Shocks\label{state_ahead}}
We consider two sets of data for $\rho$, $w$ given at $t=0$ for $x\geq 0$ and $x\leq 0$. Each of these data sets possesses a future development, i.e.~there exist functions
\begin{align}
  \label{eq:17}
  \is{\rho}^\ast(t,x),\qquad\is{w}^\ast(t,x),\qquad i=1,2,
\end{align}
which solve the equations of motion and coincide with the respective data set for $t=0$. Here we use the notation that everything which is related to the left moving shock is denoted by a stacked index 1 and analogously for quantities related to the right moving shock we use the stacked index $2$. These developments of initial data are bounded in the past by $t=0$ and in the future by $t=\is{T}$, $i=1,2$. Furthermore, the development of the data given for $x\geq 0$ is bounded to the left by a right moving characteristic originating at the interaction point $(t,x)=(0,0)$ and vice versa for the development of the data given for $x\leq 0$. We are going to denote these characteristics by $\is{\mathcal{B}}$, $i=1,2$. We assume that the two given data sets correspond to a shock interaction point in the following way.

We assume that at the interaction point, the jump conditions \eqref{eq:14}, \eqref{eq:15}, applied to the jumps across both shocks (a set of four equations), with the quantities ahead of the shocks given by
\begin{align}
  \label{eq:18}
  \is{\rho}_-=\is{\rho}^\ast(0,0),\qquad \is{w}_-=\is{w}^\ast(0,0)\qquad i=1,2
\end{align}
and the conditions behind the shocks
\begin{align}
  \label{eq:19}
\ls{\rho}_+=\rs{\rho}_+,\qquad   \ls{w}_+=\rs{w}_+,
\end{align}
possesses the solution
\begin{align}
  \label{eq:20}
  \is{\rho}_+\defeq \rho_0,\qquad \is{w}_+\defeq w_0,\qquad \is{V}\defeq \is{V}_0,\qquad i=1,2,
\end{align}
such that the determinism conditions are satisfied for both shocks. By an appropriate change of reference frame we have $w_0=0$. This change of reference frame does not affect the determinism conditions across the shocks. The condition for the left moving shock is then
  \begin{align}
  \label{eq:21}
  -\eta_0<\ls{V}_0<(\ls{\cin}_0^\ast)_0,
  \end{align}
  where
  \begin{align}
    \label{eq:22}
    \eta_0=\eta(\rho_0),\qquad (\ls{\cin}_0^\ast)_0=\ls{w}^\ast(0,0)-\eta(\ls{\rho}^\ast(0,0)),
  \end{align}
  i.e.~$\eta_0$ is the sound speed in the state behind the shocks and $(\ls{\cin}_0^\ast)_0$ is the characteristic speed of the left moving characteristic in the state ahead of the left moving shock originating at the interaction point. The condition for the right moving shock is
  \begin{align}
    \label{eq:23}
    (\rs{\cout}_0^\ast)_0<\rs{V}_0<\eta_0,
  \end{align}
  where
\begin{align}
  \label{eq:24}
  (\rs{\cout}_0^\ast)_0=\rs{w}^\ast(0,0)+\eta(\rs{\rho}^\ast(0,0))
\end{align}
is the characteristic speed of the right moving characteristic in the state ahead of the right moving shock originating at the interaction point. In addition we assume
\begin{align}
  \label{eq:25}
  \ls{V}_0<0<\rs{V}_0.
\end{align}
The above described developments of the two data sets together with the solution of the jump conditions at the interaction point, satisfying the determinism conditions \eqref{eq:21}, \eqref{eq:23} and \eqref{eq:25}, set up the shock interaction problem.
\begin{figure}[h]
  \centering
  \begin{tikzpicture}[scale=1.3]
  \fill [fill=gray!40, domain=0:.5, variable=\x] (0,0) -- plot (-\x,{2*\x+2*\x^2}) -- ++(-4,0)-- (-4.5,0) -- cycle;
  \fill [fill=gray!40, domain=0:1, variable=\x] (0,0) -- plot (\x,{1*\x+1*\x^2}) -- ++(3.5,0)-- (4.5,0) -- cycle;

  \draw[thick,domain=0:.5,smooth,variable=\x] plot (-\x,{2*\x+2*\x^2})node[anchor=south]{$\overset{\scriptscriptstyle{1}}{\mathcal{B}}$};
  \draw[thick,domain=0:1,smooth,variable=\x] plot (\x,{1*\x+1*\x^2})node[anchor=south]{$\overset{\scriptscriptstyle{2}}{\mathcal{B}}$};
  
  \draw[-latex] (-4.5,0) -- (5,0) node[anchor=west] {$x$};
  \draw[-latex] (0,-.2) -- (0,2.5) node[anchor=east] {$t$};

  \draw[latex-](.8,1.3)--(1.05,1.3)node[anchor=west] {right moving charactertistic};
  \draw[latex-](-.5,1.3)--(-1.1,1.3)node[anchor=east] {left moving charactertistic};
  \node[anchor=north] at (3.2,.6) {future development};
  \node[anchor=north] at (-3.2,.6) {future development};
  \draw[latex-](3.5,0)--(3.5,-.5)node[anchor=north] {data for $x\geq 0$};
  \draw[latex-](-3.5,0)--(-3.5,-.5)node[anchor=north] {data for $x\leq 0$};
  \node[anchor=north] at (0,-.5){shock tangents};
  \draw[-latex](.5,-.5)--(.5,.2);
  \draw[-latex](-.5,-.5)--(-.5,.4);
  \draw[dashed,thick](0,0)--(-.8,.8);
  \draw[dashed,thick](0,0)--(1.4,.8);
  \fill (0,0) circle[radius=1.2pt];
\end{tikzpicture}

  \caption{Future developments of the two initial data sets.}
  \label{future_development}
\end{figure}
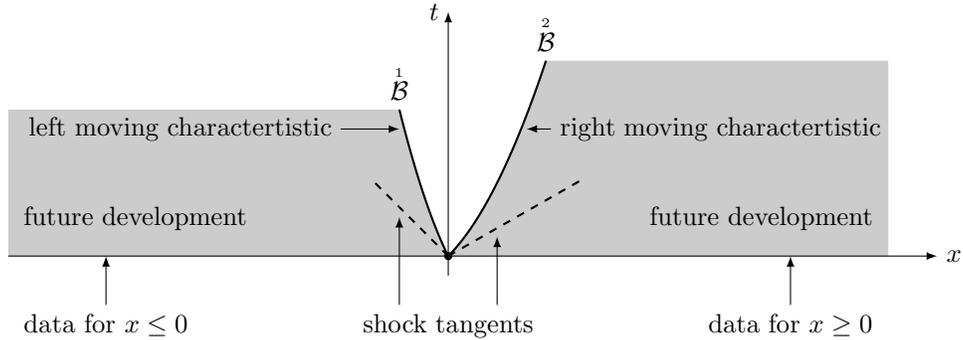

\subsection{The Shock Interaction Problem}
The shock interaction problem is the following:\\
Find two world lines $\is{\mathcal{S}}$, ($i=1,2$), lying inside the respective future developments, originating at the interaction point $(t,x)=(0,0)$, together with a solution of the equations of motion in a domain in spacetime bounded by the $\is{\mathcal{S}}$, such that across the $\is{\mathcal{S}}$ the new solution displays jumps relative to the solution in the future developments, jumps which satisfy the jump conditions. The domain to the left of $\ls{\mathcal{S}}$, where the solution in the future development of the data for $x\leq 0$ holds, is called the state ahead of $\ls{\mathcal{S}}$, the domain to the right of $\rs{\mathcal{S}}$, where the solution in the future development of the data for $x\geq 0$ holds, is called the state ahead of $\rs{\mathcal{S}}$ and the domain between the $\is{\mathcal{S}}$, where the new solution holds is called the state behind of $\is{\mathcal{S}}$. $\is{\mathcal{S}}$ are to be supersonic relative to the states ahead and subsonic relative to the state behind. The requirements in the last sentence are the determinism conditions across the shocks. We are going to bound the state behind also by a right moving characteristic. See figure \ref{interaction_problem}.
\begin{figure}[h!]
  \centering
  \begin{tikzpicture}[scale=1.3]

    \fill [fill=gray!60, domain=0:.4, variable=\x] (0,0) -- plot (-1.6*\x,{2*\x+2*\x^2}) --(0,1.56218) -- cycle;
  \fill [fill=gray!60, domain=0:1.2, variable=\x] (0,0) -- plot (1.3*\x,{1*\x+1*\x^2}) --(0,1.56218) -- cycle;

  \fill [fill=gray!40, domain=0:.4, variable=\x] (0,0) -- plot (-1.6*\x,{2*\x+2*\x^2}) -- ++(-3.36,0)-- (-4,0) -- cycle;
  \fill [fill=gray!40, domain=0:1.2, variable=\x] (0,0) -- plot (1.3*\x,{1*\x+1*\x^2}) -- ++(2.44,0)-- (4,0) -- cycle;

  \draw[dashed,domain=0:.45,smooth,variable=\x] plot (-1*\x,{2*\x+2*\x^2})node[anchor=south]{$\overset{\scriptscriptstyle{1}}{\mathcal{B}}$};
  \draw[dashed,domain=0:1.2,smooth,variable=\x] plot (1*\x,{1*\x+1*\x^2})node[anchor=south]{$\overset{\scriptscriptstyle{2}}{\mathcal{B}}$};

  \draw[thick,domain=0:.4,smooth,variable=\x] plot (-1.6*\x,{2*\x+2*\x^2})node[anchor=south]{$\overset{\scriptscriptstyle{1}}{\mathcal{S}}$};
  \draw[thick,domain=0:1.2,smooth,variable=\x] plot (1.3*\x,{1*\x+1*\x^2})node[anchor=south]{$\overset{\scriptscriptstyle{2}}{\mathcal{S}}$};


 \draw[dashed,thick](0,0)--(-.8,.8);
 \draw[dashed,thick](0,0)--(1.4,.8);
  \node[anchor=north] at (0,-.5){shock tangents};
  \draw[-latex](.5,-.5)--(.5,.2);
  \draw[-latex](-.5,-.5)--(-.5,.4);

  \node[anchor=north] at (2.8,.8) {state ahead of $\overset{\scriptscriptstyle{2}}{\mathcal{S}}$};
  \node[anchor=north] at (-2.8,.8) {state ahead of $\overset{\scriptscriptstyle{1}}{\mathcal{S}}$};
  \draw[latex-](3,0)--(3,-.5)node[anchor=north] {data for $x\geq 0$};
  \draw[latex-](-3,0)--(-3,-.5)node[anchor=north] {data for $x\leq 0$};

  \draw[-latex] (-4,0) -- (4.5,0) node[anchor=west] {$x$};
  \draw[-latex] (0,-.2) -- (0,3.3) node[anchor=east] {$t$};

  \fill (0,0) circle[radius=1.2pt];
\end{tikzpicture}

  \caption{Shock interaction Problem. The state behind the emerged shocks in dark shade.}
  \label{interaction_problem}
\end{figure}
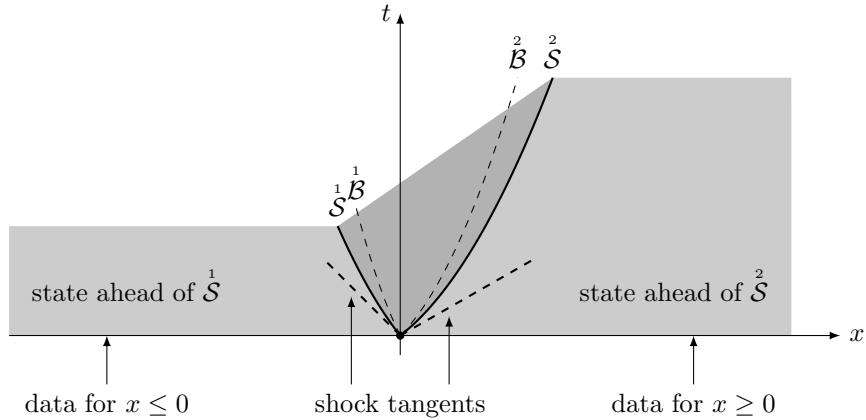

\section{Characteristic Coordinates}
\subsection{Choice of Coordinates}
Introducing coordinates $u$, $v$\label{u-v-coordinates}, such that $u$ is constant along integral curves of $L_{\textrm{out}}$ and $v$ is constant along integral curves of $L_{\textrm{in}}$ (see \eqref{eq:10}), \eqref{eq:11} becomes
\begin{align}
  \label{eq:26}
  \frac{\partial\alpha}{\partial v}=0=\frac{\partial\beta}{\partial u}.
\end{align}
In $u$-$v$-coordinates, $t$ and $x$ satisfy the characteristic equations
\begin{align}
  \label{eq:27}
  \frac{\partial x}{\partial v}=\cout\frac{\partial t}{\partial v},\qquad\frac{\partial x}{\partial u}=\cin\frac{\partial t}{\partial u}.
\end{align}
We choose the $u$-$v$-coordinates such that the following conditions hold:
\begin{enumerate}
\item The origin $(u,v)=(0,0)$ corresponds to the interaction point $(t,x)=(0,0)$.
\item The left moving shock $\ls{\mathcal{B}}$ corresponds to $u=v$.
\item The right moving shock $\rs{\mathcal{B}}$ corresponds to $u=av$, for a constant $a$ with $0<a<1$ (see \eqref{eq:36}, \eqref{eq:43} below).
\item We have
  \begin{align}
    \label{eq:28}
    \pp{x}{v}(0,0)=1.
  \end{align}
\end{enumerate}

This choice of coordinates is justified as follows: The condition $u=v$ on the left moving shock $\ls{\mathcal{B}}$ can be imposed once a $v=\textrm{const.}$~curve intersects $\ls{\mathcal{B}}$ exactly once, which is the case because $\ls{\mathcal{B}}$ is a subsonic curve. Thus $u$ is determined once $v$ is determined. The right moving shock $\rs{\mathcal{B}}$ being subsonic relative to the state behind, it is given in $u$-$v$-coordinates by an equation of the form $u=f(v)$, where $f$ is an increasing function. We must have $f(0)=0$ since the interaction point, which is the origin, is on $\ls{\mathcal{B}}$ as well as on $\rs{\mathcal{B}}$. Moreover, $f(v)<v$ for $v>0$ because $\rs{\mathcal{B}}$ is to the right of $\ls{\mathcal{B}}$. Now, we can set $u=\phi(\tilde{u})$, $v=\phi(\tilde{v})$, where $\phi$ is any increasing function such that $\phi(0)=0$, without changing the equation of $\ls{\mathcal{B}}$: $\tilde{u}=\tilde{v}$ or the fact that the interaction point corresponds to the origin. The equation of $\rs{\mathcal{B}}$ is then transformed to $\tilde{u}=\tilde{f}(\tilde{v})$, where
\begin{align}
  \label{eq:29}
  \tilde{f}=\phi^{-1}\circ f\circ \phi.
\end{align}
It follows that $\tilde{f}'(0)=f'(0)$. Let $a=f'(0)$. We have $0<a<1$. The problem is then to choose an appropriate $\phi$ such that $\tilde{f}(\tilde{v})=a\tilde{v}$. It can easily be shown by an iteration method starting with the $0$'th iterate $\phi_0$ being the identity map $\phi_0(x)=x$ that the equation $\phi\circ \tilde{f}=f\circ \phi$ with $\tilde{f}(x)=ax$ has a solution $\phi$ defined on $[0,\varepsilon^\ast]$ for suitably small $\varepsilon^\ast>0$. We can then extend this local solution to a global one using a continuity argument.

Let $\varepsilon>0$. In the following we consider the domain which is bounded by $\ls{\mathcal{B}}$, $\rs{\mathcal{B}}$ and the right moving characteristic $u=a\varepsilon$:
\begin{align}
  \label{eq:30}
  T_\varepsilon=\Big\{(u,v)\in\mathbb{R}^2:0\leq u\leq v\leq \frac{u}{a}\leq\varepsilon\Big\}.
\end{align}
See figure \ref{domain}.
\begin{figure}[h]
  \centering
  \begin{tikzpicture}[scale=1.8]

  \begin{scope}[rotate=45]
    \draw[-latex](0,0)--(2,0)node[anchor=south west]{$v$};
    \draw[-latex](0,0)--(0,2)node[anchor=south east]{$u$};
    \draw(0,0)--(1.5,1.5)node[anchor=south]{$u=v$};
    \draw(0,0)--(1.8,.9)node[anchor=south]{$u=av$};
    \draw[fill=gray!40](0,0)--(1.4,.7)--(0.7,0.7)--cycle;
    \node at(0,.7)[anchor=north east]{$a\varepsilon$};
    \draw[dashed](0,.7)--(.7,.7);
    \draw (-.05,.7)--(.05,.7);
    \draw[dashed](1.4,.7)--(1.4,0);
    \draw (1.4,.05)--(1.4,-.05);
    \node at(1.4,0)[anchor=north west]{$\varepsilon$};
    \node at(.71,.5){$T_\varepsilon$};
  \end{scope}

\end{tikzpicture}

  \caption{The domain $T_\varepsilon$. The left moving shock $\protect\ls{\mathcal{B}}$ corresponds to $\{u=v\}$, the right moving shock $\protect\rs{\mathcal{B}}$ corresponds to $\{u=av\}$.}
  \label{domain}
\end{figure}
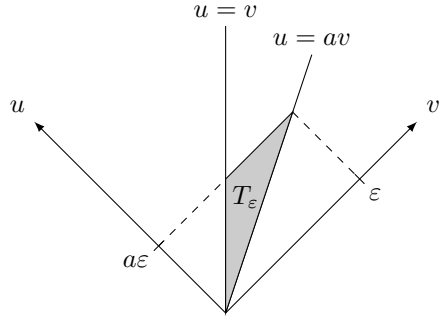

\subsection{Boundary Conditions}
We look at the boundary conditions
\begin{align}
  \label{eq:31}
  \ls{V}d\ls{t}_+=d\ls{x}_+,\qquad \rs{V}d\rs{t}_+=d\rs{x}_+
\end{align}
along $\ls{S}$, $\rs{S}$ respectively. Here $\is{V}$, $i=1,2$ are the shock speeds and we use the notation that for any function $f(u,v)$:
\begin{align}
  \label{eq:32}
  \ls{f}_+(u)=f(u,u),\qquad \rs{f}_+(v)=f(av,v).
\end{align}
I.e.~the index $+$ denotes evaluation along the shock and the stacked index indicates which shock the quantity is related to. We use the two different arguments $u$, $v$ in order to indicate that for a function $\ls{f}_+(u)$ it holds that $u\in[0,a\varepsilon]$, while for a function $\rs{f}_+(v)$ it holds that $v\in[0,\varepsilon]$ (see figure \ref{domain}). We define
\begin{align}
  \label{eq:33}
  \ls{\Gamma}\defeq\frac{\ls{\cout}_+}{\ls{\cin}_+}\,\frac{\ls{V}-\ls{\cin}_+}{\ls{\cout}_+-\ls{V}},\qquad\rs{\Gamma}\defeq a\frac{\rs{\cout}_+}{\rs{\cin}_+}\,\frac{\rs{V}-\rs{\cin}_+}{\rs{\cout}_+-\rs{V}}.
\end{align}
Then, using the characteristic system \eqref{eq:27}, the boundary conditions \eqref{eq:31} become
\begin{align}
  \label{eq:34}
  \pp{x}{v}(u,u)=\pp{x}{u}(u,u)\ls{\Gamma}(u),\qquad \pp{x}{v}(av,v)=\pp{x}{u}(av,v)\rs{\Gamma}(v),
\end{align}
respectively and it holds that (by our choice $\pp{x}{v}(0,0)=1$)
\begin{align}
  \label{eq:35}
  \Gamma_0\defeq\ls{\Gamma}_0=\rs{\Gamma}_0=\frac{1}{\pp{x}{u}(0,0)},
\end{align}
where here and in the following, the index $0$ denotes evaluation at the point of intersection. From \eqref{eq:33}, \eqref{eq:35} we obtain
\begin{align}
  \label{eq:36}
  a&=\frac{\ls{V}_0-\cin_0}{\cout_0-\ls{V}_0}\,\,\frac{\rs{V}_0-\cout_0}{\cin_0-\rs{V}_0}\nonumber\\
   &=\frac{\eta_0+\ls{V}_0}{\eta_0-\ls{V}_0}\,\,\frac{\eta_0-\rs{V}_0}{\eta_0+\rs{V}_0},
\end{align}
where we used $\cin_0=-\eta_0$, $\cout_0=\eta_0$ (recall that $w_0=0$). $\ls{S}$ being subsonic relative to the state behind (see \eqref{eq:21}), we have
\begin{align}
  \label{eq:37}
  -\eta_0<\ls{V}_0,  
\end{align}
which implies
\begin{align}
  \label{eq:38}
  0<\eta_0+\ls{V}_0.    
\end{align}
From (see \eqref{eq:25})
\begin{align}
  \label{eq:39}
  \ls{V}_0<w_0=0      
\end{align}
we have
\begin{align}
  \label{eq:40}
  0<-\ls{V}_0<\eta_0-\ls{V}_0.      
\end{align}
From \eqref{eq:38}, \eqref{eq:39}, \eqref{eq:40} we deduce
\begin{align}
  \label{eq:41}
  0<\frac{\eta_0+\ls{V}_0}{\eta_0-\ls{V}_0}<1.
\end{align}
The analogous argument with our assumption of $\rs{S}$ being subsonic relative to the state behind yields
\begin{align}
  \label{eq:42}
    0<\frac{\eta_0-\rs{V}_0}{\eta_0+\rs{V}_0}<1.
\end{align}
From \eqref{eq:41}, \eqref{eq:42}, in view of \eqref{eq:36} we obtain
\begin{align}
  \label{eq:43}
  0<a<1,
\end{align}
as imposed in the set up of the characteristic coordinate system. We note that
\begin{align}
  \label{eq:44}
  \Gamma_0=\ls{\Gamma}_0=-\frac{\eta_0+\ls{V}_0}{\eta_0-\ls{V}_0}<0,
\end{align}
i.e.
\begin{align}
  \label{eq:45}
  \pp{x}{u}(0,0)=\frac{1}{\Gamma_0}<0.
\end{align}

The functions $\alpha(u,v)$, $\beta(u,v)$ are given along the left moving shock by
\begin{alignat}{3}
  \label{eq:46}
  \ls{\alpha}_+(u)&=\alpha(u,u),\qquad& \ls{\beta}_+(u)&=\beta(u,u),\\
\intertext{and along the right moving shock by}
  \label{eq:47}
  \rs{\alpha}_+(v)&=\alpha(av,v),\qquad&\rs{\beta}_+(v)&=\beta(av,v).
\end{alignat}
From \eqref{eq:26} we then obtain\footnote{We note that equations of the type \eqref{eq:48}, \eqref{eq:49} naturally break with our convention of using the argument $u$ along $\ls{\mathcal{S}}$ and the argument $v$ along $\rs{\mathcal{S}}$. However, the factor $a$ in $\ls{\alpha}_+(av)$ takes care of the argument of the function $\ls{\alpha}$ being restricted to the interval $[0,a\varepsilon]$ as it should be.}
\begin{align}
  \label{eq:48}
  \ls{\beta}_+(u)&=\beta(u,u)=\beta(au,u)=\rs{\beta}_+(u),\\
  \rs{\alpha}_+(v)&=\alpha(av,v)=\alpha(av,av)=\ls{\alpha}_+(av),\label{eq:49}
\end{align}
which implies\footnote{The way we write the derivatives as in \eqref{eq:50}, \eqref{eq:51} (we avoid the prime notation) points out to which shock the quantity is related.}
\begin{align}
  \label{eq:50}
  \frac{d\ls{\beta}_+}{du}(u)&=\frac{d\rs{\beta}_+}{dv}(u),\\
  \frac{d\rs{\alpha}_+}{dv}(v)&=a\frac{d\ls{\alpha}_+}{du}(av).\label{eq:51}
\end{align}

\subsection{Jump Conditions}
The jump conditions \eqref{eq:14}, \eqref{eq:15} are equivalent to
\begin{align}
  \label{eq:52}
  V&=\frac{\jump{\rho w}}{\jump{\rho}},\\
    \label{eq:53}
  0&=\jump{\rho w}^2-\jump{\rho w^2+p}\jump{\rho}\defeq I(\rho_+,\rho_-,w_+,w_-).
\end{align}
$\rho$ and $w$ are given smooth functions of $\alpha$, $\beta$. Defining
\begin{align}
  \label{eq:54}
  J(\alpha_+,\beta_+,\alpha_-,\beta_-)\defeq I\Big(\rho(\alpha_+,\beta_+),\rho(\alpha_-,\beta_-),w(\alpha_+,\beta_+),w(\alpha_-,\beta_-)\Big),
\end{align}
the above jump condition \eqref{eq:53} in terms of the Riemann invariants is
\begin{align}
  \label{eq:55}
  J(\alpha_+,\beta_+,\alpha_-,\beta_-)=0.
\end{align}
We recall that the solution of the system of jump conditions \eqref{eq:52}, \eqref{eq:53} across both shocks at the interaction point is $\rho_0$, $w_0$, $\is{V}_0$, $i=1,2$ (see \eqref{eq:20}). The value of $\rho_0$ together with the condition $w_0=(\alpha_0-\beta_0)/2=0$ (see \eqref{eq:8}) determines the value $\alpha_0=\beta_0$. This value substituted for $\alpha_+$, $\beta_+$ together with the values of $\alpha$, $\beta$ at the interaction point in the states ahead (hence given by the data) satisfy the jump condition \eqref{eq:55} across both shocks, i.e.
\begin{align}
  \label{eq:56}
  J(\beta_0,\beta_0,\is{\alpha}_{-0},\is{\beta}_{-0})=0,\qquad i=1,2.
\end{align}
Here
\begin{align}
  \label{eq:57}
  \is{\alpha}_{-0}=\is{\alpha}^\ast(0,0),\qquad\is{\beta}_{-0}=\is{\beta}^\ast(0,0),\qquad i=1,2
\end{align}
are given by the solutions in the states ahead at the interaction point. We have (see the derivation in \cite{lisibach2021shock})
\begin{align}
  \label{eq:58}
  \pp{J}{\alpha_+}(\alpha_+,\beta_+,\alpha_-,\beta_-)&=-\frac{\jump{\rho}\rho_+}{2\eta_+}(\cout_+-V)^2,\\
  \pp{J}{\beta_+}(\alpha_+,\beta_+,\alpha_-,\beta_-)&=-\frac{\jump{\rho}\rho_+}{2\eta_+}(V-\cin_+)^2.  \label{eq:59}
\end{align}
In these equations all the quantities on the right hand side are functions of $\alpha_+$, $\beta_+$, $\alpha_-$, $\beta_-$, either through given functions of $\rho$ and $w$ or, as in the case for $V$, through the jump condition \eqref{eq:52}. We have
\begin{align}
  \label{eq:60}
  \ls{V}_0-\cout_0&=\ls{V}_0-\eta_0<0,\\
  \rs{V}_0-\cin_0&=\rs{V}_0+\eta_0>0.  \label{eq:61}
\end{align}
(see \eqref{eq:40} and the analogous statement for the right moving shock). Hence
\begin{align}
  \label{eq:62}
  \pp{J}{\alpha_+}(\beta_0,\beta_0,\ls{\alpha}_{-0},\ls{\beta}_{-0})&\neq 0,\\
  \pp{J}{\beta_+}(\beta_0,\beta_0,\rs{\alpha}_{-0},\rs{\beta}_{-0})&\neq 0.  \label{eq:63}
\end{align}
Using the implicit function theorem, we conclude from \eqref{eq:56}, \eqref{eq:62}, \eqref{eq:63} that there exist smooth functions $\ls{H}(\beta_+,\alpha_-,\beta_-)$, $\rs{H}(\alpha_+,\alpha_-,\beta_-)$ such that
\begin{align}
  \label{eq:64}
  \beta_0=\ls{H}(\beta_0,\ls{\alpha}_{-0},\ls{\beta}_{-0})=\rs{H}(\beta_0,\rs{\alpha}_{-0},\rs{\beta}_{-0})
\end{align}
and
\begin{align}
  \label{eq:65}
  J\Big(\ls{H}(\beta_+,\ls{\alpha}_-,\ls{\beta}_-),\beta_+,\ls{\alpha}_-,\ls{\beta}_-\Big)&=0,\\
    J\Big(\alpha_+,\rs{H}(\alpha_+,\rs{\alpha}_-,\rs{\beta}_-),\rs{\alpha}_-,\rs{\beta}_-\Big)&=0,  \label{eq:66}
\end{align}
for $(\beta_+,\ls{\alpha}_-,\ls{\beta}_-)$ sufficiently close to $(\beta_0,\ls{\alpha}_{-0},\ls{\beta}_{-0})$ and $(\alpha_+,\rs{\alpha}_-,\rs{\beta}_-)$ sufficiently close to $(\beta_0,\rs{\alpha}_{-0},\rs{\beta}_{-0})$.

Along the left and right moving shocks we have
\begin{align}
  \label{eq:67}
  \ls{\alpha}_+(u)&=\ls{H}(\ls{\beta}_+(u),\ls{\alpha}_-(u),\ls{\beta}_-(u)),\\
  \rs{\beta}_+(v)&=\rs{H}(\rs{\alpha}_+(v),\rs{\alpha}_-(v),\rs{\beta}_-(v)),\label{eq:68}               
\end{align}
respectively. Taking the derivative of the first one of these we obtain
\begin{align}
  \label{eq:69}
  \frac{d\ls{\alpha}_+}{du}(u)&=\ls{F}(\ls{\beta}_+(u),\ls{\alpha}_-(u),\ls{\beta}_-(u))\frac{d\ls{\beta}_+}{du}(u)\nonumber\\
  &\quad+\ls{M}_1(\ls{\beta}_+(u),\ls{\alpha}_-(u),\ls{\beta}_-(u))\frac{d\ls{\alpha}_-}{du}(u)\nonumber\\
  &\quad+\ls{M}_2(\ls{\beta}_+(u),\ls{\alpha}_-(u),\ls{\beta}_-(u))\frac{d\ls{\beta}_-}{du}(u),
\end{align}
where
\begin{align}
  \label{eq:70}
  \ls{F}(\beta_+,\alpha_-,\beta_-)&\defeq\pp{\ls{H}}{\beta_+}(\beta_+,\alpha_-,\beta_-)\nonumber\\
                                                     &=-\frac{\pp{J}{\beta_+}\Big(\ls{H}(\beta_+,\alpha_-,\beta_-),\beta_+,\alpha_-,\beta_-\Big)}{\pp{J}{\alpha_+}\Big(\ls{H}(\beta_+,\alpha_-,\beta_-),\beta_+,\alpha_-,\beta_-\Big)}
\end{align}
and similar expressions hold for $\ls{M}_1$, $\ls{M}_2$. Using \eqref{eq:58}, \eqref{eq:59} we obtain
\begin{align}
  \label{eq:71}
  \ls{F}(\ls{\beta}_+,\ls{\alpha}_-,\ls{\beta}_-)=-\Bigg(\frac{\ls{V}-\ls{\cin}_+}{\ls{\cout}_+-\ls{V}}\Bigg)^2.
\end{align}
Taking the derivative of \eqref{eq:68} we obtain
\begin{align}
  \label{eq:72}
  \frac{d\rs{\beta}_+}{dv}(v)=\rs{F}(\rs{\alpha}_+(v),\rs{\alpha}_-(v),\rs{\beta}_-(v))\frac{d\rs{\alpha}_+}{dv}(v)+\ldots
\end{align}
where we denote by $\ldots$ terms analogous to the ones showing up in the second and third lines of \eqref{eq:69}. Through an analogous definition as in \eqref{eq:70}, we have
\begin{align}
  \label{eq:73}
  \rs{F}(\rs{\alpha}_+,\rs{\alpha}_-,\rs{\beta}_-)=-\Bigg(\frac{\rs{\cout}_+-\rs{V}}{\rs{V}-\rs{\cin}_+}\Bigg)^2.
\end{align}
We note that by \eqref{eq:36} (and also recalling that $\cin_0=-\eta_0$, $\cout_0=\eta_0$) we have
\begin{align}
  \label{eq:74}
  \ls{F}_0\rs{F}_0=a^2,
\end{align}
where
\begin{align}
  \label{eq:75}
  \is{F}_0=\is{F}(\beta_0,\is{\alpha}_{-0},\is{\beta}_{-0}),\qquad i=1,2.
\end{align}

Using the definitions
\begin{align}
  \label{eq:76}
  \alpha_0'\defeq\frac{d\ls{\alpha}_+}{du}(0),\qquad\beta_0'\defeq\frac{d\rs{\beta}_+}{dv}(0),
\end{align}
from \eqref{eq:48}, \eqref{eq:49}, \eqref{eq:50}, \eqref{eq:51}, \eqref{eq:69}, \eqref{eq:72} in the case $u=v=0$ we obtain
\begin{align}
  \label{eq:77}
  \alpha_0'&=\ls{F}_0\beta_0'+\ls{M}_{10}\frac{d\ls{\alpha}_-}{dv}(0)+\ls{M}_{20}\frac{d\ls{\beta}_-}{dv}(0),\\
    \beta_0'&=a\rs{F}_0\alpha_0'+\rs{M}_{10}\frac{d\rs{\alpha}_-}{dv}(0)+\rs{M}_{20}\frac{d\rs{\beta}_-}{dv}(0).\label{eq:78}
\end{align}
We have
\begin{align}
  \label{eq:79}
  \ls{\alpha}_-(u)=\ls{\alpha}^\ast(\ls{t}_+(u),\ls{x}_+(u)),\qquad\ls{\beta}_-(u)=\ls{\beta}^\ast(\ls{t}_+(u),\ls{x}_+(u)).
\end{align}
Taking the derivative yields
\begin{align}
  \label{eq:80}
  \frac{d\ls{\alpha}_-}{du}(0)&=\Bigg(\pp{\ls{\alpha}^\ast}{t}\Bigg)_0\frac{d\ls{t}_+}{du}(0)+\Bigg(\pp{\ls{\alpha}^\ast}{x}\Bigg)_0\frac{d\ls{x}_+}{du}(0),\\
   \frac{d\ls{\beta}_-}{du}(0)&=\Bigg(\pp{\ls{\beta}^\ast}{t}\Bigg)_0\frac{d\ls{t}_+}{du}(0)+\Bigg(\pp{\ls{\beta}^\ast}{x}\Bigg)_0\frac{d\ls{x}_+}{du}(0).\label{eq:81}
\end{align}
The partial derivatives of $\ls{\alpha}^\ast(t,x)$, $\ls{\beta}^\ast(t,x)$ are given by the solution in the state ahead of the left moving shock. Recalling $\ls{t}_+(u)=t(u,u)$, $\ls{x}_+(u)=x(u,u)$ we have
\begin{align}
  \label{eq:82}
    \frac{d\ls{t}_+}{du}(u)=\pp{t}{u}(u,u)+\pp{t}{v}(u,u),\qquad\frac{d\ls{x}_+}{du}(u)=\pp{x}{u}(u,u)+\pp{x}{v}(u,u)
\end{align}
and using the characteristic equations \eqref{eq:27} together with $\cout_0=\eta_0=-\cin_0$, $\pp{x}{v}(0,0)=1$ and \eqref{eq:45} we obtain
\begin{align}
  \label{eq:83}
\frac{d\ls{t}_+}{du}(0)=\frac{1}{\eta_0}\left(-\frac{1}{\Gamma_0}+1\right),\qquad \frac{d\ls{x}_+}{du}(0)=\frac{1}{\Gamma_0}+1.
\end{align}
Using these in \eqref{eq:80}, \eqref{eq:81} we obtain
\begin{align}
  \label{eq:84}
  \frac{d\ls{\alpha}_-}{du}(0)&=\Bigg(\pp{\ls{\alpha}^\ast}{t}\Bigg)_0\frac{1}{\eta_0}\left(-\frac{1}{\Gamma_0}+1\right)+\Bigg(\pp{\ls{\alpha}^\ast}{x}\Bigg)_0\left(\frac{1}{\Gamma_0}+1\right),\\
  \frac{d\ls{\beta}_-}{du}(0)&=\Bigg(\pp{\ls{\beta}^\ast}{t}\Bigg)_0\frac{1}{\eta_0}\left(-\frac{1}{\Gamma_0}+1\right)+\Bigg(\pp{\ls{\beta}^\ast}{x}\Bigg)_0\left(\frac{1}{\Gamma_0}+1\right).\label{eq:85}
\end{align}
Recalling $\rs{x}_+(v)=x(av,v)$, $\rs{t}_+(v)=t(av,v)$ we obtain, analogously,
\begin{align}
  \label{eq:86}
  \frac{d\rs{\alpha}_-}{dv}(0)&=\Bigg(\pp{\rs{\alpha}^\ast}{t}\Bigg)_0\frac{1}{\eta_0}\left(-\frac{a}{\Gamma_0}+1\right)+\Bigg(\pp{\rs{\alpha}^\ast}{x}\Bigg)_0\left(\frac{a}{\Gamma_0}+1\right),\\
  \frac{d\rs{\beta}_-}{dv}(0)&=\Bigg(\pp{\rs{\beta}^\ast}{t}\Bigg)_0\frac{1}{\eta_0}\left(-\frac{a}{\Gamma_0}+1\right)+\Bigg(\pp{\rs{\beta}^\ast}{x}\Bigg)_0\left(\frac{a}{\Gamma_0}+1\right).\label{eq:87}
\end{align}

We rewrite \eqref{eq:77}, \eqref{eq:78} in the form
\begin{align}
  \label{eq:88}
  M
  \begin{pmatrix}
    \alpha_0'\\\beta_0'
  \end{pmatrix}=
  \begin{pmatrix}
    a_0\\b_0
  \end{pmatrix},
\end{align}
where
\begin{align}
  \label{eq:89}
  M&=
  \begin{pmatrix}
    1&-\ls{F}_0\\-a\rs{F}_0 & 1
  \end{pmatrix}\\
  a_0&=\ls{M}_{10}\frac{d\ls{\alpha}_-}{du}(0)+\ls{M}_{20}\frac{d\ls{\beta}_-}{du}(0),\label{eq:90}\\
  b_0&=\rs{M}_{10}\frac{d\rs{\alpha}_-}{dv}(0)+\rs{M}_{20}\frac{d\rs{\beta}_-}{dv}(0).\label{eq:91}
\end{align}
By \eqref{eq:74} we have
\begin{align}
  \label{eq:92}
  \det(M)=1-a\ls{F}_0\rs{F}_0=1-a^3>0,
\end{align}
and we have
\begin{align}
  \label{eq:93}
  \begin{pmatrix}
    \alpha_0' \\ \beta_0'
  \end{pmatrix}=M^{-1}
                \begin{pmatrix}
                  a_0 \\ b_0
                \end{pmatrix}.
\end{align}
I.e.~the constants $\alpha_0'$, $\beta_0'$ are given in terms of the solution in the states ahead at the interaction point and the solution of the jump conditions at the interaction point together with the condition $w_0=0$. In the following we use \eqref{eq:93} as definition for the constants $\alpha_0'$, $\beta_0'$.

\subsection{Equations for $\pp{x}{u}$, $\pp{x}{v}$}
In the following we consider a point in $T_\varepsilon$ with coordinates $(u,v)$. We have
\begin{align}
  \label{eq:94}
  \pp{x}{u}(u,v)=\pp{x}{u}(u,u)+\int_u^v\pppp{x}{u}{v}(u,v')dv'.
\end{align}
This corresponds to an integration along a right moving characteristic, starting on the left moving shock $\{u=v\}$ and ending at $(u,v)$. We rewrite the first term on the right as
\begin{align}
  \label{eq:95}
  \pp{x}{u}(u,u)&=\frac{1}{\ls{\Gamma}(u)}\pp{x}{v}(u,u)\nonumber\\
                &=\frac{1}{\ls{\Gamma}(u)}\left(\pp{x}{v}(au,u)+\int_{au}^u\pppp{x}{u}{v}(u',u)du'\right)\nonumber\\
                &=\frac{1}{\ls{\Gamma}(u)}\left(\rs{\Gamma}(u)\pp{x}{u}(au,u)+\int_{au}^u\pppp{x}{u}{v}(u',u)du'\right),
\end{align}
where for the first equality we use the boundary condition on the left moving shock, i.e.~the first of \eqref{eq:34}, for the second equality we integrate along a left moving characteristic starting at the right moving shock and ending on the left moving shock and for the last equality we use the boundary condition on the right moving shock, i.e.~the second of \eqref{eq:34}. Substituting \eqref{eq:95} into \eqref{eq:94} we obtain
\begin{align}
  \label{eq:96}
  \pp{x}{u}(u,v)&=\frac{\rs{\Gamma}(u)}{\ls{\Gamma}(u)}\,\pp{x}{u}(au,u)\nonumber\\
  &\qquad+\frac{1}{\ls{\Gamma}(u)}\int_{au}^u\pppp{x}{u}{v}(u',u)du'+\int_u^v\pppp{x}{u}{v}(u,v')dv'.
\end{align}
The path of integration, as given by the two integrals in this equation, are shown in figure \ref{integration_paths} on the left.

We have
\begin{align}
  \label{eq:97}
  \pp{x}{v}(u,v)=\pp{x}{v}(av,v)+\int_{av}^u\pppp{x}{u}{v}(u',v)du'.
\end{align}
This corresponds to integrating along a left moving characteristic from the right moving shock to the point $(u,v)$. We rewrite the first term on the right as
\begin{align}
  \label{eq:98}
  \pp{x}{v}(av,v)&=\rs{\Gamma}(v)\pp{x}{u}(av,v)\nonumber\\
                 &=\rs{\Gamma}(v)\left(\pp{x}{v}(av,av)+\int_{av}^v\pppp{x}{u}{v}(av,v')dv'\right)\nonumber\\
                 &=\rs{\Gamma}(v)\left(\frac{1}{\ls{\Gamma}(av)}\pp{x}{v}(av,av)+\int_{av}^v\pppp{x}{u}{v}(av,v')dv'\right),
\end{align}
where we again used the boundary conditions \eqref{eq:34}. Substituting \eqref{eq:98} into \eqref{eq:97} we obtain
\begin{align}
  \label{eq:99}
  \pp{x}{v}(u,v)&=\frac{\rs{\Gamma}(v)}{\ls{\Gamma}(av)}\,\pp{x}{v}(av,av)\nonumber\\
  &\qquad+\rs{\Gamma}(v)\int_{av}^v\pppp{x}{u}{v}(av,v')dv'+\int_{av}^u\pppp{x}{u}{v}(u',v)du'.
\end{align}
The path of integration, as given by the two integrals in this equation, is shown in figure \ref{integration_paths} on the right.

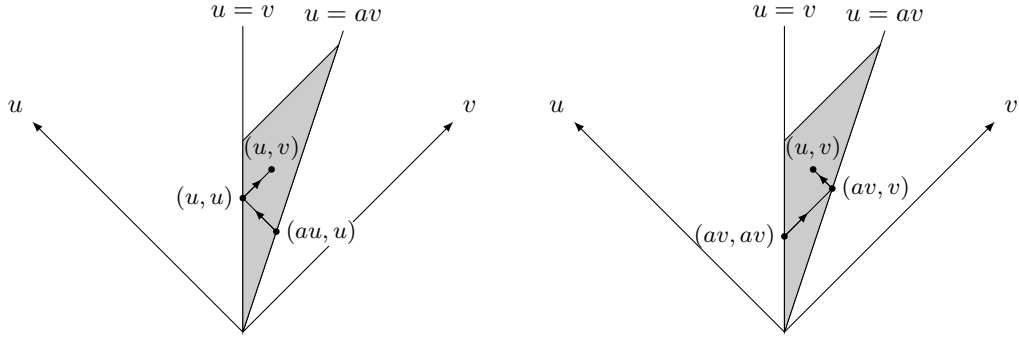
\begin{figure}[h!]
  \centering
  \contourlength{4pt}
\usetikzlibrary{decorations.markings}
\tikzset{->-/.style={decoration={
  markings,
  mark=at position .5 with {\arrow{>}}},postaction={decorate}}}
\begin{tikzpicture}[scale=1.8]

  \begin{scope}[rotate=45]
    \draw[-latex](1.15,0)--(2.2,0)node[anchor=south west]{$v$};
    \draw(0,0)--(.9,0);
    \draw[-latex](0,0)--(0,2.2)node[anchor=south east]{$u$};
    \draw(0,0)--(1.6,1.6)node[anchor=south]{$u=v$};
    \draw(0,0)--(2.1,1.05)node[anchor=south]{$u=av$};
    \draw[fill=gray!40](0,0)--(2,1)--(1,1)--cycle;
    \draw[-latex](0.7,0.7)--(.9,0.7);
    \draw(0.7,0.7)--(1,0.7);
    \draw[-latex](0.7,0.35)--(0.7,0.58);
    \draw(0.7,0.35)--(0.7,0.7);
    \node at (1,0.7)[anchor=south]{\small $(u,v)$};
    \node at (0.7,0.7)[anchor=east]{\small$(u,u)$};
    \node at (0.7,.35)[anchor=west]{\small$(au,u)$};
    \fill (1,0.7) circle[radius=.7pt];
    \fill (.7,0.7) circle[radius=.7pt];
    \fill (0.7,.35)circle[radius=.7pt];
  \end{scope}

  \begin{scope}[xshift=4cm,rotate=45]
    \draw[-latex](0,0)--(2.2,0)node[anchor=south west]{$v$};
    \draw[-latex](0,0)--(0,2.2)node[anchor=south east]{$u$};
    \draw(0,0)--(1.6,1.6)node[anchor=south]{$u=v$};
    \draw(0,0)--(2.1,1.05)node[anchor=south]{$u=av$};
    \draw[fill=gray!40](0,0)--(2,1)--(1,1)--cycle;
    \draw[-latex](1,.5)--(1,.65);
    \draw(.5,.5)--(1,.5)--(1,0.7);
    \draw[-latex](.5,.5)--(.75,.5);
    \node at (1,0.7)[anchor=south]{\small $(u,v)$};
    \node at (1,.5)[anchor=west]{\small$(av,v)$};
    \node at (.5,.5)[anchor=east]{\small$(av,av)$};
    \fill (1,0.7) circle[radius=.7pt];
    \fill (1,0.5) circle[radius=.7pt];
    \fill (.5,.5) circle[radius=.7pt];
  \end{scope}

\end{tikzpicture}

  \caption{Integration path for $\pp{x}{u}(u,v)$ on the left, as in equation \eqref{eq:96}. Integration path for $\pp{x}{v}(u,v)$ on the right, as in equation \eqref{eq:99}.
  \label{integration_paths}}
\end{figure}

Taking the derivative of the first of \eqref{eq:27} with respect to $u$ and of the second of \eqref{eq:34} with respect to $v$ and subtracting the resulting expressions, we obtain
\begin{align}
  \label{eq:100}
  \pppp{x}{u}{v}=\frac{1}{\cout-\cin}\left(\frac{\cout}{\cin}\pp{\cin}{v}\pp{x}{u}-\frac{\cin}{\cout}\pp{\cout}{u}\pp{x}{v}\right).
\end{align}
Defining
\begin{align}
  \label{eq:101}
  \mu\defeq \frac{1}{\cout-\cin}\,\frac{\cout}{\cin}\pp{\cin}{v},\qquad \nu\defeq-\frac{1}{\cout-\cin}\,\frac{\cin}{\cout}\pp{\cout}{u},
\end{align}
this is
\begin{align}
  \label{eq:102}
  \pppp{x}{u}{v}(u,v)=\mu(u,v)\pp{x}{u}(u,v)+\nu(u,v)\pp{x}{v}(u,v).
\end{align}

\subsection{Formulation of the Problem\label{shock_interaction_problem}}
The shock interaction problem is the following: Find a solution of the equations
\begin{align}
  \label{eq:103}
  \pp{x}{u}&=\cin\pp{t}{u},\\
  \label{eq:104}
  \pp{x}{v}&=\cout\pp{t}{v},\\
  \label{eq:105}
  \pp{\alpha}{v}&=0,\\
  \label{eq:106}
  \pp{\beta}{u}&=0
\end{align}
in $T_\varepsilon$, such that the boundary conditions along the shocks
\begin{align}
  \label{eq:107}
  \frac{d\ls{x}_+}{du}(u)&=\ls{V}(u)\frac{d\ls{t}_+}{du}(u),\\
  \label{eq:108}
  \frac{d\rs{x}_+}{dv}(v)&=\rs{V}(v)\frac{d\rs{t}_+}{dv}(v)
\end{align}
are satisfied. Here
\begin{alignat}{3}
  \label{eq:109}
  \ls{x}_+(u)&=x(u,u),&\qquad \ls{t}_+(u)&=t(u,u),\\
  \rs{x}_+(v)&=x(av,v),& \rs{t}_+(v)&=t(av,v)\label{eq:110}
\end{alignat}
and
\begin{align}
  \label{eq:111}
  \is{V}=\frac{\jump{\is{\rho}\is{w}}}{\jump{\is{\rho}}},\qquad i=1,2,
\end{align}
where
\begin{alignat}{3}
  \label{eq:112}
  \ls{\rho}_\pm(u)&=\rho(\ls{\alpha}_\pm(u),\ls{\beta}_\pm(u)),&\qquad \ls{w}_\pm(u)&=w(\ls{\alpha}_\pm(u),\ls{\beta}_\pm(u)),\\
    \rs{\rho}_\pm(v)&=\rho(\rs{\alpha}_\pm(v),\rs{\beta}_\pm(v)),\qquad &\ls{w}_\pm(v)&=w(\ls{\alpha}_\pm(v),\ls{\beta}_\pm(v)),\label{eq:113}
  \end{alignat}
where $\rho(\alpha,\beta)$, $w(\alpha,\beta)$ are given smooth functions of their arguments and
\begin{alignat}{3}
  \label{eq:114}
  \ls{\alpha}_+(u)&=\alpha(u,u),\qquad&\ls{\beta}_+(u)&=\beta(u,u),\\
  \label{eq:115}
  \ls{\alpha}_-(u)&=\ls{\alpha}^\ast(\ls{t}_+(u),\ls{x}_+(u)),\qquad&\ls{\beta}_-(u)&=\ls{\beta}^\ast(\ls{t}_+(u),\ls{x}_+(u))
\end{alignat}
and
\begin{alignat}{3}
  \label{eq:116}
  \rs{\alpha}_+(v)&=\alpha(av,v),&\qquad\rs{\beta}_+(v)&=\beta(av,v),\\
  \label{eq:117}
  \rs{\alpha}_-(v)&=\rs{\alpha}^\ast(\rs{t}_+(v),\rs{x}_+(v)),&\rs{\beta}_-(v)&=\rs{\beta}^\ast(\rs{t}_+(v),\rs{x}_+(v)).
\end{alignat}
Furthermore, the jump conditions
\begin{align}
  \label{eq:118}
  J(\is{\alpha}_+,\is{\beta}_+,\is{\alpha}_-,\is{\beta}_-)=0,\qquad i=1,2
\end{align}
are satisfied and at the point of interaction $(u,v)=(0,0)$ we have
\begin{align}
  \label{eq:119}
  \alpha(0,0)&=\beta(0,0)=\beta_0,\qquad \pp{\alpha}{u}(0,0)=\alpha_0',\qquad\pp{\beta}{v}(0,0)=\beta_0',\\
  \is{V}(0)&=\is{V}_0,\qquad i=1,2.\label{eq:120}
\end{align}
In addition, the determinism conditions have to be satisfied:
\begin{align}
  \label{eq:121}
  \ls{\cin}_+(u)&<\ls{V}(u)<\ls{\cin}_-(u),\\
  \rs{\cout}_-(v)&<\rs{V}(v)<\rs{\cout}_+(v),\label{eq:122}
\end{align}
where
\begin{align}
  \label{eq:123}
  \ls{\cin}_\pm(u)=\cin(\ls{\alpha}_\pm(u),\ls{\beta}_\pm(u)),\qquad \rs{\cout}_\pm(v)=\cout(\rs{\alpha}_\pm(v),\rs{\beta}_\pm(v)).
\end{align}

\section{Solution of the Interaction Problem}
We have the following theorem:
\begin{theorem}[Existence]\label{existence_theorem}
  For $\varepsilon$ sufficiently small, the shock interaction problem as stated in the previous subsection possesses a solution $\alpha$, $\beta$, $t$, $x$ in $C^2(T_\varepsilon)$. Also, for $\varepsilon$ sufficiently small, the Jacobian $\frac{\partial(t,x)}{\partial(u,v)}$ does not vanish in $T_\varepsilon$, which implies that $\alpha$, $\beta$ are $C^2$ functions of $(t,x)$ on the image of $T_\varepsilon$ by the map $(u,v)\mapsto (t(u,v),x(u,v))$.
\end{theorem}
The proof of this theorem is based on an iteration scheme set up and employed in the following subsections.

\subsection{Setup of the Iteration Scheme}
We solve the interaction problem using an iteration. We base our iteration scheme on the functions $x(u,v)$, $\ls{\alpha}_+(u)$ and $\rs{\beta}_+(v)$. The iteration scheme is as follows:
\begin{enumerate}
\item We initiate the sequence by\footnote{The index $0$ on the left of these equations denotes the $0$'th iterate in our sequence of functions. This is the only place where the index $0$ is used in this way, everywhere else it is used to denote evaluation of quantities at the point of interaction (which corresponds to the origin of our coordinates), hence no confusion should arise.}
  \begin{align}
    x_0(u,v)=\frac{1}{\Gamma_0}u+v,\qquad \ls{\alpha}_{+0}(u)=\underbrace{\alpha_0}_{=\beta_0}+\alpha_0'u,\qquad \rs{\beta}_{+0}(v)=\beta_0+\beta_0'v.\label{eq:124}
  \end{align}
\item We start with approximate solutions
  \begin{align}
    \label{eq:125}
    x_n\in C^2(T_\varepsilon),\qquad \ls{\alpha}_{+n}\in C^2[0,a\varepsilon],\qquad \rs{\beta}_{+n}\in C^2[0,\varepsilon].
  \end{align}
\item We set
  \begin{align}
    \alpha_n(u,v)=\ls{\alpha}_{+n}(u),\qquad \beta_n(u,v)=\rs{\beta}_{+n}(v).\label{eq:126}
  \end{align}
\item We set
  \begin{align}
    t_n(u,v)=\int_0^u\left(\phi_n+\psi_n\right)(u',u')du'+\int_u^v\psi_n(u,v')dv',\label{eq:127}
  \end{align}
  where we set
\begin{align}
  \label{eq:128}
  \phi_n(u,v)&=\frac{1}{\cin_n(u,v)}\pp{x_n}{u}(u,v),\\
    \label{eq:129}
  \psi_n(u,v)&=\frac{1}{\cout_n(u,v)}\pp{x_n}{v}(u,v).
\end{align}
Here we use the notation
\begin{align}
  \label{eq:130}
  \cin_n(u,v)=\cin(\alpha_n(u,v),\beta_n(u,v)),\qquad\cout_n(u,v)=\cout(\alpha_n(u,v),\beta_n(u,v)).
\end{align}
\item We then set
  \begin{alignat}{2}
    \label{eq:131}
    \ls{t}_{+n}(u)&=t_n(u,u),&\qquad \ls{x}_{+n}(u)&=x_n(u,u),\\
    \rs{t}_{+n}(v)&=t_n(av,v),&\qquad \rs{x}_{+n}(v)&=x_n(av,v)\label{eq:132}
  \end{alignat}
and compute the quantities appearing in the jumps evaluated in the states ahead:
\begin{alignat}{3}
  \label{eq:133}
  \ls{\alpha}_{-n}(u)&=\ls{\alpha}^\ast(\ls{t}_{+n}(u),\ls{x}_{+n}(u)),\qquad &\ls{\beta}_{-n}(u)&=\ls{\beta}^\ast(\ls{t}_{+n}(u),\ls{x}_{+n}(u)),\\
  \rs{\alpha}_{-n}(v)&=\rs{\alpha}^\ast(\rs{t}_{+n}(v),\rs{x}_{+n}(v)),\qquad &\rs{\beta}_{-n}(v)&=\rs{\beta}^\ast(\rs{t}_{+n}(v),\rs{x}_{+n}(v)).\label{eq:134}
\end{alignat}
\item The quantities in the state behind are given by
  \begin{align}
    \label{eq:135}
    \ls{\beta}_{+n}(u)&=\beta_n(u,u)=\rs{\beta}_{+n}(u),\\
    \rs{\alpha}_{+n}(v)&=\alpha_n(av,v)=\ls{\alpha}_{+n}(av).\label{eq:136}
  \end{align}
Hence all quantities appearing in the jumps are given and we compute $\ls{V}_n$, $\rs{V}_n$ using
\begin{align}
  \label{eq:137}
  \is{V}_n=\frac{\jump{\is{\rho}_n\is{w}_n}}{\jump{\is{\rho}_n}},\qquad i=1,2,
\end{align}
where $\jump{f_n}=f_{+n}-f_{-n}$ and
\begin{align}
  \label{eq:138}
  \is{\rho}_{\pm n}=\rho(\is{\alpha}_{\pm n},\is{\beta}_{\pm n}),\qquad   \is{w}_{\pm n}=w(\is{\alpha}_{\pm n},\is{\beta}_{\pm n}),\qquad i=1,2.
\end{align}
From \eqref{eq:137} we compute
\begin{align}
  \label{eq:139}
  \ls{\Gamma}_n=\frac{\ls{\cout}_{+n}}{\ls{\cin}_{+n}}\,\frac{\ls{V}_n-\ls{\cin}_{+n}}{\ls{\cout}_{+n}-\ls{V}_n},\qquad \rs{\Gamma}_n=a\frac{\rs{\cout}_{+n}}{\rs{\cin}_{+n}}\,\frac{\rs{V}_n-\rs{\cin}_{+n}}{\rs{\cout}_{+n}-\rs{V}_n},
\end{align}
where
\begin{align}
  \label{eq:140}
  \is{\cin}_{+n}=\cin(\is{\alpha}_{+n},\is{\beta}_{+n}),\qquad  \is{\cout}_{+n}=\cout(\is{\alpha}_{+n},\is{\beta}_{+n}),\qquad i=1,2.
\end{align}
\item We set
\begin{align}
  \label{eq:141}
  x_{n+1}(u,v)&=\frac{1}{\Gamma_0}u+v+\int_0^u\ls{\Phi}_n(u')du'+\int_0^v\rs{\Phi}_n(v')dv'\nonumber\\
              &\quad+\int_0^u\int_{av'}^{v'}M_n(u',v')du'dv'+\int_u^v\int_{av'}^uM_n(u',v')du'dv'.
\end{align}
where
\begin{align}
  \label{eq:142}
  \ls{\Phi}_n(u)&=\int_0^{u}\ls{\Lambda}_n(u')du'+\frac{1}{\ls{\Gamma}_n(u)}\int_{au}^{u}M_n(u',u)du',\\
  \rs{\Phi}_n(v)&=\int_0^{v}\rs{\Lambda}_n(v')dv'+\rs{\Gamma}_n(v)\int_{av}^{v}M_n(av,v')dv'.  \label{eq:143}
\end{align}
where we set
\begin{align}
  M_n(u,v)&=\mu_n(u,v)\pp{x_n}{u}(u,v)+\nu_n(u,v)\pp{x_n}{v}(u,v),  \label{eq:144}
\end{align}
where
\begin{align}
    \label{eq:145}
  \mu_n&=\frac{1}{\cout_n-\cin_n}\frac{\cout_n}{\cin_n}\pp{\cin_n}{v},\qquad \nu_n=-\frac{1}{\cout_n-\cin_n}\frac{\cin_n}{\cout_n}\pp{\cout_n}{u}\\\intertext{and}
  \ls{\Lambda}_n(u)&=\frac{d\ls{\gamma}_n}{du}(u)\pp{x_n}{u}(au,u)+\ls{\gamma}_n(u)a\ppp{x_n}{u}(au,u)+\ls{\gamma}_n(u)M_n(au,u),\label{eq:146}\\
  \label{eq:147}
  \rs{\Lambda}_n(v)&=\frac{d\rs{\gamma}_n}{dv}(v)\pp{x_n}{v}(av,av)+\rs{\gamma}_n(v)a\ppp{x_n}{v}(av,av)+\rs{\gamma}_n(v)aM_n(av,av)
\end{align}
and we used the definitions
\begin{align}
  \label{eq:148}
  \ls{\gamma}_n(u)\defeq\frac{\rs{\Gamma}_n(u)}{\ls{\Gamma}_n(u)},\qquad \rs{\gamma}_n(v)\defeq\frac{\rs{\Gamma}_n(v)}{\ls{\Gamma}_n(av)}.
\end{align}
In \eqref{eq:145}, we use the notation
\begin{align}
  \label{eq:149}
  \cout_n(u,v)=\cout(\alpha_n(u,v),\beta_n(u,v)),\qquad\cin_n(u,v)=\cin(\alpha_n(u,v),\beta_n(u,v)).
\end{align}

\item We compute $\ls{\alpha}_{+,n+1}(u)$, $\rs{\beta}_{+,n+1}(v)$ as
  \begin{align}
    \label{eq:150}
    \ls{\alpha}_{+,n+1}(u)&=\ls{H}(\ls{\beta}_{+n}(u),\ls{\alpha}_{-n}(u),\ls{\beta}_{-n}(u)),\\
    \rs{\beta}_{+,n+1}(v)&=\rs{H}(\rs{\alpha}_{+n}(v),\rs{\alpha}_{-n}(v),\rs{\beta}_{-n}(v)).\label{eq:151}
  \end{align}
\end{enumerate}

\subsection{Remarks on the Iteration Scheme}
\begin{enumerate}
\item The way we set things up, to each triple of functions $\ls{\alpha}_{+n}$, $\rs{\beta}_{+n}$, $x_n$ there corresponds a unique triple of functions $t_n$, $\alpha_n$, $\beta_n$ (see \eqref{eq:126}, \eqref{eq:127}). It therefore suffices to show that the iteration mapping maps the respective spaces to itself (by induction) and the convergence only for the sequence $((\ls{\alpha}_{+n},\rs{\beta}_{+n},x_n);n=0,1,2,\ldots)$.
\item From \eqref{eq:126} we see that for every member of the sequence the equations \eqref{eq:105}, \eqref{eq:106} are satisfied. I.e.
\begin{align}
  \label{eq:152}
  \pp{\alpha_n}{v}=0,\qquad \pp{\beta_n}{u}=0.
\end{align}
\item It will be shown in section \ref{inductive_step} that taking the partial derivative of \eqref{eq:141} with respect to $u$ yields an equation of the form
\begin{align}
  \label{eq:153}
  \ppp{x_{n+1}}{u}(u,v)=\ls{\gamma}_n(u)a\ppp{x_n}{u}(au,u)+\ldots
\end{align}
The terms not written out are either of mixed derivative type, of lower order, or mixed type third derivatives but appearing in an integral. Equation \eqref{eq:153} corresponds to \eqref{eq:96} once a partial derivative with respect to $u$ is taken and the $n+1$'th iterate is placed in the left hand side and the $n$'th iterate is placed in the right hand side. Since $0<a<1$ and $\ls{\gamma}_n(0)=1$, the overall prefactor in front of $\ppp{x_n}{u}(au,u)$ can be made strictly smaller than one, which is crucial in establishing that the sequence of iterates is suitably confined and that it is convergent.
\item In step 5 of our iteration scheme we have to make sure that the shock curves in space time corresponding to the $n$'th iterate:
  \begin{align}
    \label{eq:154}
    u&\mapsto (\ls{t}_{+n}(u),\ls{x}_{+n}(u)),\qquad u\in[0,a\varepsilon]\\
    v&\mapsto (\rs{t}_{+n}(v),\rs{x}_{+n}(v)),\qquad v\in[0,\varepsilon]\label{eq:155}
  \end{align}
are in the future development of the respective data sets (see figure \ref{future_development} on page \pageref{future_development}). Otherwise the evaluation of the quantities $\is{\alpha}^\ast(\is{t}_{+n},\is{x}_{+n})$, $\is{\beta}^\ast(\is{t}_{+n},\is{x}_{+n})$, $i=1,2$, would be meaningless.
\end{enumerate}

\subsection{Inductive Step\label{inductive_step}}
We choose closed balls in function spaces as follows:
\begin{align}
    \label{eq:156}
  B_{N_0}&=\Big\{f\in C^2(T_\varepsilon):f(0,0)=0,\pp{f}{u}(0,0)=\frac{1}{\Gamma_0},\pp{f}{v}(0,0)=1,\|f\|_0\leq N_0\Big\},\\
    B_{A}&=\Big\{f\in C^2[0,a\varepsilon]:f(0)=\beta_0,f'(0)=\alpha_0',\|f\|_1\leq A\Big\},\label{eq:157}\\
  \label{eq:158}
  B_{B}&=\Big\{f\in C^2[0,\varepsilon]:f(0)=\beta_0,f'(0)=\beta_0',\|f\|_2\leq B\Big\},
\end{align}
where we use
\begin{align}
  \label{eq:159}
  \|f\|_0&\defeq\max\left\{\sup_{T_\varepsilon}\left|\ppp{f}{u}\right|,\sup_{T_\varepsilon}\left|\pppp{f}{u}{v}\right|,\sup_{T_\varepsilon}\left|\ppp{f}{v}\right|\right\},\\
  \label{eq:160}
  \|f\|_1&\defeq\sup_{[0,a\varepsilon]}|f''|,\\
  \label{eq:161}
  \|f\|_2&\defeq\sup_{[0,\varepsilon]}|f''|.
\end{align}
For the constants $\alpha_0'$, $\beta_0'$ see \eqref{eq:93}.
\begin{proposition}
  Choosing the constants $A$, $B$, $N_0$ appropriately, the sequence
  \begin{align}
    \label{eq:162}
    (\ls{\alpha}_{+n},\rs{\beta}_{+n},x_n);n=0,1,2,\ldots)   
  \end{align}
is contained in $B_A\times B_B\times B_{N_0}$, provided we choose $\varepsilon$ sufficiently small.
\end{proposition}
\begin{proof}
Since we initiate the sequence by (see \eqref{eq:124})
\begin{align}
    x_0(u,v)=\frac{1}{\Gamma_0}u+v,\qquad \ls{\alpha}_{+0}(v)=\beta_0+\alpha_0'v,\qquad \rs{\beta}_{+0}(v)=\beta_0+\beta_0'v,\label{eq:163}
\end{align}
we see that
\begin{align}
  \label{eq:164}
  (\ls{\alpha}_{+0},\rs{\beta}_{+0},x_0)\in B_A\times B_B\times B_{N_0}.
\end{align}

Let now
\begin{align}
  \label{eq:165}
  (\ls{\alpha}_{+n},\rs{\beta}_{+n},x_n)\in B_A\times B_B\times B_{N_0}.
\end{align}
We have to show that from this it follows that
\begin{align}
  \label{eq:166}
  (\ls{\alpha}_{+,n+1},\rs{\beta}_{+,n+1},x_{n+1})\in B_A\times B_B\times B_{N_0}.
\end{align}
We first derive estimates for $\alpha_n(u,v)$ and $\beta_n(u,v)$ and derivatives thereof. The inductive hypothesis for $\ls{\alpha}_{+n}$ is
\begin{align}
  \label{eq:167}
  \ls{\alpha}_{+n}(0)=\beta_0,\qquad \frac{d\ls{\alpha}_{+n}}{du}(0)=\alpha_0',\qquad \sup_{[0,a\varepsilon]}\left|\frac{d^2\ls{\alpha}_{+n}}{du^2}\right|\leq A.
\end{align}
From this we obtain
\begin{align}
  \label{eq:168}
  \bigg|\frac{d\ls{\alpha}_{+n}}{du}(u)-\alpha_0'\bigg|\leq \left|\int_0^u\frac{d^2\ls{\alpha}_{+n}}{du^2}(u')du'\right|\leq Au.
\end{align}
In the following we use the notation $f(u)=\Landau_d(u^m)$ to denote
\begin{align}
  \label{eq:169}
  |f(u)|\leq C(d)u^m,
\end{align}
where $C(d)$ is a non-decreasing, continuous function of $d$. Using this, \eqref{eq:168} implies
\begin{align}
  \label{eq:170}
  \frac{d\ls{\alpha}_{+n}}{du}(u)=\alpha_0'+\Landau_A(u).
\end{align}
Similarly, we obtain
\begin{align}
  \label{eq:171}
  \frac{d\rs{\beta}_{+n}}{dv}(v)=\beta_0'+\Landau_B(v).
\end{align}
From these we obtain
\begin{align}
  \label{eq:172}
  \ls{\alpha}_{+n}(u)&=\beta_0+\alpha_0'u+\Landau_A(u^2),\\
  \rs{\beta}_{+n}(v)&=\beta_0+\beta_0'v+\Landau_B(v^2).\label{eq:173}
\end{align}
Now, for any function $f(u)=\Landau_d(u^m)$ we have
\begin{align}
  \label{eq:174}
  |f(u)|\leq C(d)u^m\leq Cu^{m-1},
\end{align}
provided we choose $\varepsilon$ sufficiently small depending on $d$, where on the right we have a fixed numerical constant. I.e.~$f(u)=\Landau(u^{m-1})$ (without an index on the Landau symbol). Hence, from \eqref{eq:172}, \eqref{eq:173} we have
\begin{align}
  \label{eq:175}
  \ls{\alpha}_{+n}(u)&=\beta_0+\Landau(u),\\
  \rs{\beta}_{+n}(v)&=\beta_0+\Landau(v),\label{eq:176}
\end{align}
provided we choose $\varepsilon$ sufficiently small. In the following, we are going to use smallness conditions on $\varepsilon$ of this type without mentioning it any further. Since (see \eqref{eq:126})
\begin{align}
  \label{eq:177}
      \alpha_n(u,v)=\ls{\alpha}_{+n}(u),\qquad \beta_n(u,v)=\rs{\beta}_{+n}(v).
\end{align}
we have
\begin{align}
  \label{eq:178}
  \alpha_n(u,v)&=\beta_0+\alpha_0'u+\Landau_A(u^2)=\beta_0+\Landau(v),\\
  \beta_n(u,v)&=\beta_0+\beta_0'v+\Landau_B(v^2)=\beta_0+\Landau(v),\label{eq:179}
\end{align}
where for the second equality in \eqref{eq:178} we made use of $u\leq v$ in $T_\varepsilon$. Taking partial derivatives of \eqref{eq:177} and using \eqref{eq:170}, \eqref{eq:171} we obtain
\begin{align}
  \label{eq:180}
    \pp{\alpha_n}{u}(u,v)=\alpha_0'+\Landau_A(v),\qquad \pp{\beta_n}{v}(u,v)=\beta_0'+\Landau_B(v).
\end{align}
The other first derivatives of $\alpha_n$, $\beta_n$ vanish. In particular, first order derivatives of $\alpha_n$, $\beta_n$ are bounded by a fixed constant, provided we choose $\varepsilon$ sufficiently small. Taking second derivatives of \eqref{eq:177}, using the third of \eqref{eq:167} and the analogous induction hypothesis for $\rs{\beta}_{+n}$, we obtain
\begin{align}
  \label{eq:181}
    \left|\ppp{\alpha_n}{u}(u,v)\right|\leq A,\qquad \left|\ppp{\beta_n}{v}(u,v)\right|\leq B.
\end{align}
All other second derivatives of $\alpha_n$, $\beta_n$ vanish.

Since
\begin{align}
  \label{eq:182}
  \rs{\alpha}_{+n}(v)=\alpha_n(av,v)=\ls{\alpha}_{+n}(av),
\end{align}
we have, using \eqref{eq:167}, \eqref{eq:170}, \eqref{eq:172},
\begin{align}
  \label{eq:183}
  \rs{\alpha}_{+n}(v)=\beta_0+a\alpha_0'v+\Landau_A(v^2),\qquad \frac{d\rs{\alpha}_{+n}}{dv}(v)=a\alpha_0'+\Landau_A(v),\qquad\left|\frac{d^2\rs{\alpha}_{+n}}{dv^2}(v)\right|\leq a^2A.
\end{align}
Since
\begin{align}
  \label{eq:184}
  \ls{\beta}_{+n}(u)=\beta_n(u,u)=\rs{\beta}_{+n}(u),
\end{align}
we have, using the induction hypothesis for $\rs{\beta}_{+n}(v)$ and \eqref{eq:171}, \eqref{eq:173},
\begin{align}
  \label{eq:185}
  \ls{\beta}_{+n}(u)=\beta_0+\beta_0'u+\Landau_B(u^2),\qquad \frac{d\ls{\beta}_{+n}}{du}(u)=\beta_0'+\Landau_B(u),\qquad\left|\frac{d^2\ls{\beta}_{+n}}{du^2}(u)\right|\leq B.
\end{align}

We derive properties of $x_n(u,v)$. The inductive hypothesis for $x_n$ is
\begin{align}
  \label{eq:186}
  x_n(0,0)=0,\qquad \pp{x_n}{u}(0,0)=\frac{1}{\Gamma_0},\qquad\pp{x_n}{v}(0,0)&=1,\\
    \label{eq:187}
  \sup_{T_\varepsilon}\left|\ppp{x_n}{u}\right|,\sup_{T_\varepsilon}\left|\pppp{x_n}{u}{v}\right|,\sup_{T_\varepsilon}\left|\ppp{x_n}{v}\right|&\leq N_0.
\end{align}
$x_n(u,v)$ satisfies
\begin{align}
  \label{eq:188}
  \pp{x_n}{u}(u,v)&=\pp{x_n}{u}(0,0)+\int_0^u\left(\ppp{x_n}{u}+\pppp{x_n}{u}{v}\right)(u',u')du'+\int_u^v\pppp{x_n}{u}{v}(u,v')dv',\\
  \label{eq:189}
  \pp{x_n}{v}(u,v)&=\pp{x_n}{v}(0,0)+\int_0^u\left(\pppp{x_n}{u}{v}+\ppp{x_n}{v}\right)(u',u')du'+\int_u^v\ppp{x_n}{v}(u,v')dv'.
\end{align}
These equations correspond to integrating the second derivative of $x_n(u,v)$ from the origin $(0,0)$ along $u=v$ until $(u,u)$ and then integrating along a right moving  characteristic until $(u,v)$. Using the inductive hypothesis for $x_n$ we obtain
\begin{align}
  \label{eq:190}
  \pp{x_n}{u}(u,v)&=\frac{1}{\Gamma_0}+\Landau_{N_0}(v),\\
  \label{eq:191}
  \pp{x_n}{v}(u,v)&=1+\Landau_{N_0}(v).
\end{align}
Integrating the first derivative of $x_n(u,v)$ from the origin $(0,0)$ along $u=v$ until $(u,u)$ and then integrating along a right moving characteristic until $(u,v)$ yields
\begin{align}
  \label{eq:192}
  x_n(u,v)=\int_0^u\left(\pp{x_n}{u}+\pp{x_n}{v}\right)(u',u')du'+\int_u^v\pp{x_n}{v}(u,v')dv'.
\end{align}
Using \eqref{eq:190}, \eqref{eq:191} we have
\begin{align}
  \label{eq:193}
  x_n(u,v)=\frac{1}{\Gamma_0}u+v+\Landau_{N_0}(v^2)=\Landau(v).
\end{align}

We derive properties of $t_n(u,v)$. We have (see \eqref{eq:127})
\begin{align}
  \label{eq:194}
      t_n(u,v)=\int_0^u\left(\phi_n+\psi_n\right)(u',u')du'+\int_u^v\psi_n(u,v')dv',
\end{align}
where
\begin{align}
  \label{eq:195}
  \phi_n=G(\alpha_n,\beta_n)\pp{x_n}{u},\qquad\psi_n=H(\alpha_n,\beta_n)\pp{x_n}{v},
\end{align}
and
\begin{align}
  \label{eq:196}
  G(\alpha,\beta)\defeq \frac{1}{\cin(\alpha,\beta)},\qquad H(\alpha,\beta)\defeq\frac{1}{\cout(\alpha,\beta)}.
\end{align}
We are going to use the notation
\begin{align}
  \label{eq:197}
  g_n(u,v)=G(\alpha_n(u,v),\beta_n(u,v)),\qquad h_n(u,v)=H(\alpha_n(u,v),\beta_n(u,v)).
\end{align}
$\cin(\alpha,\beta)$, $\cout(\alpha,\beta)$ are smooth functions of $\alpha$, $\beta$.  In view of \eqref{eq:178}, \eqref{eq:179} provided we choose $\varepsilon$ sufficiently small, so are the functions $G(\alpha,\beta)$, $H(\alpha,\beta)$. Therefore,
\begin{align}
  \label{eq:198}
  g_n(u,v)&=g_0+\left(\pp{G}{\alpha}\right)_0\alpha_0'u+\left(\pp{G}{\beta}\right)_0\beta_0'v+\Landau_A(u^2)+\Landau_B(v^2)+\Landau(uv)\nonumber\\
  &=-\frac{1}{\eta_0}+\Landau(v),\\
  \label{eq:199}
  h_n(u,v)&=h_0+\left(\pp{H}{\alpha}\right)_0\alpha_0'u+\left(\pp{H}{\beta}\right)_0\beta_0'v+\Landau_A(u^2)+\Landau_B(v^2)+\Landau(uv)\nonumber\\
  &=\frac{1}{\eta_0}+\Landau(v),
\end{align}
where we used (recall that $\alpha_0=\beta_0$)
\begin{align}
  \label{eq:200}
  g_0=\frac{1}{\cin_0}=\frac{1}{\cin(\beta_0,\beta_0)}=-\frac{1}{\eta_0},\qquad h_0=\frac{1}{\cout_0}=\frac{1}{\cout(\beta_0,\beta_0)}=\frac{1}{\eta_0}.
\end{align}
We have
\begin{align}
  \label{eq:201}
  \pp{h_n}{u}=\pp{H}{\alpha}(\alpha_n,\beta_n)\pp{\alpha_n}{u},\qquad \pp{h_n}{v}=\pp{H}{\beta}(\alpha_n,\beta_n)\pp{\beta_n}{v},
\end{align}
which, through \eqref{eq:180}, implies
\begin{align}
  \label{eq:202}
  \pp{h_n}{u}(u,v)=\left(\pp{H}{\alpha}\right)_0\alpha_0'+\Landau_A(v),\qquad\pp{h_n}{v}(u,v)=\left(\pp{H}{\beta}\right)_0\beta_0'+\Landau_B(v).
\end{align}
We have
\begin{align}
  \label{eq:203}
  \ppp{h_n}{u}&=\ppp{H}{\alpha}(\alpha_n,\beta_n)\left(\pp{\alpha_n}{u}\right)^2+\pp{H}{\alpha}(\alpha_n,\beta_n)\ppp{\alpha_n}{u},\\
  \label{eq:204}
  \pppp{h_n}{u}{v}&=\pppp{H}{\alpha}{\beta}(\alpha_n,\beta_n)\pp{\alpha_n}{u}\pp{\beta_n}{v}.
\end{align}
Therefore, using \eqref{eq:180}, we have
\begin{align}
  \label{eq:205}
  \ppp{h_n}{u}(u,v)=\Landau_A(1),\qquad\pppp{h_n}{u}{v}(u,v)=\Landau(1).
\end{align}
Analogously, we obtain
\begin{align}
  \label{eq:206}
  \ppp{g_n}{u}(u,v)=\Landau_A(1),\qquad\pppp{g_n}{u}{v}(u,v)=\Landau(1).
\end{align}
We now derive properties of $t_n(u,v)$ using \eqref{eq:194}. Using \eqref{eq:198}, \eqref{eq:199} together with \eqref{eq:190}, \eqref{eq:191} we obtain
\begin{align}
  \label{eq:207}
  \phi_n(u,v)=-\frac{1}{\eta_0\Gamma_0}+\Landau_{N_0}(v),\qquad\psi_n(u,v)=\frac{1}{\eta_0}+\Landau_{N_0}(v).
\end{align}
Using these in \eqref{eq:194} and its derivative with respect to $v$ we obtain
\begin{align}
  \label{eq:208}
  t_n(u,v)&=\frac{1}{\eta_0}\left(v-\frac{u}{\Gamma_0}\right)+\Landau_{N_0}(v^2)=\Landau(v),\\
  \label{eq:209}
  \pp{t_n}{v}(u,v)&=\psi_n(u,v)=\frac{1}{\eta_0}+\Landau_{N_0}(v),
\end{align}
respectively. From the derivative of \eqref{eq:194} with respect to $u$ we have
\begin{align}
  \label{eq:210}
  \pp{t_n}{u}(u,v)=\phi_n(u,u)+\int_u^v\pp{\psi_n}{u}(u,v')dv'.
\end{align}
We have
\begin{align}
  \label{eq:211}
  \pp{\psi_n}{u}=\pp{h_n}{u}\pp{x_n}{v}+h_n\pppp{x_n}{u}{v}.
\end{align}
Using \eqref{eq:199} and the first of \eqref{eq:202} together with \eqref{eq:187}, \eqref{eq:191} we obtain
\begin{align}
  \label{eq:212}
  \pp{\psi_n}{u}(u,v)=\left(\pp{H}{\alpha}\right)_0\alpha_0'+\Landau_A(v)+\Landau_{N_0}(1)=\Landau_{N_0}(1).
\end{align}
Using this together with the first of \eqref{eq:207} in \eqref{eq:210} we obtain
\begin{align}
  \label{eq:213}
  \pp{t_n}{u}(u,v)=-\frac{1}{\eta_0\Gamma_0}+\Landau_{N_0}(v).
\end{align}
We turn to estimates for the second partial derivatives of $t(u,v)$. Since (see \eqref{eq:210})
\begin{align}
  \label{eq:214}
  \pppp{t_n}{u}{v}(u,v)=\pp{\psi_n}{u}(u,v),
\end{align}
we have from \eqref{eq:212}
\begin{align}
  \label{eq:215}
  \pppp{t_n}{u}{v}(u,v)=\Landau_{N_0}(1).
\end{align}
Analogous to the treatment of $\pp{\psi_n}{u}(u,v)$, we obtain
\begin{align}
  \label{eq:216}
  \pp{\psi_n}{v}(u,v)=\Landau_{N_0}(1),\qquad  \pp{\phi_n}{u}(u,v)=\Landau_{N_0}(1),\qquad\pp{\phi_n}{v}(u,v)=\Landau_{N_0}(1),
\end{align}
the first of which implies
\begin{align}
  \label{eq:217}
  \ppp{t_n}{v}(u,v)=\Landau_{N_0}(1).
\end{align}
From \eqref{eq:210} we have
\begin{align}
  \label{eq:218}
  \ppp{t_n}{u}(u,v)=\pp{\phi_n}{u}(u,u)+\pp{\phi_n}{v}(u,u)+\pp{}{u}\left(\int_u^v\pp{\psi_n}{u}(u,v')dv'\right).
\end{align}
For the last term in \eqref{eq:218} we use
\begin{align}
  \label{eq:219}
  \int_u^v\pp{\psi_n}{u}(u,v')dv'&=\int_u^v\left(\pp{h_n}{u}\pp{x_n}{v}+h_n\pppp{x_n}{u}{v}\right)(u,v')dv'\nonumber\\
                               &=\int_u^v\left(\pp{h_n}{u}\pp{x_n}{v}\right)(u,v')dv'\nonumber\\
  &\quad+\left(h_n\pp{x_n}{u}\right)(u,v)-\left(h_n\pp{x_n}{u}\right)(u,u)-\int_u^v\left(\pp{h_n}{v}\pp{x_n}{u}\right)(u,v')dv',
\end{align}
where we integrated by parts for the second term in the bracket in the first line. Using this, the last term in \eqref{eq:218} becomes
\begin{align}
  \label{eq:220}
  \pp{}{u}\left(\int_u^v\pp{\psi_n}{u}(u,v')dv'\right)&=\int_u^v\left(\ppp{h_n}{u}\pp{x_n}{v}+\pp{h_n}{u}\pppp{x_n}{u}{v}\right)(u,v')dv'-\left(\pp{h_n}{u}\pp{x_n}{v}\right)(u,u)\nonumber\\
                                                    &\qquad+\left(\pp{h_n}{u}\pp{x_n}{u}\right)(u,v)+\left(h_n\ppp{x_n}{u}\right)(u,v)\nonumber\\
                                                    &\qquad-\left(\pp{h_n}{u}\pp{x_n}{u}\right)(u,u)-\left(h_n\ppp{x_n}{u}\right)(u,u)\nonumber\\
                                                    &\qquad-\left(\pp{h_n}{v}\pp{x_n}{u}\right)(u,u)-\left(h_n\pppp{x_n}{u}{v}\right)(u,u)\nonumber\\
                                                      &\qquad -\int_u^v\left(\pppp{h_n}{u}{v}\pp{x_n}{u}+\pp{h_n}{v}\ppp{x_n}{u}\right)(u,v')dv'\nonumber\\
&\hspace{50mm}  +\left(\pp{h_n}{v}\pp{x_n}{u}\right)(u,u).
\end{align}
(We don't carry out the cancellation in the last and the third to last lines in favor of readability). Using \eqref{eq:187}, \eqref{eq:190}, \eqref{eq:191} together with \eqref{eq:202}, \eqref{eq:205} we obtain
\begin{align}
  \label{eq:221}
  \pp{}{u}\left(\int_u^v\pp{\psi_n}{u}(u,v')dv'\right)=\Landau_{N_0}(1).
\end{align}
Using this in turn in \eqref{eq:218}, together with the expressions for the partial derivatives of $\phi_n$ as given in \eqref{eq:216}, we obtain
\begin{align}
  \label{eq:222}
  \ppp{t_n}{u}(u,v)=\Landau_{N_0}(1).
\end{align}
We summarize the properties of $t_n(u,v)$:
\begin{align}
  \label{eq:223}
  t_n(u,v)&=\frac{1}{\eta_0}\left(v-\frac{u}{\Gamma_0}\right)+\Landau_{N_0}(v^2)=\Landau(v),\\  
  \label{eq:224}
  \pp{t_n}{u}(u,v)&=-\frac{1}{\eta_0\Gamma_0}+\Landau_{N_0}(v),\qquad   \pp{t_n}{v}(u,v)=\frac{1}{\eta_0}+\Landau_{N_0}(v),\\
  \label{eq:225}
  \ppp{t_n}{u}(u,v)&=\Landau_{N_0}(1),\qquad  \pppp{t_n}{u}{v}(u,v)=\Landau_{N_0}(1),\qquad  \ppp{t_n}{v}(u,v)=\Landau_{N_0}(1). 
\end{align}

We now derive properties of $\is{t}_+$, $\is{x}_+$, $i=1,2$. We recall (see \eqref{eq:131}, \eqref{eq:132})
\begin{alignat}{2}
  \label{eq:226}
  \ls{t}_{+n}(u)&=t_n(u,u),&\qquad \ls{x}_{+n}(u)&=x_n(u,u),\\
  \rs{t}_{+n}(v)&=t_n(av,v),&\qquad \rs{x}_{+n}(v)&=x_n(av,v).\label{eq:227}
\end{alignat}
Using \eqref{eq:187}, \eqref{eq:190}, \eqref{eq:191}, \eqref{eq:193} we obtain the following properties of $\ls{x}_{+n}$:
\begin{align}
  \label{eq:228}
  \ls{x}_{+n}(u)&=x_n(u,u)=\left(\frac{1}{\Gamma_0}+1\right)u+\Landau_{N_0}(u^2),\\
  \label{eq:229}
  \frac{d\ls{x}_{+n}}{du}(u)&=\pp{x_n}{u}(u,u)+\pp{x_n}{v}(u,u)=\frac{1}{\Gamma_0}+1+\Landau_{N_0}(u),\\
  \label{eq:230}
  \frac{d^2\ls{x}_{+n}}{du^2}(u)&=\ppp{x_n}{u}(u,u)+2\pppp{x_n}{u}{v}(u,u)+\ppp{x_n}{v}(u,u)=\Landau_{N_0}(1)
\end{align}
and analogously, the following properties of $\rs{x}_{+n}$:
\begin{align}
  \label{eq:231}
  \rs{x}_{+n}(v)&=x_n(av,v)=\left(\frac{a}{\Gamma_0}+1\right)v+\Landau_{N_0}(v^2),\\
  \label{eq:232}
  \frac{d\rs{x}_{+n}}{dv}(v)&=a\pp{x_n}{u}(av,v)+\pp{x_n}{v}(av,v)=\frac{a}{\Gamma_0}+1+\Landau_{N_0}(v),\\
  \label{eq:233}
  \frac{d^2\rs{x}_{+n}}{dv^2}(v)&=a^2\ppp{x_n}{u}(av,v)+2a\pppp{x_n}{u}{v}(av,v)+\ppp{x_n}{v}(av,v)=\Landau_{N_0}(1).
\end{align}
Similarly, but using \eqref{eq:223}, \eqref{eq:224}, \eqref{eq:225} we obtain the following properties of $\ls{t}_{+n}$:
\begin{align}
  \label{eq:234}
  \ls{t}_{+n}(u)&=t_n(u,u)=\frac{1}{\eta_0}\left(1-\frac{1}{\Gamma_0}\right)u+\Landau_{N_0}(u^2),\\
  \label{eq:235}
  \frac{d\ls{t}_{+n}}{du}(u)&=\pp{t_n}{u}(u,u)+\pp{t_n}{v}(u,u)=\frac{1}{\eta_0}\left(1-\frac{1}{\Gamma_0}\right)+\Landau_{N_0}(u).\\
  \label{eq:236}
  \frac{d^2\ls{t}_{+n}}{du^2}(u)&=\ppp{t_n}{u}(u,u)+2\pppp{t_n}{u}{v}(u,u)+\ppp{t_n}{v}(u,u)=\Landau_{N_0}(1)
\end{align}
and analogously, the following properties of $\rs{t}_{+n}$:
\begin{align}
  \label{eq:237}
  \rs{t}_{+n}(v)&=t_n(av,v)=\frac{1}{\eta_0}\left(1-\frac{a}{\Gamma_0}\right)u+\Landau_{N_0}(v^2),\\
  \label{eq:238}
  \frac{d\rs{t}_{+n}}{dv}(v)&=a\pp{t_n}{u}(av,v)+\pp{t_n}{v}(av,v)=\frac{1}{\eta_0}\left(1-\frac{a}{\Gamma_0}\right)+\Landau_{N_0}(v),\\
  \label{eq:239}
  \frac{d^2\rs{t}_{+n}}{dv^2}(v)&=a^2\ppp{t_n}{u}(av,v)+2a\pppp{t_n}{u}{v}(av,v)+\ppp{t_n}{v}(av,v)=\Landau_{N_0}(1).
\end{align}

Let $t\mapsto \ls{x}_{0}^\ast(t)$ describe the left moving characteristic in the state ahead of the left moving shock, starting at the origin (this coincides with the right boundary  $\ls{\mathcal{B}}$ of the future development of the data set given for $x\leq 0$, see figure \ref{future_development} on page \pageref{future_development}). We use the notation
\begin{align}
  \label{eq:240}
  \ls{x}^\ast_{0,n}(u)=\ls{x}^\ast_0(\ls{t}_{+n}(u)).
\end{align}
Here the right hand side is well defined for $\varepsilon$ sufficiently small. We have
\begin{align}
  \label{eq:241}
  \frac{d\ls{x}_{0,n}^\ast}{du}(u)=\frac{d\ls{x}_{0}^\ast}{dt}\Big(\ls{t}_{+n}(u)\Big)\frac{d\ls{t}_{+n}}{du}(u).
\end{align}
Let us denote by $\ls{\cin}^\ast_0(t)$ the characteristic speed of the left moving characteristic starting at the origin, i.e.
\begin{align}
  \label{eq:242}
  \frac{d\ls{x}_0^\ast}{dt}(t)=\ls{\cin}^\ast_0(t).
\end{align}
$\ls{\cin}_0^\ast(t)$ being a smooth function of $t$, we have
\begin{align}
  \label{eq:243}
  \ls{\cin}_0^\ast(t)=(\ls{\cin}_0^\ast)_0+\Landau(t),
\end{align}
therefore (see \eqref{eq:234}),
\begin{align}
  \label{eq:244}
  \ls{\cin}_0^\ast(\ls{t}_{+n}(u))=(\ls{\cin}_0^\ast)_0+\Landau(u).
\end{align}
Using this together with \eqref{eq:235} in \eqref{eq:241} we obtain
\begin{align}
  \label{eq:245}
  \frac{d\ls{x}_{0,n}^\ast}{du}(u)=(\ls{\cin}_0^\ast)_0\frac{1}{\eta_0}\left(1-\frac{1}{\Gamma_0}\right)+\Landau_{N_0}(u).
\end{align}
Using this and \eqref{eq:229} we have (see also \eqref{eq:44})
\begin{align}
  \label{eq:246}
  \frac{d\ls{x}_{+n}}{du}(u)-\frac{d\ls{x}_{0,n}^\ast}{du}(u)=\frac{2}{\ls{V}_0+\eta_0}\Big(\ls{V}_0-(\ls{\cin}_0^\ast)_0\Big)+\Landau_{N_0}(u).
\end{align}
Hence
\begin{align}
  \label{eq:247}
  \ls{x}_{+n}(u)-\ls{x}_{0,n}^\ast(u)=\frac{2}{\ls{V}_0+\eta_0}\Big(\ls{V}_0-(\ls{\cin}_0^\ast)_0\Big)u+\Landau_{N_0}(u^2).
\end{align}
By the determinism condition \eqref{eq:21}, the factor in front of $u$ in the first term on the right is strictly negative. By choosing $\varepsilon$ sufficiently small, the remainder in \eqref{eq:247} in absolute value can be made less or equal to $\left|\frac{2}{\ls{V}_0+\eta_0}\Big(\ls{V}_0-(\ls{\cin}_0^\ast)_0\Big)\right|u$. Hence we obtain
\begin{align}
  \label{eq:248}
  \ls{x}_{+n}(u)-\ls{x}_{0,n}^\ast(u)\leq 0
\end{align}
for $u\in[0,a\varepsilon]$. I.e.~the curve $u\mapsto (\ls{t}_{+n}(u),\ls{x}_{+n}(u))$ lies in the domain of the future development of the data to the left of the origin (see figure \ref{interaction_problem} on page \pageref{interaction_problem}).

Let $t\mapsto\rs{x}_0^\ast(t)$ describe the right moving characteristic in the state ahead of the right moving shock, starting at the origin (this coincides with the left boundary $\rs{\mathcal{B}}$ of the future development of the data given for $x\geq 0$, see figure \ref{future_development} on page \pageref{future_development}). We use the notation
\begin{align}
  \label{eq:249}
  \rs{x}_{0,n}^\ast(v)=\rs{x}^\ast_0(\rs{t}_{+n}(v)).
\end{align}
We have
\begin{align}
  \label{eq:250}
  \frac{d\rs{x}^\ast_{0,n}}{dv}(v)=\frac{d\rs{x}_0^\ast}{dt}\Big(\rs{t}_{+n}(v)\Big)\frac{d\rs{t}_{+n}}{dv}(v).
\end{align}
Let us denote by $\rs{\cin}_0^\ast(t)$ the characteristic speed of the right moving characteristic starting at the origin, i.e.
\begin{align}
  \label{eq:251}
  \frac{d\rs{x}_0^\ast}{dt}(t)=\rs{\cout}^\ast_0(t).
\end{align}
$\rs{\cout}^\ast_0(t)$ being a smooth function of $t$, we have
\begin{align}
  \label{eq:252}
  \rs{\cout}^\ast_0(t)=(\rs{\cout}^\ast_0)_0+\Landau(t),
\end{align}
therefore (see \eqref{eq:237}),
\begin{align}
  \label{eq:253}
  \rs{\cout}^\ast_0(\rs{t}_{+n}(v))=(\rs{\cout}^\ast_0)_0+\Landau(v).
\end{align}
Using this together with \eqref{eq:238} in \eqref{eq:250} we obtain
\begin{align}
  \label{eq:254}
  \frac{d\rs{x}^\ast_{0,n}}{dv}(v)=(\rs{\cout}^\ast_0)_0\frac{1}{\eta_0}\left(1-\frac{a}{\Gamma_0}\right)+\Landau_{N_0}(v).
\end{align}
Using this and \eqref{eq:232} we have (see also \eqref{eq:36} and \eqref{eq:44})
\begin{align}
  \label{eq:255}
  \frac{d\rs{x}_{+n}}{dv}(v)-\frac{d\rs{x}^\ast_{0,n}}{dv}(v)=\frac{2}{\rs{V}_0+\eta_0}\left(\rs{V}_0-(\rs{\cout}^\ast_0)_0\right)+\Landau_{N_0}(v).
\end{align}
Hence
\begin{align}
  \label{eq:256}
  \rs{x}_{+n}(v)-\rs{x}^\ast_{0,n}(v)=\frac{2}{\rs{V}_0+\eta_0}\left(\rs{V}_0-(\rs{\cout}^\ast_0)_0\right)v+\Landau_{N_0}(v^2).
\end{align}
By the determinism condition \eqref{eq:23} and \eqref{eq:25}, the factor in front of $v$ in the first term on the right is strictly positive. By choosing $\varepsilon$ sufficiently small, the remainder in \eqref{eq:256} in absolute value can be made less or equal to $\frac{2}{\rs{V}_0+\eta_0}\left(\rs{V}_0-(\rs{\cout}^\ast_0)_0\right)v$. Hence we obtain
\begin{align}
  \label{eq:257}
  \rs{x}_{+n}(v)-\rs{x}^\ast_{0,n}(v)\geq 0
\end{align}
for $v\in[0,\varepsilon]$. I.e. the curve $v\mapsto(\rs{t}_{+n}(v),\rs{x}_{+n}(v))$ lies in the domain of the future development of the data to the right of the origin (see figure \ref{interaction_problem} on page \pageref{interaction_problem}).

We derive properties of $\is{\alpha}_{-n}$, $\is{\beta}_{-n}$, $i=1,2$. We have
\begin{align}
  \label{eq:258}
  \ls{\alpha}_{-n}(u)=\ls{\alpha}^\ast(\ls{t}_{+n}(u),\ls{x}_{+n}(u)),
\end{align}
where $\ls{\alpha}^\ast(t,x)$ is a smooth function of its arguments, given by the solution in the future development in the state ahead of the left moving shock. The second derivative of $\ls{\alpha}_{-n}(u)$ is given by
\begin{align}
  \label{eq:259}
    \frac{d^2\ls{\alpha}_{-n}}{du^2}(u)&=\pp{\ls{\alpha}^\ast}{t}(\ls{t}_{+n}(u),\ls{x}_{+n}(u))\frac{d^2\ls{t}_{+n}}{du^2}(u)+\pp{\ls{\alpha}^\ast}{x}(\ls{t}_{+n}(u),\ls{x}_{+n}(u))\frac{d^2\ls{x}_{+n}}{du^2}(u)\nonumber\\
                                       &\qquad+\ppp{\ls{\alpha}^\ast}{t}(\ls{t}_{+n}(u),\ls{x}_{+n}(u))\left(\frac{d\ls{t}_{+n}}{du}(u)\right)^2\nonumber\\
  &\qquad+2\pppp{\ls{\alpha}^\ast}{t}{x}(\ls{t}_{+n}(u),\ls{x}_{+n}(u))\frac{d\ls{t}_{+n}}{du}(u)\frac{d\ls{x}_{+n}}{du}(u)\nonumber\\
  &\qquad\qquad\qquad+\ppp{\alpha^\ast}{x}(\ls{t}_{+n}(u),\ls{x}_{+n}(u))\left(\frac{d\ls{x}_{+n}}{du}(u)\right)^2.
\end{align}
Making use of \eqref{eq:228},\ldots,\eqref{eq:230} and \eqref{eq:234},\ldots,\eqref{eq:236} we deduce
\begin{align}
  \label{eq:260}
  \frac{d^2\ls{\alpha}_{-n}}{du^2}(u)=\Landau_{N_0}(1).
\end{align}
Together with (see \eqref{eq:229}, \eqref{eq:235})
\begin{align}
  \label{eq:261}
  \frac{d\ls{\alpha}_{-n}}{du}(0)&=\left(\pp{\ls{\alpha}^\ast}{t}\right)_0\frac{d\ls{t}_{+n}}{du}(0)+\left(\pp{\ls{\alpha}^\ast}{x}\right)_0\frac{d\ls{x}_{+n}}{du}(0)\nonumber\\
                                 &=\left(\pp{\ls{\alpha}^\ast}{t}\right)_0\frac{1}{\eta_0}\left(1-\frac{1}{\Gamma_0}\right)+\left(\pp{\ls{\alpha}^\ast}{x}\right)_0\left(\frac{1}{\Gamma_0}+1\right)
\end{align}
we obtain
\begin{align}
  \label{eq:262}
  \frac{d\ls{\alpha}_{-n}}{du}(u)=\left(\pp{\ls{\alpha}^\ast}{t}\right)_0\frac{1}{\eta_0}\left(1-\frac{1}{\Gamma_0}\right)+\left(\pp{\ls{\alpha}^\ast}{x}\right)_0\left(\frac{1}{\Gamma_0}+1\right)+\Landau_{N_0}(u).
\end{align}
With $\ls{\alpha}_{-n}^\ast(0)=\ls{\alpha}_0^\ast$, this implies
\begin{align}
  \label{eq:263}
  \ls{\alpha}_{-n}(u)=\ls{\alpha}_0^\ast+\left(\left(\pp{\ls{\alpha}^\ast}{t}\right)_0\frac{1}{\eta_0}\left(1-\frac{1}{\Gamma_0}\right)+\left(\pp{\ls{\alpha}^\ast}{x}\right)_0\left(\frac{1}{\Gamma_0}+1\right)\right)u+\Landau_{N_0}(u^2).
\end{align}
We note that analogous estimates hold for $\ls{\beta}_{-n}(u)$, $\frac{d\ls{\beta}_{-n}}{du}(u)$, $\frac{d^2\ls{\beta}_{-n}}{du^2}(u)$.

We have
\begin{align}
  \label{eq:264}
    \rs{\alpha}_{-n}(v)=\rs{\alpha}^\ast(\rs{t}_{+n}(v),\rs{x}_{+n}(v)),
\end{align}
where $\rs{\alpha}^\ast(t,x)$ is a smooth function of its arguments, given by the solution in the future development in the state ahead of the right moving shock. Analogously to the derivation of \eqref{eq:260}, \eqref{eq:262}, \eqref{eq:263} but now using \eqref{eq:231},\ldots,\eqref{eq:233} and \eqref{eq:237},\ldots,\eqref{eq:239} we obtain
\begin{align}
  \label{eq:265}
  \frac{d^2\rs{\alpha}_{-n}}{dv^2}(v)&=\Landau_{N_0}(1),\\
  \frac{d\rs{\alpha}_{-n}}{dv}(v)&=\left(\pp{\rs{\alpha}^\ast}{t}\right)_0\frac{1}{\eta_0}\left(1-\frac{a}{\Gamma_0}\right)+\left(\pp{\rs{\alpha}^\ast}{x}\right)_0\left(\frac{a}{\Gamma_0}+1\right)+\Landau_{N_0}(v),  \label{eq:266}\\
  \rs{\alpha}_{-n}(v)&=\rs{\alpha}_0^\ast+\left(\left(\pp{\rs{\alpha}^\ast}{t}\right)_0\frac{1}{\eta_0}\left(1-\frac{a}{\Gamma_0}\right)+\left(\pp{\rs{\alpha}^\ast}{x}\right)_0\left(\frac{a}{\Gamma_0}+1\right)\right)v+\Landau_{N_0}(v^2).      \label{eq:267}
\end{align}
We note that analogous estimates hold for $\rs{\beta}_{-n}(v)$, $\frac{d\rs{\beta}_{-n}}{dv}(v)$, $\frac{d^2\rs{\beta}_{-n}}{dv^2}(v)$.

We derive properties of $\is{V}_n$, $i=1,2$. We have
\begin{align}
  \label{eq:268}
  \is{V}_n=\frac{\jump{\is{\rho}_n\is{w}_n}}{\jump{\is{\rho}_n}},\qquad i=1,2.
\end{align}
Here
\begin{align}
  \label{eq:269}
  \jump{\is{\rho}_n}&=\is{\rho}_{+n}-\is{\rho}_{-n}\nonumber\\
                &=\rho(\is{\alpha}_{+n},\is{\beta}_{+n})-\rho(\is{\alpha}_{-n},\is{\beta}_{-n}),\qquad i=1,2
\end{align}
where $\rho(\alpha,\beta)$ is a given smooth function of its arguments. Similarly for $\jump{\is{\rho}_n\is{w}_n}$, $i=1,2$. Using the properties of $\ls{\alpha}_{+n}(u)$ as given by \eqref{eq:167}, \eqref{eq:170}, the properties of $\ls{\beta}_{+n}(u)$, as given by \eqref{eq:185}, the properties of $\ls{\alpha}_{-n}(u)$ (and $\ls{\beta}_{-n}(u)$) as given by \eqref{eq:262}, \eqref{eq:263} (and analogous estimates with $\ls{\beta}_{-n}$ in the role of $\ls{\alpha}_{-n}$), the properties of $\rs{\alpha}_{+n}(v)$ as given by \eqref{eq:183}, the properties of $\rs{\beta}_{+n}(v)$ as given by \eqref{eq:171}, \eqref{eq:173}, the properties of $\rs{\alpha}_{-n}(v)$ (and $\rs{\beta}_{-n}(v)$) as given by \eqref{eq:266}, \eqref{eq:267} (and analogous estimates with $\rs{\beta}_{-n}$ in the role of $\rs{\alpha}_{-n}$), we obtain
\begin{alignat}{3}
  \label{eq:270}
  \ls{V}_n(u)&=\ls{V}_0+\Landau(u),\qquad &\bigg|\frac{d\ls{V}_n}{du}(u)\bigg|&\leq C,\\
  \rs{V}_n(v)&=\rs{V}_0+\Landau(v),\qquad &\bigg|\frac{d\rs{V}_n}{dv}(v)\bigg|&\leq C.\label{eq:271}
\end{alignat}

We derive properties of $\cin_n(u,v)$, $\cout_n(u,v)$. We recall the notation (see \eqref{eq:149})
\begin{align}
  \label{eq:272}
  \cin_n(u,v)=\cin(\alpha_n(u,v),\beta_n(u,v)),\qquad  \cout_n(u,v)=\cout(\alpha_n(u,v),\beta_n(u,v)).
\end{align}
$\cin(\alpha,\beta)$, $\cout(\alpha,\beta)$ being smooth functions, the argument is analogous to the argument used to derive properties of $g_n(u,v)$, $h_n(u,v)$ (see \eqref{eq:196}, \eqref{eq:197}). Analogous to \eqref{eq:198}, \eqref{eq:199} we have
\begin{align}
  \label{eq:273}
  \cin_n(u,v)&=(\cin)_0+\left(\pp{\cin}{\alpha}\right)_0\alpha_0'u+\left(\pp{\cin}{\beta}\right)_0\beta_0'v+\Landau_A(v^2)+\Landau_B(v^2)+\Landau(uv)\nonumber\\
             &=-\eta_0+\Landau(v),\\
  \label{eq:274}
  \cout_n(u,v)&=(\cout)_0+\left(\pp{\cout}{\alpha}\right)_0\beta_0'u+\left(\pp{\cout}{\beta}\right)_0\beta_0'v+\Landau_A(u^2)+\Landau_B(v^2)+\Landau(uv)\nonumber\\
             &=\eta_0+\Landau(v).
\end{align}
Analogous to \eqref{eq:201} we have
\begin{align}
  \label{eq:275}
  \pp{\cin_n}{u}(u,v)=\left(\pp{\cin}{\alpha}\right)_0\alpha_0'+\Landau_A(v),\qquad\pp{\cin_n}{v}(u,v)=\left(\pp{\cin}{\beta}\right)_0\beta_0'+\Landau_B(v).
\end{align}
Analogous to \eqref{eq:205} we have
\begin{align}
  \label{eq:276}
\ppp{\cin_n}{u}(u,v)=\Landau_A(1),\qquad\ppp{\cin_n}{v}(u,v)=\Landau_B(1),\qquad\pppp{\cin_n}{u}{v}(u,v)=\Landau(1).
\end{align}
We note that analogous estimates hold for $\cout_n(u,v)$ and derivatives thereof.

We derive properties of
\begin{align}
  \label{eq:277}
  \ls{\cin}_{+n}(u)&=\cin(\ls{\alpha}_{+n}(u),\ls{\beta}_{+n}(u)),\\
  \rs{\cin}_{+n}(v)&=\cin(\rs{\alpha}_{+n}(v),\rs{\beta}_{+n}(v)).  \label{eq:278}
\end{align}
Using the properties of $\ls{\alpha}_{+n}(u)$, $\rs{\beta}_{+n}(v)$ as given by \eqref{eq:167}, \eqref{eq:170},\ldots,\eqref{eq:173} and the properties of $\rs{\alpha}_{+n}(u)$, $\ls{\beta}_{+n}(v)$ as given by \eqref{eq:183}, \eqref{eq:185}, we obtain
\begin{alignat}{3}
  \label{eq:279}
  \ls{\cin}_{+n}(u)&=-\eta_0+\Landau(u),&\qquad \frac{d\ls{\cin}_{+n}}{du}(u)&=\Landau(1),\\
  \rs{\cin}_{+n}(v)&=-\eta_0+\Landau(v),& \frac{d\rs{\cin}_{+n}}{dv}(v)&=\Landau(1).  \label{eq:280}
\end{alignat}
Analogously we obtain
\begin{alignat}{3}
  \label{eq:281}
  \ls{\cout}_{+n}(u)&=\eta_0+\Landau(u),&\qquad \frac{\ls{d\cout}_{+n}}{du}(u)&=\Landau(1),\\
  \rs{\cout}_{+n}(v)&=\eta_0+\Landau(v),& \frac{d\rs{\cout}_{+n}}{dv}(v)&=\Landau(1).  \label{eq:282}
\end{alignat}
Using these together with \eqref{eq:270}, \eqref{eq:271} in (see \eqref{eq:139})
\begin{align}
  \label{eq:283}
    \ls{\Gamma}_n=\frac{\ls{\cout}_{+n}}{\ls{\cin}_{+n}}\frac{\ls{V}_n-\ls{\cin}_{+n}}{\ls{\cout}_{+n}-\ls{V}_n},\qquad\rs{\Gamma}_n=a\frac{\rs{\cout}_{+n}}{\rs{\cin}_{+n}}\frac{\rs{V}_n-\rs{\cin}_{+n}}{\rs{\cout}_{+n}-\rs{V}_n},
\end{align}
we obtain
\begin{alignat}{5}
  \label{eq:284}
  \ls{\Gamma}_n(u)&=\ls{\Gamma}_0+\Landau(u)=-\frac{\eta_0+\ls{V}_0}{\eta_0-\ls{V}_0}+\Landau(u),&\qquad\frac{d\ls{\Gamma}_n}{du}(u)&=\Landau(1)\\
  \rs{\Gamma}_n(v)&=\rs{\Gamma}_0+\Landau(v)=-a\frac{\rs{V}_0+\eta_0}{\eta_0-\rs{V}_0}+\Landau(v),&\frac{d\rs{\Gamma}_n}{dv}(v)&=\Landau(1).  \label{eq:285}
\end{alignat}

We derive properties of $M_n(u,v)$. We have (see \eqref{eq:144})
\begin{align}
  \label{eq:286}
    M_n(u,v)=\mu_n(u,v)\pp{x_n}{u}(u,v)+\nu_n(u,v)\pp{x_n}{v}(u,v).
\end{align}
Here (see \eqref{eq:149})
\begin{align}
  \label{eq:287}
    \mu_n=\frac{1}{\cout_n-\cin_n}\frac{\cout_n}{\cin_n}\pp{\cin_n}{v},\qquad \nu_n=-\frac{1}{\cout_n-\cin_n}\frac{\cin_n}{\cout_n}\pp{\cout_n}{u}.
\end{align}
Using the properties of $\cin_n$, as given by \eqref{eq:273}, \eqref{eq:275}, \eqref{eq:276} (and analogous equations for $\cout_n$), we obtain
\begin{align}
  \label{eq:288}
  \mu_n(u,v)=-\frac{1}{2\eta_0}\left(\pp{\cin}{\beta}\right)_0\beta_0'+\Landau_B(v),\qquad\nu_n(u,v)=\frac{1}{2\eta_0}\left(\pp{\cout}{\alpha}\right)_0\alpha_0'+\Landau_A(v),\\
    \pp{\mu_n}{u}(u,v)=\Landau(1),\quad\pp{\mu_n}{v}(u,v)=\Landau_B(1),\quad\pp{\nu_n}{u}(u,v)=\Landau_A(1),\quad\pp{\nu_n}{v}(u,v)=\Landau(1).  \label{eq:289}
\end{align}
Together with \eqref{eq:187}, \eqref{eq:190}, \eqref{eq:191} we obtain
\begin{align}
  \label{eq:290}
  M_n(u,v)&=\frac{1}{2\eta_0}\Bigg(\left(\pp{\cout}{\alpha}\right)_0\alpha_0'-\left(\pp{\cin}{\beta}\right)_0\frac{\beta_0'}{\Gamma_0}\Bigg)+\Landau_B(v)+\Landau_A(v)+\Landau_{N_0}(v),\\
    \pp{M_n}{u}(u,v)&=\Landau_A(1)+\Landau_{N_0}(1),\qquad\pp{M_n}{v}(u,v)= \Landau_B(1)+\Landau_{N_0}(1).  \label{eq:291}
\end{align}

We turn to the properties of $x_{n+1}(u,v)$. We have (see \eqref{eq:141})
\begin{align}
  \label{eq:292}
  x_{n+1}(u,v)&=\frac{1}{\Gamma_0}u+v+\int_0^u\ls{\Phi}_n(u')du'+\int_0^v\rs{\Phi}_n(v')dv'\nonumber\\
              &\quad+\int_0^u\int_{av'}^{v'}M_n(u',v')du'dv'+\int_u^v\int_{av'}^uM_n(u',v')du'dv'.
\end{align}
where
\begin{align}
  \label{eq:293}
  \ls{\Phi}_n(u)&=\int_0^{u}\ls{\Lambda}_n(u')du'+\frac{1}{\ls{\Gamma}_n(u)}\int_{au}^{u}M_n(u',u)du',\\
  \rs{\Phi}_n(v)&=\int_0^{v}\rs{\Lambda}_n(v')dv'+\rs{\Gamma}_n(v)\int_{av}^{v}M_n(av,v')dv',  \label{eq:294}
\end{align}
where $M_n(u,v)$ is given by \eqref{eq:286} and
\begin{align}
  \ls{\Lambda}_n(u)&=\frac{d\ls{\gamma}_n}{du}(u)\pp{x_n}{u}(au,u)+\ls{\gamma}_n(u)a\ppp{x_n}{u}(au,u)+\ls{\gamma}_n(u)M_n(au,u),\label{eq:295}\\
  \label{eq:296}
  \rs{\Lambda}_n(v)&=\frac{d\rs{\gamma}_n}{dv}(v)\pp{x_n}{v}(av,av)+\rs{\gamma}_n(v)a\ppp{x_n}{v}(av,av)+\rs{\gamma}_n(v)aM_n(av,av),
\end{align}
where
\begin{align}
  \label{eq:297}
  \ls{\gamma}_n(u)=\frac{\rs{\Gamma}_n(u)}{\ls{\Gamma}_n(u)},\qquad\rs{\gamma}_n(v)=\frac{\rs{\Gamma}_n(v)}{\ls{\Gamma}_n(av)}.
\end{align}
The first order derivatives of $x_{n+1}(u,v)$ are
\begin{align}
  \label{eq:298}
  \pp{x_{n+1}}{u}(u,v)&=\frac{1}{\Gamma_0}+\ls{\Phi}_n(u)+\int_u^vM_n(u,v')dv'\nonumber\\
                      &=\frac{1}{\Gamma_0}+\int_0^{u}\ls{\Lambda}_n(u')du'+\frac{1}{\ls{\Gamma}_n(u)}\int_{au}^{u}M_n(u',u)du'+\int_u^vM_n(u,v')dv',\\
  \pp{x_{n+1}}{v}(u,v)&=1+\rs{\Phi}_n(v)+\int_{av}^uM_n(u',v)du'\nonumber\\
                      &=1+\int_0^{v}\rs{\Lambda}_n(v')dv'+\rs{\Gamma}_n(v)\int_{av}^{v}M_n(av,v')dv'+\int_{av}^uM_n(u',v)du'.\label{eq:299}
\end{align}
From \eqref{eq:298}, \eqref{eq:299} we see that
\begin{align}
  \label{eq:300}
  \pp{x_{n+1}}{u}(0,0)=\frac{1}{\Gamma_0},\qquad\pp{x_{n+1}}{v}(0,0)=1.
\end{align}
Taking another derivative of \eqref{eq:298}, \eqref{eq:299}, we obtain
\begin{align}
  \label{eq:301}
  \pppp{x_{n+1}}{u}{v}(u,v)&=M_n(u,v),\\
  \ppp{x_{n+1}}{u}(u,v)&=\ls{\Lambda}_n(u)-\frac{1}{\left(\ls{\Gamma}_n(u)\right)^2}\frac{d\ls{\Gamma}_n}{du}(u)\int_{au}^{u}M_n(u',u)du'\nonumber\\
                           &\quad +\frac{1}{\ls{\Gamma}_n(u)}\left(M_n(u,u)-aM_n(au,u)+\int_{au}^u\pp{M_n}{v}(u',u)du'\right)\nonumber\\
                           &\quad-M_n(u,u)+\int_u^v\pp{M_n}{u}(u,v')dv',  \label{eq:302}\\
  \ppp{x_{n+1}}{v}(u,v)&=\rs{\Lambda}_n(v)+\frac{d\rs{\Gamma}_n}{dv}(v)\int_{av}^{v}M_n(av,v')dv'\nonumber\\
                           &\quad +\rs{\Gamma}_n(v)\left(M_n(av,v)-aM_n(av,av)+\int_{av}^va\pp{M_n}{u}(av,v')dv'\right)\nonumber\\
  &\quad -aM_n(av,v)+\int_{av}^u\pp{M_n}{v}(u',v)du'.  \label{eq:303}
\end{align}
Using \eqref{eq:290} we have
\begin{align}
  \label{eq:304}
    \left|\pppp{x_{n+1}}{u}{v}(u,v)\right|\leq C.
\end{align}
Using \eqref{eq:284}, \eqref{eq:285}, \eqref{eq:290}, \eqref{eq:291} we obtain
\begin{align}
  \label{eq:305}
  \left|\ppp{x_{n+1}}{u}(u,v)\right|&\leq\left|\ls{\Lambda}_n(u)\right|+C,\\
  \left|\ppp{x_{n+1}}{v}(u,v)\right|&\leq\left|\rs{\Lambda}_n(v)\right|+C.  \label{eq:306}
\end{align}
In order to estimate $\is{\Lambda}_n$, $i=1,2$ we use \eqref{eq:284}, \eqref{eq:285} and the induction hypothesis \eqref{eq:187} to obtain
\begin{align}
  \left|\ls{\Lambda}_n(u)\right|&\leq C+\left|a\ls{\gamma}_n(u)\ppp{x_n}{u}(au,u)\right|,\label{eq:307}\\
  \left|\rs{\Lambda}_n(v)\right|&\leq C+\left|a\rs{\gamma}_n(v)\ppp{x_n}{v}(av,av)\right|,  \label{eq:308}
\end{align}
provided we choose $\varepsilon$ sufficiently small. In view of \eqref{eq:35} we have
\begin{align}
  \label{eq:309}
  \ls{\gamma}_n(u)=\frac{\rs{\Gamma}_0}{\ls{\Gamma}_0}+\Landau(u)=1+\Landau(u),\\
  \rs{\gamma}_n(v)=\frac{\rs{\Gamma}_0}{\ls{\Gamma}_0}+\Landau(v)=1+\Landau(v).\label{eq:310}
\end{align}
Therefore,
\begin{align}
  \left|\ls{\Lambda}_n(u)\right|&\leq C+(a+Cu)\left|\ppp{x_n}{u}(au,u)\right|\leq C+(a+Cu)N_0\leq C+aN_0,  \label{eq:311}\\
  \label{eq:312}
  \left|\rs{\Lambda}_n(v)\right|&\leq C+(a+Cv)\left|\ppp{x_n}{v}(av,av)\right|\leq C+(a+Cv)N_0\leq C+aN_0,
\end{align}
where we also used the induction hypothesis \eqref{eq:187}. Using these in \eqref{eq:305}, \eqref{eq:306} we obtain
\begin{align}
  \label{eq:313}
  \left|\ppp{x_{n+1}}{u}(u,v)\right|&\leq C+aN_0,\\
  \left|\ppp{x_{n+1}}{v}(u,v)\right|&\leq C+aN_0.  \label{eq:314}
\end{align}
Taking the supremum in $T_\varepsilon$ of \eqref{eq:304}, \eqref{eq:313}, \eqref{eq:314} we obtain
\begin{align}
  \label{eq:315}
  \sup_{T_\varepsilon}\left|\pppp{x_{n+1}}{u}{v}\right|\leq C,\qquad \sup_{T_\varepsilon}\left|\ppp{x_{n+1}}{u}\right|\leq C+aN_0,\qquad \sup_{T_\varepsilon}\left|\ppp{x_{n+1}}{v}\right|\leq C+aN_0.
\end{align}
Therefore,
\begin{align}
  \label{eq:316}
    \|x_{n+1}\|_0=\max\left\{\sup_{T_\varepsilon}\left|\ppp{x_{n+1}}{u}\right|,\sup_{T_\varepsilon}\left|\pppp{x_{n+1}}{u}{v}\right|,\sup_{T_\varepsilon}\left|\ppp{x_{n+1}}{v}\right|\right\}\leq C+aN_0.
\end{align}
Choosing the constant $N_0$ sufficiently large, such that
\begin{align}
  \label{eq:317}
    \frac{C}{1-a}\leq N_0,
\end{align}
where the constant $C$ on the left is the constant $C$ appearing in \eqref{eq:316}, we obtain
\begin{align}
  \label{eq:318}
    \|x_{n+1}\|_0\leq N_0.
\end{align}
Together with  \eqref{eq:292} (which implies $x_{n+1}(0,0)=0$) and \eqref{eq:300} we see that the function $x_{n+1}(u,v)$ lies in $B_{N_0}$ (compare the definition of $B_{N_0}$ in \eqref{eq:156}).

We derive properties of $\ls{\alpha}_{+,n+1}(u)$, $\rs{\beta}_{+,n+1}(u)$. We have (see \eqref{eq:150}, \eqref{eq:151})
\begin{align}
  \label{eq:319}
    \ls{\alpha}_{+,n+1}(u)&=\ls{H}(\ls{\beta}_{+n}(u),\ls{\alpha}_{-n}(u),\ls{\beta}_{-n}(u)),\\
  \rs{\beta}_{+,n+1}(v)&=\rs{H}(\rs{\alpha}_{+n}(v),\rs{\alpha}_{-n}(v),\rs{\beta}_{-n}(v)).\label{eq:320}
\end{align}
Here
\begin{align}
  \label{eq:321}
  \rs{\alpha}_{+n}(v)=\ls{\alpha}_{+n}(av),\qquad \ls{\beta}_{+n}(u)=\rs{\beta}_{+n}(u).
\end{align}
Using the induction hypotheses for $\ls{\alpha}_{+n}$, $\rs{\beta}_{+n}$ we have
\begin{align}
  \label{eq:322}
  \rs{\alpha}_{+n}(0)=\ls{\alpha}_{+n}(0)=\beta_0,\qquad \ls{\beta}_{+n}(0)=\rs{\beta}_{+n}(0)=\beta_0.
\end{align}
From \eqref{eq:263}, \eqref{eq:267} and the analogous equations for $\is{\beta}_{-n}$, $i=1,2$, we have
\begin{align}
  \label{eq:323}
  \ls{\alpha}_{-n}(0)=\ls{\alpha}_0^\ast,\qquad\ls{\beta}_{-n}(0)=\ls{\beta}_0^\ast,\\
  \rs{\alpha}_{-n}(0)=\rs{\alpha}_0^\ast,\qquad\rs{\beta}_{-n}(0)=\rs{\beta}_0^\ast.  \label{eq:324}
\end{align}
Using \eqref{eq:322}, \eqref{eq:323}, \eqref{eq:324} in \eqref{eq:319}, \eqref{eq:320}, we obtain, in view of \eqref{eq:64}, that
\begin{align}
  \label{eq:325}
  \ls{\alpha}_{+,n+1}(0)=\beta_0,\qquad \rs{\beta}_{+,n+1}(0)=\beta_0.
\end{align}
For the first derivative of $\ls{\alpha}_{+,n+1}(u)$ we have (for the functions $\ls{F}$, $\ls{M}_1$, $\ls{M}_2$ see \eqref{eq:69}, \eqref{eq:70})
\begin{align}
  \label{eq:326}
  \frac{d\ls{\alpha}_{+,n+1}}{du}(u)&=\ls{F}(\ls{\beta}_{+n}(u),\ls{\alpha}_{-n}(u),\ls{\beta}_{-n}(u))\frac{d\ls{\beta}_{+n}}{du}(u)\nonumber\\
                                    &\quad +\ls{M}_1(\ls{\beta}_{+n}(u),\ls{\alpha}_{-n}(u),\ls{\beta}_{-n}(u))\frac{d\ls{\alpha}_{-n}}{du}(u)\nonumber\\
  &\quad+\ls{M}_2(\ls{\beta}_{+n}(u),\ls{\alpha}_{-n}(u),\ls{\beta}_{-n}(u))\frac{d\ls{\beta}_{-n}}{du}(u).
\end{align}
Hence, (see \eqref{eq:185}, \eqref{eq:262} (and the analogous statement for $\ls{\beta}_{-n}(u)$)
\begin{align}
  \label{eq:327}
\frac{d\ls{\alpha}_{+,n+1}}{du}(0)&=\ls{F}_0\beta_0'+\ls{M}_{10}\frac{d\ls{\alpha}_-}{du}(0)+\ls{M}_{20}\frac{d\ls{\beta}_-}{du}(0).
\end{align}
(see \eqref{eq:84}, \eqref{eq:85} for $\frac{d\ls{\alpha}_-}{du}(0)$, $\frac{d\ls{\beta}_-}{du}(0)$). Analogously we arrive at
\begin{align}
  \label{eq:328}
      \frac{d\rs{\beta}_{+,n+1}}{dv}(0)&=a\rs{F}_0\alpha_0'+\rs{M}_{10}\frac{d\rs{\alpha}_-}{dv}(0)+\rs{M}_{20}\frac{d\rs{\beta}_-}{dv}(0).
\end{align}
(see \eqref{eq:86}, \eqref{eq:87} for $\frac{d\rs{\alpha}_-}{dv}(0)$, $\frac{d\rs{\beta}_-}{dv}(0)$). We rewrite \eqref{eq:327}, \eqref{eq:328} in the form (see \eqref{eq:90}, \eqref{eq:91} for $a_0$, $b_0$)
\begin{align}
  \label{eq:329}
  \begin{pmatrix}
    \frac{d\ls{\alpha}_{+,n+1}}{du}(0)\\\frac{d\rs{\beta}_{+,n+1}}{dv}(0)
  \end{pmatrix}=
  \begin{pmatrix}
    0 & \ls{F}_0\\a\rs{F}_0 & 0
  \end{pmatrix}
                              \begin{pmatrix}
                                \alpha_0'\\\beta_0'
                              \end{pmatrix}+
  \begin{pmatrix}
    a_0\\b_0
  \end{pmatrix}
\end{align}
Using (see \eqref{eq:88}, \eqref{eq:89})
\begin{align}
  \label{eq:330}
  \begin{pmatrix}
    a_0\\b_0
  \end{pmatrix}=M
  \begin{pmatrix}
    \alpha_0'\\\beta_0'
  \end{pmatrix},\qquad\textrm{where}\qquad M&=
  \begin{pmatrix}
    1&-\ls{F}_0\\-a\rs{F}_0 & 1
  \end{pmatrix},
\end{align}
we obtain
\begin{align}
  \label{eq:331}
  \frac{d\ls{\alpha}_{+,n+1}}{du}(0)=\alpha_0',\qquad\frac{d\rs{\beta}_{+,n+1}}{dv}(0)=\beta_0'.
\end{align}

For the second derivative we have
\begin{align}
  \label{eq:332}
  \frac{d^2\ls{\alpha}_{+,n+1}}{du}(u)&=\ls{F}(\ls{\beta}_{+n}(u),\ls{\alpha}_{-n}(u),\ls{\beta}_{-n}(u))\frac{d^2\ls{\beta}_{+n}}{du^2}(u)\nonumber\\
&\quad +\ls{M}_1(\ls{\beta}_{+n}(u),\ls{\alpha}_{-n}(u),\ls{\beta}_{-n}(u))\frac{d^2\ls{\alpha}_{-n}}{du^2}(u)\nonumber\\
                                      &\quad+\ls{M}_2(\ls{\beta}_{+n}(u),\ls{\alpha}_{-n}(u),\ls{\beta}_{-n}(u))\frac{d^2\ls{\beta}_{-n}}{du^2}(u)\nonumber\\
                                      &\quad +\ls{G}\bigg(\ls{\beta}_{+n}(u),\ls{\alpha}_{-n}(u),\ls{\beta}_{-n}(u),\frac{d\ls{\beta}_{+n}}{du}(u),\frac{d\ls{\alpha}_{-n}}{du}(u),\frac{d\ls{\beta}_{-n}}{du}(u)\bigg).
\end{align}
Here $\ls{G}$ is a smooth function of its arguments. Using \eqref{eq:260}, \eqref{eq:262}, \eqref{eq:263} for the properties of $\ls{\alpha}_{-n}$ and analogous estimates for the properties of $\ls{\beta}_{-n}$ together with the fact that $N_0$ is a fixed numerical constant and \eqref{eq:185} for the properties of $\ls{\beta}_{+n}$ we obtain
\begin{align}
  \label{eq:333}
  \frac{d^2\ls{\alpha}_{+,n+1}}{du^2}(u)\leq(\ls{F}_0+Cu)B+C.
\end{align}
Analogously we arrive at (see the third of \eqref{eq:183} for the origin of the additional prefactor $a^2$)
\begin{align}
  \label{eq:334}
  \frac{d^2\rs{\beta}_{+,n+1}}{dv^2}(v)\leq(a^2\rs{F}_0+Cv)A+C.
\end{align}
Taking the supremum of \eqref{eq:333} in $[0,a\varepsilon]$ and of \eqref{eq:334} in $[0,\varepsilon]$ we obtain (see \eqref{eq:160}, \eqref{eq:161} for the definitions of $\|f\|_1$, $\|f\|_2$)
\begin{align}
  \label{eq:335}
  \|\ls{\alpha}_{+,n+1}\|_1&\leq(\ls{F}_0+C\varepsilon)B+C,\\
  \|\ls{\beta}_{+,n+1}\|_2&\leq(a^2\rs{F}_0+C\varepsilon)A+C.\label{eq:336}
\end{align}
Choosing the constant $A$ as the right hand side of \eqref{eq:335}, i.e.
\begin{align}
  \label{eq:337}
  A=(\ls{F}_0+C\varepsilon)B+C
\end{align}
and using this in \eqref{eq:336} we obtain
\begin{align}
  \label{eq:338}
\|\ls{\beta}_{+,n+1}\|_2&\leq(a^2\rs{F}_0+C\varepsilon)(\ls{F}_0+C\varepsilon)B+C.
\end{align}
Recalling that $\ls{F}_0\rs{F}_0=a^2$ (see \eqref{eq:74}) and choosing $\varepsilon$ sufficiently small, such that
\begin{align}
  \label{eq:339}
  (a^2\rs{F}_0+C\varepsilon)(\ls{F}_0+C\varepsilon)\leq k<1,
\end{align}
we have
\begin{align}
  \label{eq:340}
\|\ls{\beta}_{+,n+1}\|_2&\leq kB+C.
\end{align}
Choosing the constant $B$ such that
\begin{align}
  \label{eq:341}
  \frac{C}{1-k}\leq B,
\end{align}
where the constant $C$ is the one appearing in \eqref{eq:340}, we have
\begin{align}
  \label{eq:342}
\|\ls{\beta}_{+,n+1}\|_2&\leq B.
\end{align}
The choice of $B$ also makes $A$ into a fixed numerical constant. We have
\begin{align}
  \label{eq:343}
  \|\ls{\alpha}_{+,n+1}\|_1&\leq A.
\end{align}
Together with \eqref{eq:325}, \eqref{eq:331} we conclude that $(\ls{\alpha}_{+,n+1},\rs{\beta}_{+,n+1})\in B_A\times B_B$ (see \eqref{eq:157}, \eqref{eq:158} for the definition of $B_A$, $B_B$). This completes the proof of the inductive step.
\end{proof}

\subsection{Convergence}
\begin{proposition}
  For $\varepsilon$ sufficiently small, the sequence
  \begin{align}
    \label{eq:344}
    (\ls{\alpha}_{+n},\rs{\beta}_{+n},x_n);n=0,1,2,\ldots)
  \end{align}
  converges in $B_A\times B_B\times B_{N_0}$.
\end{proposition}
\begin{proof}
  We use the notation
\begin{align}
  \label{eq:345}
  \Delta_nf\defeq f_n-f_{n-1}.
\end{align}
We first look at differences of $\ls{\alpha}_+$, $\rs{\beta}_+$. Using the leading order behavior of $\ls{\alpha}_{+n}(u)$ and $\rs{\beta}_{+n}(v)$ as given in \eqref{eq:172}, \eqref{eq:173} respectively and recalling the definitions of $\|f\|_1$, $\|f\|_2$ in \eqref{eq:160}, \eqref{eq:161} respectively, we have
\begin{alignat}{3}
  \label{eq:346}
  |\Delta_n\ls{\alpha}_+(u)|&\leq u^2\|\Delta_n\ls{\alpha}_+\|_1,\qquad& \bigg|\Delta_n\frac{d\ls{\alpha}_+}{du}(u)\bigg|&\leq u\|\Delta_n\ls{\alpha}_+\|_1,\\
  |\Delta_n\rs{\beta}_+(v)|&\leq v^2\|\Delta_n\rs{\beta}_+\|_2,\qquad& \bigg|\Delta_n\frac{d\rs{\beta}_+}{dv}(v)\bigg|&\leq v\|\Delta_n\rs{\beta}_+\|_2.  \label{eq:347}
\end{alignat}
Using $\rs{\alpha}_{+n}(v)=\ls{\alpha}_{+n}(av)$ we have
\begin{alignat}{5}
  \label{eq:348}
  |\Delta_n\rs{\alpha}_+(v)|&\leq v^2\|\Delta_n\ls{\alpha}_+\|_1,\quad&\bigg|\Delta_n\frac{d\rs{\alpha}_+}{dv}(v)\bigg|&\leq av\|\Delta_n\ls{\alpha}_+\|_1,\quad&\bigg|\Delta_n\frac{d^2\rs{\alpha}_+}{dv^2}(v)\bigg|&\leq a^2\|\Delta_n\ls{\alpha}_+\|_1,\\
\intertext{while using $\ls{\beta}_{+n}(u)=\rs{\beta}_{+n}(u)$ we have}
  \label{eq:349}
    |\Delta_n\ls{\beta}_+(u)|&\leq u^2\|\Delta_n\rs{\beta}_+\|_2,\quad&\bigg|\Delta_n\frac{d\ls{\beta}_+}{du}(u)\bigg|&\leq u\|\Delta_n\rs{\beta}_+\|_2,\quad&\bigg|\Delta_n\frac{d^2\ls{\beta}_+}{du^2}(u)\bigg|&\leq \|\Delta_n\rs{\beta}_+\|_2.
\end{alignat}
Using $\alpha_n(u,v)=\ls{\alpha}_{+n}(u)$ and $\beta_n(u,v)=\rs{\beta}_{+n}(v)$ we have
\begin{alignat}{3}
  \label{eq:350}
  |\Delta_n\alpha|&\leq u^2\|\Delta_n\ls{\alpha}_+\|,\qquad&\left|\Delta_n\pp{\alpha}{u}\right|&\leq u\|\Delta_n\ls{\alpha}_+\|,\qquad&\left|\Delta_n\ppp{\alpha}{u}\right|&\leq\|\Delta_n\ls{\alpha}_{+}\|_1,\\
    |\Delta_n\beta|&\leq v^2\|\Delta_n\rs{\beta}_+\|,\qquad&\left|\Delta_n\pp{\beta}{v}\right|&\leq v\|\Delta_n\rs{\beta}_+\|,\qquad&\left|\Delta_n\ppp{\beta}{v}\right|&\leq\|\Delta_n\rs{\beta}_{+}\|_2.\label{eq:351}
\end{alignat}
All other derivatives of second or lower order vanish.

We turn to estimates of $\Delta_nx(u,v)$ and derivatives thereof. From \eqref{eq:188}, \eqref{eq:189} together with \eqref{eq:186} we have
\begin{align}
  \label{eq:352}
  \left|\Delta_n\pp{x}{u}\right|&\leq v\left(\sup_{T_\varepsilon}\left|\Delta_n\ppp{x}{u}\right|+\sup_{T_\varepsilon}\left|\Delta_n\pppp{x}{u}{v}\right|\right),\\
    \left|\Delta_n\pp{x}{v}\right|&\leq v\left(\sup_{T_\varepsilon}\left|\Delta_n\ppp{x}{v}\right|+\sup_{T_\varepsilon}\left|\Delta_n\pppp{x}{u}{v}\right|\right).  \label{eq:353}
\end{align}
Using these in \eqref{eq:192} we obtain
\begin{align}
  \label{eq:354}
  |\Delta_nx|\leq v^2\|\Delta_nx\|_0.
\end{align}

We look at $\Delta_ng$, $\Delta_nh$. We recall that
\begin{align}
  \label{eq:355}
  g_n(u,v)=G(\alpha_n(u,v),\beta_n(u,v)),\qquad h_n(u,v)=H(\alpha_n(u,v),\beta_n(u,v)),
\end{align}
where $G(\alpha,\beta)$, $H(\alpha,\beta)$ are smooth functions of their arguments (see \eqref{eq:196} for the definition of those functions). In view of \eqref{eq:172}, \eqref{eq:173}, the points $(\ls{\alpha}_{+n}(u),\rs{\beta}_{+n}(v))$ and $(\ls{\alpha}_{+n-1}(u),\rs{\beta}_{+n-1}(v))$ lie in a ball in $\mathbb{R}^2$ centered at $(\beta_0,\beta_0)$ (recall that $\alpha_0=\beta_0$). Hence so does the line segment joining them. We therefore have
\begin{align}
  \label{eq:356}
  |\Delta_ng|&\leq Cv^2\left(\|\Delta_n\ls{\alpha}_{+}\|_1+\|\Delta_n\rs{\beta}_+\|_2\right),\\
  \left|\Delta_n\pp{g}{u}\right|,\left|\Delta_n\pp{g}{v}\right|&\leq Cv\left(\|\Delta_n\ls{\alpha}_{+}\|_1+\|\Delta_n\rs{\beta}_+\|_2\right),\label{eq:357}\\
  \left|\Delta_n\ppp{g}{u}\right|,\left|\Delta_n\ppp{g}{v}\right|&\leq C\left(\|\Delta_n\ls{\alpha}_{+}\|_1+\|\Delta_n\rs{\beta}_+\|_2\right),\label{eq:358}\\
  \left|\Delta_n\pppp{g}{u}{v}\right|&\leq Cv\left(\|\Delta_n\ls{\alpha}_{+}\|_1+\|\Delta_n\rs{\beta}_+\|_2\right).\label{eq:359}
\end{align}
The same estimates hold with $h$ in the role of $g$.

We look at $\phi$, $\psi$. We recall
\begin{align}
  \label{eq:360}
  \phi_n(u,v)=g_n(u,v)\pp{x_n}{u}(u,v),\qquad\psi_n(u,v)=h_n(u,v)\pp{x_n}{v}(u,v).
\end{align}
Using \eqref{eq:352}, \eqref{eq:353}, \eqref{eq:356} (and the latter also with $h$ in the role of $g$) we obtain
\begin{align}
  \label{eq:361}
  |\Delta_n\phi|&\leq C\left(v^2\left(\|\Delta_n\ls{\alpha}_{+}\|_1+\|\Delta_n\rs{\beta}_+\|_2\right)+v\|\Delta_nx\|_0\right),\\
  \label{eq:362}
  |\Delta_n\psi|&\leq C\left(v^2\left(\|\Delta_n\ls{\alpha}_{+}\|_1+\|\Delta_n\rs{\beta}_+\|_2\right)+v\|\Delta_nx\|_0\right).
\end{align}
Using in addition also \eqref{eq:357} (and this one also with $h$ in the role of $g$) we obtain for differences of derivatives of $\phi$ and $\psi$
\begin{align}
  \label{eq:363}
  \left|\Delta_n\pp{\phi}{u}\right|,\left|\Delta_n\pp{\phi}{v}\right|,\left|\Delta_n\pp{\psi}{u}\right|,\left|\Delta_n\pp{\psi}{v}\right|&\leq C\left(v\left(\|\Delta_n\ls{\alpha}_{+}\|_1+\|\Delta_n\rs{\beta}_+\|_2\right)+\|\Delta_nx\|_0\right).
\end{align}

We turn to $\Delta_nt(u,v)$. We recall
\begin{align}
  \label{eq:364}
  t_n(u,v)=\int_0^u(\phi_n+\psi_n)(u',u')du'+\int_u^v\psi_n(u,v')dv'.
\end{align}
Using \eqref{eq:361}, \eqref{eq:362} we obtain
\begin{align}
  \label{eq:365}
  |\Delta_nt|\leq C\left(v^3\left(\|\Delta_n\ls{\alpha}_{+}\|_1+\|\Delta_n\rs{\beta}_+\|_2\right)+v^2\|\Delta_nx\|_0\right).
\end{align}
From $\pp{t_n}{v}(u,v)=\psi_n(u,v)$ we obtain, using \eqref{eq:362},
\begin{align}
  \label{eq:366}
  \left|\Delta_n\pp{t}{v}\right|\leq C\left(v^2\left(\|\Delta_n\ls{\alpha}_{+}\|_1+\|\Delta_n\rs{\beta}_+\|_2\right)+v\|\Delta_nx\|_0\right).
\end{align}
From
\begin{align}
  \label{eq:367}
  \pp{t_n}{u}(u,v)=\phi_n(u,u)+\int_u^v\pp{\psi_n}{u}(u,v')dv'
\end{align}
we obtain, using \eqref{eq:361}, \eqref{eq:363},
\begin{align}
  \label{eq:368}
  \left|\Delta_n\pp{t}{u}\right|\leq C\left(v^2\left(\|\Delta_n\ls{\alpha}_{+}\|_1+\|\Delta_n\rs{\beta}_+\|_2\right)+v\|\Delta_nx\|_0\right).
\end{align}
From $\pppp{t_n}{u}{v}(u,v)=\pp{\psi_n}{u}(u,v)$ we obtain, using \eqref{eq:363},
\begin{align}
  \label{eq:369}
  \left|\Delta_n\pppp{t}{u}{v}\right|\leq C\left(v\left(\|\Delta_n\ls{\alpha}_{+}\|_1+\|\Delta_n\rs{\beta}_+\|_2\right)+\|\Delta_nx\|_0\right).
\end{align}
From $\ppp{t_n}{v}(u,v)=\pp{\psi_n}{v}(u,v)$ we obtain, using \eqref{eq:363},
\begin{align}
  \label{eq:370}
  \left|\Delta_n\ppp{t}{v}\right|\leq C\left(v\left(\|\Delta_n\ls{\alpha}_{+}\|_1+\|\Delta_n\rs{\beta}_+\|_2\right)+\|\Delta_nx\|_0\right).
\end{align}
For $\Delta_n\ppp{t}{u}$ we start with equation (see \eqref{eq:218})
\begin{align}
  \label{eq:371}
  \ppp{t_n}{u}(u,v)=\pp{\phi_n}{u}(u,u)+\pp{\phi_n}{v}(u,u)+\pp{}{u}\left(\int_u^v\pp{\psi_n}{u}(u,v')dv'\right).
\end{align}
Differences of the first two terms can be estimated by \eqref{eq:363}. For the last term we use \eqref{eq:220}. Taking the difference of \eqref{eq:220} and itself but with $\psi_{n-1}$ in the role of $\psi_n$ and using the estimates for the differences $\Delta_nh$ and derivatives thereof as given in \eqref{eq:356}, \eqref{eq:357}, \eqref{eq:358}, \eqref{eq:359} (but with $h$ in the role of $g$) together with the estimates for the differences $\Delta_nx$ and derivatives thereof as given in \eqref{eq:352}, \eqref{eq:353}. We obtain
\begin{align}
  \label{eq:372}
  \left|\Delta_n\left(\pp{}{u}\left(\int_u^v\pp{\psi}{u}(u,v')dv'\right)\right)\right|&\leq C\left(v\left(\|\Delta_n\ls{\alpha}_{+}\|_1+\|\Delta_n\rs{\beta}_+\|_2\right)+\|\Delta_nx\|_0\right).
\end{align}
Using this together with \eqref{eq:363} to estimate differences of the first two terms in \eqref{eq:371} we obtain
\begin{align}
  \label{eq:373}
  \left|\Delta_n\ppp{t}{u}\right|\leq C\left(v\left(\|\Delta_n\ls{\alpha}_{+}\|_1+\|\Delta_n\rs{\beta}_+\|_2\right)+\|\Delta_nx\|_0\right).
\end{align}
Now we look at
\begin{alignat}{3}
  \label{eq:374}
  \ls{t}_{+n}(u)&=t_n(u,u),&\qquad \rs{t}_{+n}(v)&=t_n(av,v),\\
  \ls{x}_{+n}(u)&=x_n(u,u),& \rs{x}_{+n}(v)&=x_n(av,v).  \label{eq:375}
\end{alignat}
From \eqref{eq:365}, \eqref{eq:366}, \eqref{eq:368}, \eqref{eq:369}, \eqref{eq:370}, \eqref{eq:373} we have
\begin{align}
  \label{eq:376}
  |\Delta_n\ls{t}_+(u)|&\leq C\left(u^3\left(\|\Delta_n\ls{\alpha}_{+}\|_1+\|\Delta_n\rs{\beta}_+\|_2\right)+u^2\|\Delta_nx\|_0\right),\\
  \label{eq:377}
  \left|\Delta_n\frac{d\ls{t}_+}{du}(u)\right|&\leq C\left(u^2\left(\|\Delta_n\ls{\alpha}_{+}\|_1+\|\Delta_n\rs{\beta}_+\|_2\right)+u\|\Delta_nx\|_0\right),\\
  \label{eq:378}
  \left|\Delta_n\frac{d^2\ls{t}_+}{du^2}(u)\right|&\leq C\left(u\left(\|\Delta_n\ls{\alpha}_{+}\|_1+\|\Delta_n\rs{\beta}_+\|_2\right)+\|\Delta_nx\|_0\right).
\end{align}
The same estimates hold also for $\Delta_n\rs{t}_+$, $\Delta_n\frac{d\rs{t}_+}{dv}$, $\Delta_n\frac{d^2\rs{t}_+}{dv^2}$ but with $v$ in the role of $u$. From \eqref{eq:352}, \eqref{eq:353}, \eqref{eq:354} we have
\begin{align}
  \label{eq:379}
  |\Delta_n\ls{x}_+(u)|&\leq u^2\|\Delta_nx\|_0,\\
  \label{eq:380}
  \bigg|\Delta_n\frac{d\ls{x}_+}{du}(u)\bigg|&\leq Cu\|\Delta_nx\|_0,\\
  \label{eq:381}
  \bigg|\Delta_n\frac{d^2\ls{x}_+}{du^2}(u)\bigg|&\leq C\|\Delta_nx\|_0.
\end{align}
The same estimates hold also for $\Delta_n\rs{x}_+$, $\Delta\frac{d\rs{x}_+}{dv}$, $\Delta\frac{d^2\rs{x}_+}{dv^2}$ but with $v$ in the role of $u$.

We turn to estimates for $\Delta_n\is{\alpha}_-$, $\Delta_n\is{\beta}_-$, $i=1,2$. In view of
\begin{align}
  \label{eq:382}
  \ls{\alpha}_{-n}(u)=\ls{\alpha}_-^\ast(\ls{t}_+(u),\ls{x}_+(u)),
\end{align}
the above estimates for $\Delta_n\ls{t}_+$, $\Delta_n\ls{x}_+$ yield
\begin{align}
  \label{eq:383}
  |\Delta_n\ls{\alpha}_-(u)|,|\Delta_n\ls{\beta}_-(u)|&\leq C\left(u^3\left(\|\Delta_n\ls{\alpha}_{+}\|_1+\|\Delta_n\rs{\beta}_+\|_2\right)+u^2\|\Delta_nx\|_0\right),\\
  \bigg|\Delta_n\frac{d\ls{\alpha}_-}{du}(u)\bigg|,\bigg|\Delta_n\frac{d\ls{\beta}_-}{du}(u)\bigg|&\leq C\left(u^2\left(\|\Delta_n\ls{\alpha}_{+}\|_1+\|\Delta_n\rs{\beta}_+\|_2\right)+u\|\Delta_nx\|_0\right),\label{eq:384}\\
  \bigg|\Delta_n\frac{d^2\ls{\alpha}_-}{du^2}(u)\bigg|,\bigg|\Delta_n\frac{d^2\ls{\beta}_-}{du^2}(u)\bigg|&\leq C\left(u\left(\|\Delta_n\ls{\alpha}_{+}\|_1+\|\Delta_n\rs{\beta}_+\|_2\right)+\|\Delta_nx\|_0\right).\label{eq:385}
\end{align}
The same estimates hold for $\Delta_n\rs{\alpha}_-$, $\Delta_n\rs{\beta}_-$ and its derivatives but with $v$ in the role of $u$.

We look at $\Delta_n\is{V}$, $i=1,2$ and its derivatives. We recall
\begin{align}
  \label{eq:386}
  \ls{V}_n(u)=\frac{\jump{\ls{\rho}_n(u)\ls{w}_n(u)}}{\jump{\ls{\rho}_n(u)}}.
\end{align}
where
\begin{align}
  \label{eq:387}
  \jump{f_n(u)}=f_{+n}(u)-f_{-n}(u)
\end{align}
and
\begin{align}
  \label{eq:388}
  \ls{\rho}_{\pm}(u)=\rho(\ls{\alpha}_{\pm n}(u),\ls{\beta}_{\pm n}(u)),\qquad \ls{w}_{\pm}(u)=w(\ls{\alpha}_{\pm n}(u),\ls{\beta}_{\pm n}(u)).
\end{align}
We recall that $\rho(\alpha,\beta)$, $w(\alpha,\beta)$ are smooth functions of their arguments. In view of \eqref{eq:263} and the analogous estimate with $\beta$ in the role of $\alpha$, the points $(\ls{\alpha}_{-n}(u),\ls{\beta}_{-n}(u))$ and $(\ls{\alpha}_{-n-1}(u),\ls{\beta}_{-n-1}(u))$ lie in a ball in $\mathbb{R}^2$ centered at $(\ls{\alpha}_0^\ast,\ls{\beta}_0^\ast)$. Hence so does the line segment joining them. The analogous statement involving $\alpha_{+n}$, $\beta_{+n}$ has been made right above \eqref{eq:356} (the statement right above \eqref{eq:356} is actually about $(\ls{\alpha}_{+n},\rs{\beta}_{+n}$) but since $\ls{\beta}_{+n}(u)=\rs{\beta}_{+n}(u)$ it is equally applicable). Therefore, the above estimates for $\Delta_n\ls{\alpha}_\pm$, $\Delta_n\ls{\beta}_\pm$ and their first derivative (see \eqref{eq:346}, \eqref{eq:349}, \eqref{eq:383}, \eqref{eq:384}) yield
\begin{align}
  \label{eq:389}
  |\Delta_n\ls{V}(u)|&\leq Cu^2\Big(\|\Delta_n\ls{\alpha}_{+}\|_1+\|\Delta_n\rs{\beta}_+\|_2+\|\Delta_nx\|_0\Big),\\
  \bigg|\Delta_n\frac{d\ls{V}}{du}(u)\bigg|&\leq Cu\Big(\|\Delta_n\ls{\alpha}_{+}\|_1+\|\Delta_n\rs{\beta}_+\|_2+\|\Delta_nx\|_0\Big).\label{eq:390}
\end{align}
The same estimates hold for $\Delta_n\rs{V}(v)$, $\Delta_n\frac{d\rs{V}}{dv}(v)$ but with $v$ in the role of $u$ on the right hand sides of the estimates.

We turn to estimates for $\Delta_n\cin$, $\Delta_n\cout$. We recall
\begin{align}
  \label{eq:391}
  \cin_n(u,v)=\cin(\alpha_n(u,v),\beta_n(u,v)),\qquad  \cout_n(u,v)=\cout(\alpha_n(u,v),\beta_n(u,v)).  
\end{align}
$\cin$, $\cout$ being smooth functions of their arguments we deduce, using \eqref{eq:350}, \eqref{eq:351},
\begin{align}
  \label{eq:392}
  |\Delta_n\cin|&\leq Cv^2\Big(\|\Delta_n\ls{\alpha}_{+}\|_1+\|\Delta_n\rs{\beta}_+\|_2\Big),\\
  \bigg|\Delta_n\pp{\cin}{u}\bigg|,\bigg|\Delta_n\pp{\cin}{v}\bigg|,\bigg|\Delta_n\pppp{\cin}{u}{v}\bigg|&\leq Cv\Big(\|\Delta_n\ls{\alpha}_{+}\|_1+\|\Delta_n\rs{\beta}_+\|_2\Big),\label{eq:393}\\
  \bigg|\Delta_n\ppp{\cin}{u}\bigg|,\bigg|\Delta_n\ppp{\cin}{v}\bigg|&\leq C\Big(\|\Delta_n\ls{\alpha}_{+}\|_1+\|\Delta_n\rs{\beta}_+\|_2\Big)\label{eq:394}
\end{align}
and the same estimates hold with $\cout$ in the role of $\cin$.

We turn to estimate $\Delta\is{\Gamma}$, $i=1,2$ and its derivative. We recall
\begin{align}
  \label{eq:395}
  \ls{\Gamma}_n=\frac{\ls{\cout}_{+n}}{\ls{\cin}_{+n}}\,\,\frac{\ls{V}_n-\ls{\cin}_{+n}}{\ls{\cout}_{+n}-\ls{V}_n},
\end{align}
where
\begin{align}
  \label{eq:396}
    \ls{\cin}_{+n}(u)=\cin_n(u,u),\qquad  \ls{\cout}_{+n}(u)=\cout_n(u,u)
\end{align}
and $\cin_n(u,v)$, $\cout_n(u,v)$ are given by \eqref{eq:391}. Using \eqref{eq:389}, \eqref{eq:390} together with \eqref{eq:392}, \eqref{eq:393}, \eqref{eq:394} (and the analogous estimates with $\cout$ in the role of $\cin$), we obtain
\begin{align}
  \label{eq:397}
  |\Delta_n\ls{\Gamma}(u)|&\leq Cu^2\Big(\|\Delta_n\ls{\alpha}_{+}\|_1+\|\Delta_n\rs{\beta}_+\|_2+\|\Delta_nx\|_0\Big),\\
  \bigg|\Delta_n\frac{d\ls{\Gamma}}{du}(u)\bigg|&\leq Cu\Big(\|\Delta_n\ls{\alpha}_{+}\|_1+\|\Delta_n\rs{\beta}_+\|_2+\|\Delta_nx\|_0\Big).\label{eq:398}
\end{align}
The same estimates hold for $\Delta_n\rs{\Gamma}(v)$, $\Delta_n\frac{d\rs{\Gamma}}{dv}(v)$ but with $v$ in the role of $u$ on the right hand sides of the estimates (see \eqref{eq:139} for $\rs{\Gamma}_n$).

We turn to estimating $\Delta_nM$ and derivatives thereof. In order to do this we first need to estimate differences of $\mu_n$, $\nu_n$ and derivatives thereof. We recall
\begin{align}
  \label{eq:399}
  \mu_n(u,v)&=\frac{1}{\cout_n(u,v)-\cin_n(u,v)}\,\frac{\cout_n(u,v)}{\cin_n(u,v)}\,\pp{\cin_n}{v}(u,v),\\
  \label{eq:400}
  \nu_n(u,v)&=\frac{1}{\cout_n(u,v)-\cin_n(u,v)}\,\frac{\cin_n(u,v)}{\cout_n(u,v)}\,\pp{\cout_n}{u}(u,v).
\end{align}
Using \eqref{eq:392}, \eqref{eq:393}, \eqref{eq:394} we obtain
\begin{align}
  \label{eq:401}
  |\Delta_n\mu|,|\Delta_n\nu|,\left|\Delta_n\pp{\mu}{u}\right|,\left|\Delta_n\pp{\nu}{v}\right|&\leq Cv\Big(\|\Delta_n\ls{\alpha}_{+}\|_1+\|\Delta_n\rs{\beta}_+\|_2\Big),\\
  \label{eq:402}
  \left|\Delta_n\pp{\mu}{v}\right|,\left|\Delta_n\pp{\nu}{u}\right|&\leq C\Big(\|\Delta_n\ls{\alpha}_{+}\|_1+\|\Delta_n\rs{\beta}_+\|_2\Big).
\end{align}
In view of
\begin{align}
  \label{eq:403}
  M_n(u,v)=\mu_n(u,v)\pp{x_n}{u}(u,v)+\nu_n(u,v)\pp{x_n}{v}(u,v),
\end{align}
the estimates \eqref{eq:401}, \eqref{eq:402} together with \eqref{eq:352}, \eqref{eq:353} imply
\begin{align}
  \label{eq:404}
  |\Delta_nM|&\leq Cv\Big(\|\Delta_n\ls{\alpha}_{+}\|_1+\|\Delta_n\rs{\beta}_+\|_2+\|\Delta_nx\|_0\Big),\\
\label{eq:405}
  \left|\Delta_n\pp{M}{u}\right|,\left|\Delta_n\pp{M}{v}\right|&\leq C\Big(\|\Delta_n\ls{\alpha}_{+}\|_1+\|\Delta_n\rs{\beta}_+\|_2+\|\Delta_nx\|_0\Big).
\end{align}

Now we estimate differences of the second derivatives of $x_{n+1}(u,v)$. From (see \eqref{eq:301})
\begin{align}
  \label{eq:406}
  \pppp{x_{n+1}}{u}{v}(u,v)=M_n(u,v)
\end{align}
we obtain
\begin{align}
  \label{eq:407}
  \left|\Delta_{n+1}\pppp{x}{u}{v}\right|\leq Cv\Big(\|\Delta_n\ls{\alpha}_{+}\|_1+\|\Delta_n\rs{\beta}_+\|_2+\|\Delta_nx\|_0\Big).
\end{align}
To estimate $\Delta_{n+1}\ppp{x}{u}(u,v)$ we use \eqref{eq:302}:
\begin{align}
  \label{eq:408}
    \ppp{x_{n+1}}{u}(u,v)&=\ls{\Lambda}(u)-\frac{1}{\left(\ls{\Gamma}_n(u)\right)^2}\frac{d\ls{\Gamma}_n}{du}(u)\int_{au}^{u}M_n(u',u)du'\nonumber\\
                           &\quad +\frac{1}{\ls{\Gamma}_n(u)}\left(M_n(u,u)-aM_n(au,u)+\int_{au}^u\pp{M_n}{v}(u',u)du'\right)\nonumber\\
                           &\quad-M_n(u,u)+\int_u^v\pp{M_n}{u}(u,v')dv',
\end{align}
where
\begin{align}
  \label{eq:409}
  \ls{\Lambda}_n(u)&=\frac{d\ls{\gamma}_n}{du}(u)\pp{x_n}{u}(au,u)+\ls{\gamma}_n(u)a\ppp{x_n}{u}(au,u)+\ls{\gamma}_n(u)M_n(au,u)
\end{align}
and
\begin{align}
  \label{eq:410}
  \ls{\gamma}_n(u)=\frac{\rs{\Gamma}_n(u)}{\ls{\Gamma}_n(u)}.
\end{align}
Recalling (see \eqref{eq:309})
\begin{align}
  \label{eq:411}
    \ls{\gamma}_n(u)=1+\Landau(u)
\end{align}
and using \eqref{eq:397}, \eqref{eq:398} together with \eqref{eq:352}, \eqref{eq:404} we obtain
\begin{align}
  \label{eq:412}
  |\Delta_n\ls{\Lambda}(u)|\leq a\sup_{T_\varepsilon}\bigg|\Delta_n\ppp{x}{u}\bigg|+Cu\Big(\|\Delta_n\ls{\alpha}_{+}\|_1+\|\Delta_n\rs{\beta}_+\|_2+\|\Delta_nx\|_0\Big).
\end{align}
Using this together with \eqref{eq:397}, \eqref{eq:398} and the analogous estimates for $\Delta_n\rs{\Gamma}$ and its derivative, together with \eqref{eq:404}, \eqref{eq:405} we obtain
\begin{align}
  \label{eq:413}
  \bigg|\Delta_{n+1}\ppp{x}{u}\bigg|\leq a\sup_{T_\varepsilon}\bigg|\Delta_n\ppp{x}{u}\bigg|+Cv\Big(\|\Delta_n\ls{\alpha}_{+}\|_1+\|\Delta_n\rs{\beta}_+\|_2+\|\Delta_nx\|_0\Big).
\end{align}
Analogously, but using \eqref{eq:296}, \eqref{eq:303} and the second of \eqref{eq:297} as a starting point, we obtain
\begin{align}
  \label{eq:414}
  \bigg|\Delta_{n+1}\ppp{x}{v}\bigg|\leq a\sup_{T_\varepsilon}\bigg|\Delta_n\ppp{x}{v}\bigg|+Cv\Big(\|\Delta_n\ls{\alpha}_{+}\|_1+\|\Delta_n\rs{\beta}_+\|_2+\|\Delta_nx\|_0\Big).
\end{align}
In view of \eqref{eq:407}, \eqref{eq:413}, \eqref{eq:414}, choosing $\varepsilon$ sufficiently small, we have
\begin{align}
  \label{eq:415}
  \|\Delta_{n+1}x\|_0\leq k\|\Delta_nx\|_0+C\varepsilon\Big(\|\Delta_n\ls{\alpha}_{+}\|_1+\|\Delta_n\rs{\beta}_+\|_2\Big),
\end{align}
where $k$ is a constant satisfying
\begin{align}
  \label{eq:416}
  0<k<1.
\end{align}

We turn to estimate $\Delta_{n+1}\frac{d^2\ls{\alpha}_+}{du^2}$. We use \eqref{eq:332}:
\begin{align}
  \label{eq:417}
    \frac{d^2\ls{\alpha}_{+,n+1}}{du}(u)&=\ls{F}(\ls{\beta}_{+n}(u),\ls{\alpha}_{-n}(u),\ls{\beta}_{-n}(u))\frac{d^2\ls{\beta}_{+n}}{du^2}(u)\nonumber\\
&\quad +\ls{M}_1(\ls{\beta}_{+n}(u),\ls{\alpha}_{-n}(u),\ls{\beta}_{-n}(u))\frac{d^2\ls{\alpha}_{-n}}{du^2}(u)\nonumber\\
                                      &\quad+\ls{M}_2(\ls{\beta}_{+n}(u),\ls{\alpha}_{-n}(u),\ls{\beta}_{-n}(u))\frac{d^2\ls{\beta}_{-n}}{du^2}(u)\nonumber\\
                                      &\quad +\ls{G}\bigg(\ls{\beta}_{+n}(u),\ls{\alpha}_{-n}(u),\ls{\beta}_{-n}(u),\frac{d\ls{\beta}_{+n}}{du}(u),\frac{d\ls{\alpha}_{-n}}{du}(u),\frac{d\ls{\beta}_{-n}}{du}(u)\bigg)
\end{align}
and recall that $\ls{F}$, $\ls{M}_1$, $\ls{M}_2$, $\ls{G}$ are smooth functions of their arguments. From \eqref{eq:417} we deduce
\begin{align}
  \label{eq:418}
  \bigg|\Delta_{n+1}\frac{d^2\ls{\alpha}_+}{du^2}\bigg|&\leq (\ls{F}_0+Cu)\bigg|\Delta_{n}\frac{d^2\ls{\beta}_+}{du^2}\bigg|\nonumber\\
                                                       &\quad+C\bigg(|\Delta_n\ls{\beta}_+|+|\Delta_n\ls{\alpha}_-|+|\Delta_n\ls{\beta}_-|\nonumber\\
  &\qquad\qquad+\bigg|\Delta_n\frac{d\ls{\beta}_+}{du}\bigg|+\bigg|\Delta_n\frac{d\ls{\alpha}_-}{du}\bigg|+\bigg|\Delta_n\frac{d\ls{\beta}_-}{du}\bigg|\nonumber\\
  &\qquad\qquad\qquad\qquad\qquad+\bigg|\Delta_n\frac{d^2\ls{\alpha}_-}{du^2}\bigg|+\bigg|\Delta_n\frac{d^2\ls{\beta}_-}{du^2}\bigg|\bigg).
\end{align}
Using \eqref{eq:349}, \eqref{eq:383}, \eqref{eq:384}, \eqref{eq:385} we obtain
\begin{align}
  \label{eq:419}
  \bigg|\Delta_{n+1}\frac{d^2\ls{\alpha}_+}{du^2}(u)\bigg|\leq (\ls{F}_0+Cu)\bigg|\Delta_{n}\frac{d^2\ls{\beta}_+}{du^2}(u)\bigg|+C\bigg(u\|\Delta_n\ls{\alpha}_+\|_1+\|\Delta_nx\|_0\bigg).
\end{align}
Analogously we obtain
\begin{align}
  \label{eq:420}
  \bigg|\Delta_{n+1}\frac{d^2\rs{\beta}_+}{dv^2}(v)\bigg|\leq (\rs{F}_0+Cv)\bigg|\Delta_{n}\frac{d^2\rs{\alpha}_+}{dv^2}(v)\bigg|+C\bigg(v\|\Delta_n\rs{\beta}_+\|_1+\|\Delta_nx\|_0\bigg).
\end{align}
Taking the supremum of \eqref{eq:419} in $[0,a\varepsilon]$ and of \eqref{eq:420} in $[0,\varepsilon]$ and using the third of \eqref{eq:348} and the third of \eqref{eq:349}, we obtain
\begin{align}
  \label{eq:421}
  \|\Delta_{n+1}\ls{\alpha}_+\|_1&\leq (\ls{F}_0+C\varepsilon)\|\Delta_n\rs{\beta}_+\|_2+C\bigg(\varepsilon\|\Delta_n\ls{\alpha}_+\|_1+\|\Delta_nx\|_0\bigg),\\
  \|\Delta_{n+1}\rs{\beta}_+\|_2&\leq (a^2\rs{F}_0+C\varepsilon)\|\Delta_n\ls{\alpha}_+\|_1+C\bigg(\varepsilon\|\Delta_n\rs{\beta}_+\|_2+\|\Delta_nx\|_0\bigg).\label{eq:422}
\end{align}
In obvious notation, the estimates \eqref{eq:415}, \eqref{eq:421}, \eqref{eq:422} can be written as
\begin{align}
  \label{eq:423}
  \begin{pmatrix}
\|\Delta_{n+1}\ls{\alpha}_+\|_1\\\|\Delta_{n+1}\rs{\beta}_+\|_2\\    \|\Delta_{n+1}x\|_0
  \end{pmatrix}\leq A  \begin{pmatrix}
\|\Delta_{n}\ls{\alpha}_+\|_1\\\|\Delta_{n}\rs{\beta}_+\|_2\\    \|\Delta_{n}x\|_0
  \end{pmatrix},
\end{align}
where
\begin{align}
  \label{eq:424}
  A=
  \begin{pmatrix}
        C\varepsilon & \ls{F}_0+C\varepsilon & C\\
        a^2\rs{F}_0+C\varepsilon & C\varepsilon & C\\
        C\varepsilon & C\varepsilon & k
  \end{pmatrix}.
\end{align}
The eigenvalues $\lambda$ of $A$ are the roots of the polynomial
\begin{align}
  \label{eq:425}
  (k-\lambda)\Big(\lambda^2-a^2\ls{F}_0\rs{F}_0\Big)+\varepsilon p(\lambda),
\end{align}
where $p(\lambda)$ is a first order polynomial in $\lambda$ with bounded coefficients. Taking into account that $\ls{F}_0\rs{F}_0=a^2$, we see that for $\varepsilon$ sufficiently small, the eigenvalues of $A$, in absolute value, are smaller than one. It follows that the sequence
\begin{align}
  \label{eq:426}
      (\ls{\alpha}_{+n},\rs{\beta}_{+n},x_n);n=0,1,2,\ldots)
\end{align}
converges in $B_A\times B_B\times B_{N_0}$.
\end{proof}

\subsection{Proof of the Existence Theorem}
The two propositions above show that the sequence $(\ls{\alpha}_{+n},\rs{\beta}_{+n},x_n);n=0,1,2,\ldots)$ converges uniformly in $[0,a\varepsilon]\times[0,\varepsilon]\times T_\varepsilon$ to $(\ls{\alpha}_+,\rs{\beta}_+,x)\in B_A\times B_B\times B_{N_0}$. From $\alpha_n(u,v)=\ls{\alpha}_{+n}(u)$, $\beta_n(u,v)=\rs{\beta}_{+n}(v)$ it follows that $\alpha_n$, $\beta_n$ converge uniformly in $T_\varepsilon$ and the limits $\alpha$, $\beta$ satisfy
\begin{align}
  \label{eq:427}
  \pp{\alpha}{v}=0=\pp{\beta}{u}.
\end{align}
The convergence of $\alpha_n$, $\beta_n$ implies the convergence of $\cin_n$, $\cout_n$ to $\cin$, $\cout$, given by\footnote{We recall that in the proofs of the above propositions, the notation
  \begin{align}
    \label{eq:428}
    \cin_n(u,v)=\cin(\alpha_n(u,v),\beta_n(u,v))
  \end{align}
was used. The notation $\cin(u,v)$, $\cout(u,v)$ (without an index and the arguments are $u$, $v$) is introduced as the limits of $\cin_n(u,v)$, $\cout_n(u,v)$.}
\begin{align}
  \label{eq:429}
  \cin(u,v)=\cin(\alpha(u,v),\beta(u,v)),\qquad \cout(u,v)=\cout(\alpha(u,v),\beta(u,v)),
\end{align}
which in turn implies the convergence of $\mu_n$, $\nu_n$ to $\mu$, $\nu$, given by
\begin{align}
  \label{eq:430}
  \mu=\frac{1}{\cout-\cin}\frac{\cout}{\cin}\pp{\cin}{v},\qquad \nu=\frac{1}{\cout-\cin}\frac{\cin}{\cout}\pp{\cout}{u}.
\end{align}
This together with the convergence of $x_n$ to $x$ implies, in view of \eqref{eq:403}, the convergence of $M_n$ to $M$ given by
\begin{align}
  \label{eq:431}
  M=\mu\pp{x}{u}+\nu\pp{x}{v}.
\end{align}
From \eqref{eq:406} we see that the limits $x$ and $M$ satisfy
\begin{align}
  \label{eq:432}
  \pppp{x}{u}{v}(u,v)=M(u,v).
\end{align}

The convergence of $\alpha_n$, $\beta_n$ to $\alpha$, $\beta$ implies the convergence of $g_n$, $h_n$ to $g$, $h$ (see \eqref{eq:197}), given by
\begin{align}
  \label{eq:433}
  g(u,v)=\frac{1}{\cin(u,v)},\qquad h(u,v)=\frac{1}{\cout(u,v)},
\end{align}
which, together with the convergence of $x_n$ to $x$, in view of \eqref{eq:360}, implies the convergence of $\phi_n$, $\psi_n$ to $\phi$, $\psi$ given by
\begin{align}
  \label{eq:434}
  \phi(u,v)&=\frac{1}{\cin(u,v)}\pp{x}{u}(u,v),\\
  \label{eq:435}
  \psi(u,v)&=\frac{1}{\cout(u,v)}\pp{x}{v}(u,v).
\end{align}

From \eqref{eq:194} we have that $t_n$ converges to $t$ given by
\begin{align}
  \label{eq:436}
  t(u,v)=\int_0^u\left(\phi+\psi\right)(u',u')du'+\int_u^v\psi(u,v')dv'.
\end{align}
Hence,
\begin{align}
  \label{eq:437}
  \pp{t}{v}(u,v)=\psi(u,v)=\frac{1}{\cout(u,v)}\pp{x}{v}(u,v)
\end{align}
and 
\begin{align}
  \label{eq:438}
  \pp{t}{u}(u,v)=\phi(u,u)+\int_u^v\pp{\psi}{u}(u,v')dv'.
\end{align}
It follows by a direct computation that \eqref{eq:432} (using also \eqref{eq:430}, \eqref{eq:431}) is equivalent to
\begin{align}
  \label{eq:439}
  \pp{\psi}{u}(u,v)=\pp{\phi}{v}(u,v).
\end{align}
Using this in \eqref{eq:438} we obtain
\begin{align}
  \label{eq:440}
  \pp{t}{u}(u,v)&=\phi(u,u)+\int_u^v\pp{\phi}{v}(u,v')dv'\nonumber\\
                &=\phi(u,v)\nonumber\\
                &=\frac{1}{\cin(u,v)}\pp{x}{u}(u,v).
\end{align}
Equations \eqref{eq:437}, \eqref{eq:440} are the characteristic equations, i.e.~the functions $x(u,v)$, $t(u,v)$ satisfy the characteristic equations in $T_\varepsilon$.

The convergence of $\alpha_n$, $\beta_n$ to $\alpha$, $\beta$ implies the convergence of the boundary functions $\rs{\alpha}_{+n}$, $\ls{\beta}_{+n}$ to $\rs{\alpha}_+$, $\ls{\beta}_+$, given by
\begin{align}
  \label{eq:441}
  \rs{\alpha}_+(v)=\alpha(av,v),\qquad \ls{\beta}_+(u)=\beta(u,u).
\end{align}
The convergence of $t_n$, $x_n$ to $t$, $x$ implies the convergence of the boundary functions $\is{t}_{+n}$, $\is{x}_{+n}$, to $\is{t}_+$, $\is{x}_+$, $i=1,2$ given by
\begin{alignat}{3}
  \label{eq:442}
  \ls{t}_+(u)&=t(u,u),&\qquad \rs{t}_+(v)&=t(av,v),\\
  \ls{x}_+(u)&=x(u,u),&\qquad \rs{x}_+(v)&=x(av,v).\label{eq:443}
\end{alignat}
This in turn implies the convergence of $\is{\alpha}_{-n}$, $\is{\beta}_{-n}$ to $\is{\alpha}_-$, $\is{\beta}_-$, $i=1,2$ (see \eqref{eq:258} and the corresponding expressions for $\ls{\beta}_-$ as well as \eqref{eq:264} and the corresponding expression for $\rs{\beta}_-$), given by
\begin{alignat}{3}
  \label{eq:444}
  \ls{\alpha}_-(u)&=\ls{\alpha}^\ast(\ls{t}_+(u),\ls{x}_+(u)),&\qquad  \rs{\alpha}_-(v)&=\rs{\alpha}^\ast(\rs{t}_+(v),\rs{x}_+(v)),\\
  \ls{\beta}_-(u)&=\ls{\beta}^\ast(\ls{t}_+(u),\ls{x}_+(u)),&\qquad  \rs{\beta}_-(v)&=\rs{\beta}^\ast(\rs{t}_+(v),\rs{x}_+(v)).\label{eq:445}
\end{alignat}
The convergence of the boundary functions then implies the convergence of $\is{V}_n$ to $\is{V}$, $i=1,2$, given by
\begin{align}
  \label{eq:446}
  \ls{V}(u)=\frac{\jump{\ls{\rho}(u)\ls{w}(u)}}{\jump{\ls{\rho}(u)}},\qquad\rs{V}(v)=\frac{\jump{\rs{\rho}(v)\rs{w}(v)}}{\jump{\rs{\rho}(v)}},
\end{align}
where
\begin{alignat}{3}
  \label{eq:447}
  \ls{\rho}_+(u)&=\rho(\ls{\alpha}_+(u),\ls{\beta}_+(u)),&\qquad\ls{\rho}_-(u)&=\rho(\ls{\alpha}_-(u),\ls{\beta}_-(u)),\\
  \rs{\rho}_+(v)&=\rho(\rs{\alpha}_+(v),\rs{\beta}_+(v)),&\qquad\rs{\rho}_-(v)&=\rho(\rs{\alpha}_-(v),\rs{\beta}_-(v))\label{eq:448}
\end{alignat}
and the same expressions with $w$ in the role of $\rho$. This implies the convergence of $\is{\Gamma}_n$ to $\is{\Gamma}$, $i=1,2$, given by
\begin{align}
  \label{eq:449}
    \ls{\Gamma}=\frac{\ls{\cout}_+}{\ls{\cin}_+}\frac{\ls{V}-\ls{\cin}_+}{\ls{\cout}_+-\ls{V}},\qquad\rs{\Gamma}=a\frac{\rs{\cout}_+}{\rs{\cin}_+}\frac{\rs{V}-\rs{\cin}_+}{\rs{\cout}_+-\rs{V}}.
\end{align}
Here the boundary functions $\is{\cout}_+$, $\is{\cin}_+$, $i=1,2$, are given by
\begin{alignat}{3}
  \label{eq:450}
  \ls{\cout}_+(u)&=\cout(u,u),&\qquad \ls{\cin}_+(u)&=\cin(u,u),\\
  \rs{\cout}_+(v)&=\cout(av,v),&\qquad \rs{\cin}_+(v)&=\cin(av,v).\label{eq:451}
\end{alignat}
The convergence of $x_n$, $\is{\Gamma}_n$ and $M_n$ imply the convergence of $\is{\Lambda}_n$ to $\is{\Lambda}$ (see \eqref{eq:295}, \eqref{eq:296}), given by
\begin{align}
  \ls{\Lambda}(u)&=\frac{d\ls{\gamma}}{du}(u)\pp{x}{u}(au,u)+\ls{\gamma}(u)a\ppp{x}{u}(au,u)+\ls{\gamma}(u)M(au,u)\nonumber\\
  &=\frac{d}{du}\left(\ls{\gamma}(u)\pp{x}{u}(au,u)\right),\label{eq:452}\\
  \label{eq:453}
  \rs{\Lambda}(v)&=\frac{d\rs{\gamma}}{dv}(v)\pp{x}{v}(av,av)+\rs{\gamma}(v)a\ppp{x}{v}(av,av)+\rs{\gamma}(v)aM(av,av)\nonumber\\
                 &=\frac{d}{dv}\left(\rs{\gamma}(v)\pp{x}{v}(av,av)\right),
\end{align}
where
\begin{align}
  \label{eq:454}
    \ls{\gamma}(u)=\frac{\rs{\Gamma}(u)}{\ls{\Gamma}(u)},\qquad\rs{\gamma}(v)=\frac{\rs{\Gamma}(v)}{\ls{\Gamma}(av)}.
\end{align}

In order to investigate the limit of the boundary conditions for $x$ along the shocks we look at the limiting form of the equations \eqref{eq:298}, \eqref{eq:299}. Recalling $\pp{x}{u}(0,0)=\frac{1}{\Gamma_0}$ and using $\ls{\gamma}(0)=1$ together with \eqref{eq:432}, \eqref{eq:452}, the limiting equation of \eqref{eq:298} with $u=v$ is
\begin{align}
  \label{eq:455}
  \pp{x}{u}(u,u)=\ls{\gamma}(u)\pp{x}{u}(au,u)+\frac{1}{\ls{\Gamma}(u)}\left(\pp{x}{v}(u,u)-\pp{x}{v}(au,u)\right).
\end{align}
Using $\rs{\gamma}(0)=1$, $\pp{x}{v}(0,0)=1$, the limiting equation of \eqref{eq:299} with $u=av$ is
\begin{align}
  \label{eq:456}
  \pp{x}{v}(av,v)=\rs{\gamma}(v)\pp{x}{v}(av,av)+\rs{\Gamma}(v)\left(\pp{x}{u}(av,v)-\pp{x}{u}(av,av)\right).
\end{align}
We recall that in \eqref{eq:455} $u\in[0,a\varepsilon]$, while in \eqref{eq:456} $v\in[0,\varepsilon]$. Defining
\begin{align}
  \label{eq:457}
  \ls{b}(u)&\defeq \pp{x}{v}(u,u)-\pp{x}{u}(u,u)\ls{\Gamma}(u),\\
  \rs{b}(v)&\defeq \pp{x}{v}(av,v)-\pp{x}{u}(av,v)\rs{\Gamma}(v)\label{eq:458}
\end{align}
and recalling the expressions for $\is{\gamma}$, $i=1,2$, as given in \eqref{eq:454}, we rewrite \eqref{eq:455}, \eqref{eq:456} as
\begin{align}
  \label{eq:459}
  \ls{b}(u)&=\rs{b}(u),\\
  \rs{b}(v)&=\rs{\gamma}(v)\ls{b}(av),\label{eq:460}
\end{align}
respectively. Substituting \eqref{eq:460} into the right hand side of \eqref{eq:459} (recall that \eqref{eq:460} holds for $v\in[0,\varepsilon]$), we obtain
\begin{align}
  \label{eq:461}
  \ls{b}(u)=\rs{\gamma}(u)\ls{b}(au).
\end{align}
In order to shorten the notation we investigate this equation without the stacked indices, hence writing it as
\begin{align}
  \label{eq:462}
  b(u)=\gamma(u)b(au).
\end{align}
From \eqref{eq:457}, using $\pp{x}{v}(0,0)=1$, $\pp{x}{u}(0,0)=\frac{1}{\Gamma_0}=\frac{1}{\ls{\Gamma}(0)}$ we have $b(0)=0$. Taking the derivative of \eqref{eq:462} yields
\begin{align}
  \label{eq:463}
  b'(u)=\gamma'(u)b(au)+\gamma(u)ab'(au).
\end{align}
Integrating from $u=0$ we obtain
\begin{align}
  \label{eq:464}
  b(u)=\int_0^u\Big(\gamma'(\tilde{u})b(a\tilde{u})+\gamma(\tilde{u})ab'(a\tilde{u})\Big)d\tilde{u}.
\end{align}
Using the notation $\|f\|=\sup_{u\in[0,a\varepsilon]}|f(u)|$, \eqref{eq:464} implies
\begin{align}
  \label{eq:465}
  \|b\|\leq C\varepsilon\|b\|+\varepsilon(1+C\varepsilon)a^2\|b'\|,
\end{align}
where we made use of $\gamma(u)=1+\Landau(u)$ (see \eqref{eq:310}). Choosing $\varepsilon$ sufficiently small, this implies
\begin{align}
  \label{eq:466}
  \|b\|\leq C\varepsilon\|b'\|.
\end{align}
From \eqref{eq:463} we deduce
\begin{align}
  \label{eq:467}
  \|b'\|&\leq C\|b\|+(1+C\varepsilon)a\|b'\|\nonumber\\
               &\leq (a+C\varepsilon)\|b'\|,
\end{align}
where for the second inequality we used \eqref{eq:466}. Choosing $\varepsilon$ sufficiently small, such that $a+C\varepsilon<1$, we obtain $b'(u)=0$ for $u\in[0,a\varepsilon]$. Together with $b(0)=0$ this implies $b(u)=0$ for $u\in[0,a\varepsilon]$. We have shown that our limit satisfies (putting back the stacked index)
\begin{align}
  \label{eq:468}
  \ls{b}(u)\equiv 0.
\end{align}
This implies, through \eqref{eq:460}, that
\begin{align}
  \label{eq:469}
  \rs{b}(v)\equiv 0.
\end{align}
Hence, the system of equations
\begin{align}
  \label{eq:470}
  \pp{x}{v}(u,u)&=\pp{x}{u}(u,u)\ls{\Gamma}(u),\\
  \pp{x}{v}(av,v)&=\pp{x}{u}(av,v)\rs{\Gamma}(v)\label{eq:471}
\end{align}
hold in $T_\varepsilon$. Using the characteristic equations, these are equivalent to (see \eqref{eq:31}, \eqref{eq:33}, \eqref{eq:34})
\begin{align}
  \label{eq:472}
    \ls{V}d\ls{t}=d\ls{x},\qquad \rs{V}d\rs{t}=d\rs{x},
\end{align}
i.e.~the boundary conditions on the shocks are satisfied in the limit.

Equations \eqref{eq:319}, \eqref{eq:320} are satisfied in the limit, i.e.
\begin{align}
  \label{eq:473}
    \ls{\alpha}_{+}(u)&=\ls{H}(\ls{\beta}_{+}(u),\ls{\alpha}_{-}(u),\ls{\beta}_{-}(u)),\\
  \rs{\beta}_{+}(v)&=\rs{H}(\rs{\alpha}_{+}(v),\rs{\alpha}_{-}(v),\rs{\beta}_{-}(v)).\label{eq:474}
\end{align}
This implies (see \eqref{eq:65}, \eqref{eq:66})
\begin{align}
  \label{eq:475}
  J\Big(\ls{\alpha}_+(u),\ls{\beta}_+(u),\ls{\alpha}_-(u),\ls{\beta}_-(u)\Big)&=0,\\
  J\Big(\rs{\alpha}_+(v),\rs{\beta}_+(v),\rs{\alpha}_-(v),\rs{\beta}_-(v)\Big)&=0.\label{eq:476}
\end{align}
Together with \eqref{eq:446} we see that the jump conditions along the shocks are satisfied in the limit.

We recall that the determinism conditions are satisfied at the interaction point by assumption on the data (see \eqref{eq:21}, \eqref{eq:23}):
\begin{align}
  \label{eq:477}
      -\eta_0<&\ls{V}_0<(\ls{\cin}_0^\ast)_0,\\
  (\rs{\cout}_0^\ast)_0<&\rs{V}_0<\eta_0,\label{eq:478}
\end{align}
where
\begin{align}
  \label{eq:479}
  (\ls{\cin}_0^\ast)_0&=\ls{\cin}_-(0)=\cin(\ls{\alpha}_-(0),\ls{\beta}_-(0)),\\
  (\rs{\cout}_0^\ast)_0&=\rs{\cout}_-(0)=\cout(\rs{\alpha}_-(0),\rs{\beta}_-(0)).\label{eq:480}
\end{align}
Therefore, choosing $\varepsilon$ sufficiently small, the determinism conditions are satisfied for $v\in[0,\varepsilon]$, $u\in[0,a\varepsilon]$, i.e.
\begin{align}
  \label{eq:481}
  \ls{\cin}_+(u)&<\ls{V}(u)<\ls{\cin}_-(u),\\
  \rs{\cout}_-(v)&<\rs{V}(v)<\rs{\cout}_+(u).\label{eq:482}
\end{align}
In view of \eqref{eq:427}, \eqref{eq:437}, \eqref{eq:440}, \eqref{eq:446}, \eqref{eq:475}, \eqref{eq:476}, \eqref{eq:481}, \eqref{eq:482} we have proven the existence of a twice differentiable solution $(\alpha,\beta,t,x)$ to the shock interaction problem as stated in subsection \ref{shock_interaction_problem}.

The determinant of the Jacobian $\frac{\partial(t,x)}{\partial(u,v)}$ is given by
\begin{align}
  \label{eq:483}
  \begin{vmatrix}
    \pp{t}{u} & \pp{t}{v}\\
    \pp{x}{u} & \pp{x}{v}
  \end{vmatrix}=(\cout-\cin)\pp{t}{u}\pp{t}{v}=2\eta\pp{t}{u}\pp{t}{v}=-\frac{2}{\eta_0\Gamma_0}+\Landau(v),
\end{align}
where for the last equality we used \eqref{eq:209}, \eqref{eq:213}. Therefore, by choosing $\varepsilon$ sufficiently small, the functions $\alpha$, $\beta$ are twice differentiable functions of $(t,x)$ on the image of $T_\varepsilon$ by the map $(u,v)\mapsto(t(u,v),x(u,v))$.

This concludes the proof of theorem \ref{existence_theorem} from page \pageref{existence_theorem}.

\subsection{Asymptotic Form}
By straightforward expansions of the functions $(\alpha,\beta,t,x)$ at the interaction point in the state behind, one can show that any twice differentiable solution of the interaction problem is of the same asymptotic form as the solution constructed in the existence proof:
\begin{align*}
  \alpha(u,v)&=\beta_0+\alpha_0'u+\Landau(v^2),\\
  \beta(u,v)&=\beta_0+\beta_0'v+\Landau(v^2),\\
  t(u,v)&=\frac{1}{\eta_0}\left(v-\frac{u}{\Gamma_0}\right)+\Landau(v^2),\\
  x(u,v)&=\frac{1}{\Gamma_0}u+v+\Landau(v^2).
\end{align*}

\subsection{Uniqueness}
We have the following uniqueness result:
\begin{theorem}[Uniqueness]
  Let $(\alpha_1,\beta_1,t_1,x_1)$, $(\alpha_2,\beta_2,t_2,x_2)$, both in $C^2(T_\varepsilon)$, be two solutions of the interaction problem as stated in \ref{shock_interaction_problem} corresponding to the same future developments of the data. Then, for $\varepsilon$ sufficiently small, the two solutions coincide.
\end{theorem}
Using similar estimates as in the convergence proof above, the proof is straightforward.

\section{Higher Regularity}
\begin{theorem}[Higher Regularity]
  For $\varepsilon$ sufficiently small, the established solution $\alpha$, $\beta$, $t$, $x$ of the interaction problem is infinitely differentiable.
\end{theorem}

\begin{proof}
We show that all derivatives of the functions $\alpha$, $\beta$, $t$, $x$ are bounded. This can be done by induction, i.e.~once it is assumed that the $m$'th order derivatives are bounded it can be shown that the $m+1$'th order derivatives are bounded. The base case of this induction is given by the solution being in $C^2(T_{\varepsilon})$, i.e.~it is already shown that the solution is two times continuously differentiable. Instead of showing the inductive step $m\mapsto m+1$ we only show that the third order derivatives are bounded, the general inductive step being completely analogous. We restrict to this case in order to simplify and shorten the presentation of the argument. However, it is important to note that the encountered smallness conditions on $\varepsilon$ in the process are independent of the order of derivatives studied, i.e.~also for the general inductive step, no other smallness conditions would be necessary. This will become apparent during the argument to follow.

Through the characteristic equations \eqref{eq:27}, the derivatives of $t$ can be expressed in terms of the derivatives of $x$ and functions of $\alpha$, $\beta$ and derivatives thereof. From this together with
\begin{align}
  \label{eq:484}
  \alpha(u,v)=\ls{\alpha}_+(u),\qquad \beta(u,v)=\rs{\beta}_+(v),
\end{align}
we see that it suffices to establish bounds on the derivatives of the functions $x(u,v)$, $\ls{\alpha}_+(u)$, $\rs{\beta}_+(v)$.

We first consider the function $\ls{\alpha}_+(u)$. We recall \eqref{eq:67}:
\begin{align}
  \label{eq:485}
  \ls{\alpha}_+(u)=\ls{H}(\ls{\beta}_+(u),\ls{\alpha}_-(u),\ls{\beta}_-(u)).
\end{align}
Here $\ls{H}$ is a smooth function of its arguments and
\begin{align}
  \label{eq:486}
  \ls{\alpha}_-(u)=\ls{\alpha}^\ast(\ls{t}_+(u),\ls{x}_+(u)),\qquad\ls{\beta}_-(u)=\ls{\beta}^\ast(\ls{t}_+(u),\ls{x}_+(u)),
\end{align}
where
\begin{align}
  \label{eq:487}
  \ls{t}_+(u)=t(u,u),\qquad \ls{x}_+(u)=x(u,u).
\end{align}
Taking the third derivative of \eqref{eq:485} we obtain
\begin{align}
  \label{eq:488}
  \frac{d^3\ls{\alpha}_+}{du^3}(u)&=\ls{F}\Big(\ls{\beta}_+(u),\ls{\alpha}_-(u),\ls{\beta}_-(u)\Big)\frac{d^3\ls{\beta}_+}{du^3}(u)\nonumber\\
                                  &\quad +\ls{M}_1\Big(\ls{\beta}_+(u),\ls{\alpha}_-(u),\ls{\beta}_-(u)\Big)\frac{d^3\ls{\alpha}_-}{du^3}(u)\nonumber\\
  &\quad +\ls{M}_2\Big(\ls{\beta}_+(u),\ls{\alpha}_-(u),\ls{\beta}_-(u)\Big)\frac{d^3\ls{\beta}_-}{du^3}(u)+\lot
\end{align}
Here and in the following we denote by $\lot$ terms that are of lower order. Since we are studying the third order derivatives this means we denote by $\lot$ terms which can involve the functions $\alpha$, $\beta$, $t$, $x$, derivatives thereof and second derivatives thereof. Terms of lower order are bounded in absolute value. The functions $\ls{F}$, $\ls{M}_1$, $\ls{M}_2$ correspond to the partial derivatives of $\ls{H}$ and are therefore smooth functions of their arguments. They already show up in \eqref{eq:69}.

Taking the third derivative of $\ls{\alpha}_-(u)=\ls{\alpha}^\ast(t_+(u),x_+(u))$, we obtain
\begin{align}
  \label{eq:489}
  \frac{d^3\ls{\alpha}_-}{du^3}(u)=\pp{\ls{\alpha}^\ast}{t}(\ls{t}_+(u),\ls{x}_+(u))\frac{d^3\ls{t}_+}{du^3}(u)+\pp{\ls{\alpha}^\ast}{x}(\ls{t}_+(u),\ls{x}_+(u))\frac{d^3\ls{x}_+}{du^3}(u)+\lot.
\end{align}
The third derivative of $\ls{x}_+(u)=x(u,u)$ is given by
\begin{align}
  \label{eq:490}
  \frac{d^3\ls{x}_+}{du^3}(u)=\frac{\partial^3x}{\partial u^3}(u,u)+3\frac{\partial^3x}{\partial u^2\partial v}(u,u)+3\frac{\partial^3 x}{\partial u\partial v^2}(u,u)+\frac{\partial^3 x}{\partial v^3}(u,u).
\end{align}
The same equation holds with $t$ in the role of $x$. In view of \eqref{eq:102}, the partial derivatives of mixed type can be expressed by lower order derivatives, hence
\begin{align}
  \label{eq:491}
  \frac{d^3\ls{x}_+}{du^3}(u)=\frac{\partial^3x}{\partial u^3}(u,u)+\frac{\partial^3 x}{\partial v^3}(u,u)+\lot.
\end{align}
In view of the characteristic equations \eqref{eq:27} we have
\begin{align}
  \label{eq:492}
  \frac{\partial^3t}{\partial u^3}(u,v)&=\frac{1}{\cin(u,v)}\frac{\partial^3x}{\partial u^3}(u,v)+\lot,\\
  \label{eq:493}
  \frac{\partial^3t}{\partial v^3}(u,v)&=\frac{1}{\cout(u,v)}\frac{\partial^3x}{\partial u^3}(u,v)+\lot,
\end{align}
where we use the short notation
\begin{align}
  \label{eq:494}
 \cin(u,v)=\cin(\alpha(u,v),\beta(u,v)),\qquad \cout(u,v)=\cout(\alpha(u,v),\beta(u,v)).
\end{align}
The partial derivatives of $t(u,v)$ of mixed type can be expressed by partial derivatives of $x(u,v)$ of mixed type. Therefore,
\begin{align}
  \label{eq:495}
    \frac{d^3\ls{t}_+}{du^3}(u)=\frac{1}{\cin(u,u)}\frac{\partial^3x}{\partial u^3}(u,u)+\frac{1}{\cout(u,u)}\frac{\partial^3 x}{\partial v^3}(u,u)+\lot.
\end{align}
Substituting \eqref{eq:491}, \eqref{eq:495} into \eqref{eq:489} we obtain
\begin{align}
  \label{eq:496}
  \bigg|\frac{d^3\ls{\alpha}_-}{du^3}(u)\bigg|,\bigg|\frac{d^3\ls{\beta}_-}{du^3}(u)\bigg|\leq C+\overline{C}_1\sup_{[0,a\varepsilon]}\bigg|\frac{\partial^3x}{\partial u^3}(u,u)\bigg|+\overline{C}_2\sup_{[0,a\varepsilon]}\bigg|\frac{\partial^3x}{\partial v^3}(u,u)\bigg|.
\end{align}
where the statement for the third derivative of $\ls{\beta}_-(u)$ is obtained analogously. Here and in the following we denote by $C$ and $\overline{C}$ (possibly with an index) generic numerical constants whose values may change from line to line, in agreement with standard notation. However, by $\overline{C}$ we denote numerical constants whose values are the same no matter what order of derivative is studied. For example if we would estimate the fourth order derivatives (assuming third order derivatives are bounded) equation \eqref{eq:496} would look the same, just with fourth order derivatives instead of third order. But the constant $C$ would possess a different numerical value while the constants $\overline{C}_1$, $\overline{C}_2$ would possess the same numerical values as in \eqref{eq:496}.

For the function $\ls{F}$ in \eqref{eq:488} we have the estimate
\begin{align}
  \label{eq:497}
\bigg|\ls{F}\Big(\ls{\beta}_+(u),\ls{\alpha}_-(u),\ls{\beta}_-(u)\Big)\bigg|\leq \ls{F}_0+\overline{C}u.
\end{align}
Similar estimates hold for the functions $\ls{M}_1$, $\ls{M}_2$ in \eqref{eq:488}. Using this together with \eqref{eq:496} and $\ls{\beta}_+(u)=\rs{\beta}_+(u)$ (see \eqref{eq:184}) in \eqref{eq:488} we obtain
\begin{align}
  \label{eq:498}
  \bigg|\frac{d^3\ls{\alpha}_+}{du^3}(u)\bigg|&\leq C+(\ls{F}_0+\overline{C}_1u)\bigg|\frac{d^3\rs{\beta}_+}{dv^3}(u)\bigg|\nonumber\\
  &\quad+\overline{C}_2\sup_{[0,a\varepsilon]}\bigg|\frac{\partial^3x}{\partial u^3}(u,u)\bigg|+\overline{C}_3\sup_{[0,a\varepsilon]}\bigg|\frac{\partial^3x}{\partial v^3}(u,u)\bigg|.
\end{align}
Analogously, but using as a starting point equation \eqref{eq:68}:
\begin{align}
  \label{eq:499}
  \rs{\beta}_+(v)&=\rs{H}(\rs{\alpha}_+(v),\rs{\alpha}_-(v),\rs{\beta}_-(v))
\end{align}
and recalling $\rs{\alpha}_+(v)=\ls{\alpha}_+(av)$, we obtain
\begin{align}
  \label{eq:500}
  \bigg|\frac{d^3\rs{\beta}_+}{dv^3}(v)\bigg|&\leq C+(\rs{F}_0+\overline{C}_1v)a^3\bigg|\frac{d^3\ls{\alpha}_+}{du^3}(av)\bigg|\nonumber\\
  &\quad+\overline{C}_2\sup_{[0,\varepsilon]}\bigg|\frac{\partial^3x}{\partial u^3}(av,v)\bigg|+\overline{C}_3\sup_{[0,\varepsilon]}\bigg|\frac{\partial^3x}{\partial v^3}(av,v)\bigg|.
\end{align}
Putting \eqref{eq:500} into \eqref{eq:498} yields
\begin{align}
  \label{eq:501}
  \bigg|\frac{d^3\ls{\alpha}_+}{du^3}(u)\bigg|&\leq C+(\ls{F}_0+\overline{C}_1u)(\rs{F}_0+\overline{C}_2u)a^3\bigg|\frac{d^3\ls{\alpha}_+}{du^3}(au)\bigg|\nonumber\\
  &\quad+\overline{C}_3\sup_{T_\varepsilon}\bigg|\frac{\partial^3x}{\partial u^3}\bigg|+\overline{C}_4\sup_{T_\varepsilon}\bigg|\frac{\partial^3x}{\partial v^3}\bigg|.
\end{align}
Taking the supremum of this equation in $[0,a\varepsilon]$, recalling that $\ls{F}_0\rs{F}_0=a^2<1$ and choosing $\varepsilon$ sufficiently small, we deduce
\begin{align}
  \label{eq:502}
  \sup_{[0,a\varepsilon]}\bigg|\frac{d^3\ls{\alpha}_+}{du^3}\bigg|\leq C+\overline{C}_1\sup_{T_\varepsilon}\bigg|\frac{\partial^3x}{\partial u^3}\bigg|+\overline{C}_2\sup_{T_\varepsilon}\bigg|\frac{\partial^3x}{\partial v^3}\bigg|.
\end{align}
Using this in the right hand side of \eqref{eq:500} and taking the supremum of the resulting equation in $[0,\varepsilon]$, we obtain
\begin{align}
  \label{eq:503}
  \sup_{[0,\varepsilon]}\bigg|\frac{d^3\rs{\beta}_+}{dv^3}\bigg|&\leq C+\overline{C}_1\sup_{T_\varepsilon}\bigg|\frac{\partial^3x}{\partial u^3}\bigg|+\overline{C}_2\sup_{T_\varepsilon}\bigg|\frac{\partial^3x}{\partial v^3}\bigg|.
\end{align}

We now consider the function $x(u,v)$. We recall \eqref{eq:96}:
\begin{align}
  \label{eq:504}
  \pp{x}{u}(u,v)&=\ls{\gamma}(u)\,\pp{x}{u}(au,u)\nonumber\\
  &\quad+\frac{1}{\ls{\Gamma}(u)}\int_{au}^uM(u',u)du'+\int_u^vM(u,v')dv',
\end{align}
where (see \eqref{eq:33}, \eqref{eq:148})
\begin{align}
  \ls{\gamma}(u)&=\frac{\rs{\Gamma}(u)}{\ls{\Gamma}(u)},\label{eq:505}\\
  \ls{\Gamma}&=\frac{\ls{\cout}_{+}}{\ls{\cin}_{+}}\frac{\ls{V}-\ls{\cin}_{+}}{\ls{\cout}_{+}-\ls{V}},\qquad \rs{\Gamma}=a\frac{\rs{\cout}_{+}}{\rs{\cin}_{+}}\frac{\rs{V}-\rs{\cin}_{+}}{\rs{\cout}_{+}-\rs{V}}  \label{eq:506}
\end{align}
and $\is{V}$, $i=1,2$ satisfies
\begin{align}
  \label{eq:507}
  \is{V}=\frac{\jump{\is{\rho}\is{w}}}{\jump{\is{\rho}}},\qquad i=1,2,  
\end{align}
where $\ls{\rho}_+(u)=\rho(\ls{\alpha}_+(u),\ls{\beta}_+(u))$, $\ls{\rho}_-(u)=\rho(\ls{\alpha}_-(u),\ls{\beta}_-(u))$ and analogous for $\ls{w}_\pm(u)$, $\rs{\rho}_\pm(v)$, $\rs{w}_\pm(v)$. We have
\begin{align}
  \label{eq:508}
  \ls{\Gamma}(u),\rs{\Gamma}(v),\frac{d\ls{\Gamma}}{du}(u),\frac{d\rs{\Gamma}}{dv}(v),\frac{d^2\ls{\Gamma}}{du^2}(u),\frac{d^2\rs{\Gamma}}{dv^2}(v)=\lot.
\end{align}

The function $M(u,v)$ depends on first derivatives of $\alpha(u,v)$ with respect to $u$ and first derivatives of $\beta(u,v)$ with respect to $v$. Therefore,
\begin{align}
  \label{eq:509}
  M(u,v),\pp{M}{u}(u,v),\pp{M}{v}(u,v)=\lot.
\end{align}

Using \eqref{eq:508}, \eqref{eq:509} and taking two partial derivatives with respect to $u$ of \eqref{eq:504} we obtain
\begin{align}
  \label{eq:510}
  \frac{\partial^3x}{\partial u^3}(u,v)&=\ls{\gamma}(u)a^2\frac{\partial^3x}{\partial u^3}(au,u)\nonumber\\
  &\quad+\frac{1}{\ls{\Gamma}(u)}\int_{au}^u\ppp{M}{v}(u',u)du'+\int_u^v\ppp{M}{u}(u,v')dv'+\lot.
\end{align}

By taking second order derivatives of \eqref{eq:431}, (see also \eqref{eq:429}, \eqref{eq:430}), we obtain
\begin{align}
  \label{eq:511}
  \bigg|\ppp{M}{v}(u,v)\bigg|&\leq C+\overline{C}_1\bigg|\frac{d^3\rs{\beta}_+}{dv^3}(v)\bigg|+\overline{C}_2\bigg|\frac{\partial^3x}{\partial v^3}(u,v)\bigg|,\\
    \bigg|\ppp{M}{u}(u,v)\bigg|&\leq C+\overline{C}_1\bigg|\frac{d^3\ls{\alpha}_+}{du^3}(u)\bigg|+\overline{C}_2\bigg|\frac{\partial^3x}{\partial u^3}(u,v)\bigg|.\label{eq:512}
\end{align}

Putting these into \eqref{eq:510} and taking $\varepsilon$ sufficiently small, such that $|\ls{\gamma}(u)a|\leq 1$ (recall that $\ls{\gamma}(0)=1$), we obtain
\begin{align}
  \label{eq:513}
  \bigg|\frac{\partial^3x}{\partial u^3}(u,v)\bigg|&\leq C+a\bigg|\frac{\partial^3x}{\partial u^3}(au,u)\bigg|+\overline{C}_1u\sup_{[0,\varepsilon]}\bigg|\frac{d^3\rs{\beta}_+}{dv^3}\bigg|+\overline{C}_2v\sup_{[0,a\varepsilon]}\bigg|\frac{d^3\ls{\alpha}_+}{du^3}\bigg|\nonumber\\
  &\quad +\overline{C}_3\int_{au}^u\bigg|\frac{\partial^3x}{\partial v^3}(u',u)\bigg|du'+\overline{C}_4\int_u^v\bigg|\frac{\partial^3x}{\partial u^3}(u,v')\bigg|dv'.
\end{align}
Analogous to the way we arrived at this equation, we find, based on
\begin{align}
  \label{eq:514}
    \pp{x}{v}(u,v)&=\rs{\gamma}(v)\,\pp{x}{v}(av,av)\nonumber\\
  &\quad+\rs{\Gamma}(v)\int_{av}^vM(av,v')dv'+\int_{av}^uM(u',v)du',
\end{align}
the estimate
\begin{align}
  \label{eq:515}
  \bigg|\frac{\partial^3x}{\partial v^3}(u,v)\bigg|&\leq C+a\bigg|\frac{\partial^3x}{\partial v^3}(av,av)\bigg|+\overline{C}_1v\sup_{[0,a\varepsilon]}\bigg|\frac{d^3\ls{\alpha}_+}{du^3}\bigg|+\overline{C}_2u\sup_{[0,\varepsilon]}\bigg|\frac{d^3\rs{\beta}_+}{dv^3}\bigg|\nonumber\\
  &\quad +\overline{C}_3\int_{av}^v\bigg|\frac{\partial^3x}{\partial u^3}(av,v')\bigg|dv'+\overline{C}_4\int_{av}^u\bigg|\frac{\partial^3x}{\partial v^3}(u',v)\bigg|du'.
\end{align}
Using \eqref{eq:502}, \eqref{eq:503} we rewrite \eqref{eq:513}, \eqref{eq:515} as
\begin{align}
  \label{eq:516}
  \bigg|\frac{\partial^3x}{\partial u^3}(u,v)\bigg|&\leq C+(a+\overline{C}_1\varepsilon)\sup_{T_\varepsilon}\bigg|\frac{\partial^3x}{\partial u^3}\bigg|+\overline{C}_2\varepsilon\sup_{T_\varepsilon}\bigg|\frac{\partial^3x}{\partial v^3}\bigg|,\\
  \bigg|\frac{\partial^3x}{\partial v^3}(u,v)\bigg|&\leq C+(a+\overline{C}_1\varepsilon)\sup_{T_\varepsilon}\bigg|\frac{\partial^3x}{\partial v^3}\bigg|+\overline{C}_2\varepsilon\sup_{T_\varepsilon}\bigg|\frac{\partial^3x}{\partial u^3}\bigg|.\label{eq:517}
\end{align}
Taking the supremum of \eqref{eq:516} in $T_\varepsilon$ and choosing $\varepsilon$ sufficiently small, such that (recall that $a<1$)
\begin{align}
  \label{eq:518}
  1-(a+\overline{C}_1\varepsilon)>0,
\end{align}
where $\overline{C}_1$ is the constant showing up in \eqref{eq:516}, we deduce
\begin{align}
  \label{eq:519}
  \sup_{T_\varepsilon}\bigg|\frac{\partial^3x}{\partial u^3}\bigg|&\leq C+\overline{C}\varepsilon\sup_{T_\varepsilon}\bigg|\frac{\partial^3x}{\partial v^3}\bigg|.
\end{align}
Similarly, we obtain from \eqref{eq:517}
\begin{align}
  \label{eq:520}
  \sup_{T_\varepsilon}\bigg|\frac{\partial^3x}{\partial v^3}\bigg|&\leq C+\overline{C}\varepsilon\sup_{T_\varepsilon}\bigg|\frac{\partial^3x}{\partial u^3}\bigg|.
\end{align}
Using \eqref{eq:520} in \eqref{eq:519} yields
\begin{align}
  \label{eq:521}
  \sup_{T_\varepsilon}\bigg|\frac{\partial^3x}{\partial u^3}\bigg|&\leq C+\overline{C}\varepsilon\sup_{T_\varepsilon}\bigg|\frac{\partial^3x}{\partial u^3}\bigg|.
\end{align}
Choosing $\varepsilon$ sufficiently small, we obtain
\begin{align}
  \label{eq:522}
  \sup_{T_\varepsilon}\bigg|\frac{\partial^3x}{\partial u^3}\bigg|&\leq C,
\end{align}
which, through \eqref{eq:520}, implies
\begin{align}
  \label{eq:523}
  \sup_{T_\varepsilon}\bigg|\frac{\partial^3x}{\partial v^3}\bigg|&\leq C.
\end{align}
Using the bounds \eqref{eq:522}, \eqref{eq:523} in \eqref{eq:502}, \eqref{eq:503} we obtain
\begin{align}
  \label{eq:524}
\sup_{[0,a\varepsilon]}\bigg|\frac{d^3\ls{\alpha}_+}{du^3}\bigg|,\sup_{[0,\varepsilon]}\bigg|\frac{d^3\rs{\beta}_+}{dv^3}\bigg|\leq C,
\end{align}
i.e.~we have shown that the third order derivative of $\ls{\alpha}_+(u)$, $\rs{\beta}_+(v)$, $x(u,v)$ are bounded. As mentioned above, by $\alpha(u,v)=\ls{\alpha}_+(u)$, $\beta(u,v)=\rs{\beta}_+(v)$ and the fact that the derivatives of $t(u,v)$ can, through the characteristic equations, be expressed in terms of the derivatives of $x(u,v)$, all derivatives of third order of $\alpha$, $\beta$, $t$, $x$ are bounded.

We note again that all the smallness conditions on $\varepsilon$ made in the argument above depend only on constants of the type $\overline{C}$, i.e.~do not depend on the order of derivative studied.
\end{proof}

\section*{Acknowledgments}
The author thanks Anne Franzen for many stimulating discussions on the subject.

\bibliography{lisibach}
\bibliographystyle{amsplain}

\end{document}